\documentclass[10pt,reqno]{amsart}

\usepackage{epsfig,amsmath,amsthm,amssymb,cancel,mathtools}
\usepackage{color}
\usepackage{mathrsfs}
\usepackage{geometry}
\usepackage{latexsym,amsfonts}
\usepackage{bm}
\usepackage{graphicx}
\usepackage{epsfig}
\usepackage{tikz}
\usepackage{subfigure}
\usepackage{comment}
\usepgflibrary{decorations.markings}
\usetikzlibrary{decorations.markings, calc}

\def \Sp{\mathfrak {Sp}}

\newcommand{\V}{{\mathcal V}}
\newcommand{\Q}{{\mathcal Q}}

\newcommand{\U}{T}
\def\ra{\rightarrow}
\def\a{\alpha}

\newcommand{\ka}{\varkappa}
\def\l{\lambda}

\def\n{\nu}

\def\s{\sigma}

\def\e{\varepsilon}
\def\z{\zeta}

\def\De{\Delta}
\def\R{{\mathbb R}}
\def\C{{\mathbb C}}
\def\N{{\mathbb N}}

\def\Z{{\mathbb Z}}
\def\er{\end{remark}}
\def\bt{\begin{theorem}}
\def\et{\end{theorem}}
\def\bc{\begin{coroll}}
\def\ec{\end{coroll}}
\def\brs{\begin{remarks} \rm\
\begin{enumerate}}
\def\ers{\end{enumerate}\end{remarks}}
\def\bl{\begin{lemma}}
\def\el{\end{lemma}}
\def\bxs{\begin{examps}. \rm\begin{enumerate}}
\def\exs{\end{enumerate}\end{examps}}
\def\bd{\begin{definition}}
\def\ed{\end{definition}}
\def\bp{\begin{proposition}}
\def\ep{\end{proposition}}
\def\be{\begin{equation}}
\def\ee{\end{equation}}
\def\hf{\frac{1}{2}}
\def\wh{\widehat}
\def\bea{\begin{eqnarray}}
\def\eea{\end{eqnarray}} 

\def\Om{\Omega}
\def\Id{\mathrm {Id}}

\newcommand{\E}{\mathcal E}
\newcommand{\A}{\mathcal A}

\newcommand{\BigO}[1]{\ensuremath{\operatorname{\mathcal{O}}(#1)}}

\usepackage{wrapfig}
\usepackage{color}
\usepackage{graphicx,framed}
\usepackage[colorlinks=true, pdfstartview=FitV, linkcolor=blue, citecolor=blue, urlcolor=blue]{hyperref}
\definecolor{shadecolor}{rgb}{0.95, 0.95, 0.86}

\def\blue#1{\textcolor[rgb]{0,0,1}{#1}}

\usepackage{amsbsy}
\usepackage{amssymb}

\newtheorem{theorem}{Theorem}[section]

\newtheorem{exercise}{Exercise}[section]
\newtheorem{examps}{Examples}[section]
\newtheorem{lemma} [theorem] {Lemma} 
\newtheorem{remark}[theorem]{Remark}
\newtheorem{problem}[theorem]{Problem}
\def\mod{\,\hbox{mod}\,}
\newtheorem{remarks}[theorem]{Remarks}
\newtheorem{proposition}[theorem]{Proposition}
\newtheorem{corollary}[theorem]{Corollary} 
\newtheorem{definition}[theorem]{Definition}
\def\le{\left}
\def\ri{\right}
\def\ds{\displaystyle}
\def\V {\mathcal V}

\def\bt{\begin{theorem}}
\def\et{\end{theorem}}
\def\bc{\begin{corollary}}
\def\ec{\end{corollary}}
\def\bx{\begin{examp}\small}
\def\ex{\end{examp}}
\def\bxr{\begin{exercise}\small}
\def\exr{\end{exercise}}
\def\bl{\begin{lemma}}
\def\el{\end{lemma}}
\def\bd{\begin{definition}}
\def\ed{\end{definition}}
\def\bp{\begin{proposition}}
\def\ep{\end{proposition}}

\def\br{\begin{remark}}
\def\er{\end{remark}}

\def\be{\begin{equation}}
\def\ee{\end{equation}}
\def\bi{\begin{itemize}}
\def\ei{\end{itemize}}
\def\ov {\overline}

\def\&{\hspace{-15pt}&}
\def\bea{\begin{eqnarray}}
\def\eea{\end{eqnarray}}
\def\beas{\begin{eqnarray*}}
\def\eeas{\end{eqnarray*}}

\def\C{{\mathbb C}}
\def\E{\mathcal E}

\def\R{{\mathbb R}}
\def\N{{\mathbb N}}

\def\wh{\widehat}

\def\G{{\Gamma}}
\def\Do{{\mathcal{D}}}
\def\Z{{\mathbb Z}}

\def\a{\alpha}

\def\l{\lambda}
\def\Pot{\mathcal {P}}
\def\1{{\bf 1}}
\def\qt{\frac 14}

\def\gg{{\mathfrak g}}
\def\tgg{\tilde{\mathfrak g}}

\begin{document}

\numberwithin{equation}{section}

\title{Singular limits of certain Hilbert-Schmidt integral operators}
\author{M. Bertola, E. Blackstone, A. Katsevich and A. Tovbis }
\thanks{The work of the first author was supported in part by the Natural Sciences and Engineering Research Council of Canada (NSERC). The work of the second author was supported in part by the European Research Council, Grant Agreement No. 682537. The work of the third author was supported in part by NSF grants DMS-1615124 and DMS-1906361. The work of the fourth author was supported in part by NSF grant DMS- 2009647.\\
Concordia University, Montreal, Canada (first author), Department of Mathematics, University of Michigan, Ann Arbor, MI 48109-1043 (second author), and
Department of Mathematics, University of Central Florida, Orlando, FL 32816-1364 (third and fourth authors)\\
{\it E-mail address:} bertola@mathstat.concordia.ca, eblackst@umich.edu, Alexander.Katsevich@ucf.edu,\\Alexander.Tovbis@ucf.edu}
\date{}

\begin{abstract}
In this paper we study the small-$\lambda$ spectral asymptotics of an integral operator $\mathscr{K}$ defined on two multi-intervals $J$ and $E$, when the multi-intervals touch each other (but their interiors are disjoint). The operator $\mathscr{K}$ is closely related to the multi-interval Finite Hilbert Transform (FHT). This case can be viewed as a singular limit of self-adjoint Hilbert-Schmidt integral operators with so-called integrable kernels, where the limiting operator is still bounded, but has a continuous spectral component. The regular case when $\text{dist}(J,E)>0$, and $\mathscr{K}$ is of the Hilbert-Schmidt class, was studied in an earlier paper by the authors. The main assumption in this paper is that $U=J\cup E$ is a single interval. We show that the eigenvalues of $\mathscr{K}$, if they exist, do not accumulate at $\l=0$. Combined with the results in an earlier paper by the authors, this implies that $H_p$, the subspace of discontinuity (the span of all eigenfunctions) of $\mathscr{K}$, is finite dimensional and consists of functions that are smooth in the interiors of $J$ and $E$. We also obtain an approximation to the kernel of the unitary transformation that diagonalizes $\mathscr{K}$, and obtain a precise estimate of the exponential instability of inverting $\mathscr{K}$. Our work is based on the method of Riemann-Hilbert problem and the nonlinear steepest-descent method of Deift and Zhou.
\end{abstract}

\maketitle

\section{Introduction}\label{sec-intro}

Let $J,E$ be multi-intervals, where a multi-interval is defined to be a union of finitely many non-intersecting, closed, possibly unbounded, intervals, each with non empty interior.  We assume that $J,E$ are bounded and $\mathring{J}\cap\mathring{E}=\emptyset$, where $\mathring{J}$ denotes the interior of the set $J$. Endpoints of $J$ or $E$ which belong to both $J,E$ are called \textit{double} endpoints. Let
\begin{align}\label{kernel K}
    K(x,y):=\frac{\chi_J(x)\chi_E(y)-\chi_J(y)\chi_E(x)}{\pi(x-y)},
\end{align}
where $\chi_J$ denotes the characteristic function of $J$, and define the operator $\mathscr{K}:L^2(U)\to L^2(U)$, with $U:=J\cup E$, as
\begin{align}\label{operator K}
    \mathscr{K}[f](x):=\int_{U}K(x,y)f(y)dy.
\end{align}
The operator $A:L^2(J)\to L^2(E)$, defined as
\begin{align}\label{FHT A}
    A[f](x)=\frac{1}{\pi}\int_J\frac{f(y)}{y-x}dy,
\end{align}
is known as the finite Hilbert transform (FHT). Let $A^\dagger$ denote its adjoint.  It is simple to verify that $A^\dagger:L^2(E)\to L^2(J)$ is given by
\begin{align}
    A^\dagger[g](y)=\frac{1}{\pi}\int_E\frac{g(x)}{y-x}dx,
\end{align}
$\mathscr{K}$ is self-adjoint, $\mathscr{K}=A\oplus A^\dagger$, $\mathscr{K}$ can be represented in matrix form as
\be\label{K-def}
    \mathscr{K}=\begin{bmatrix} 0 & A \\ A^\dagger & 0 \end{bmatrix}:L^2(J)\oplus L^2(E)\to L^2(J)\oplus L^2(E),
\ee
and
\begin{align*}
    \mathscr{K}^2=AA^\dagger\oplus A^\dagger A=\begin{bmatrix} AA^\dagger & 0 \\ 0 & A^\dagger A \end{bmatrix}:L^2(J)\oplus L^2(E)\to L^2(J)\oplus L^2(E).
\end{align*}
The kernel of $\mathscr{K}$ is an `integrable' kernel, and thus it is well-known that the resolvent of $\mathscr{K}$ can be expressed in terms of the solution to a particular Riemann-Hilbert problem (RHP) \cite{IIKS}.

Our main interest is to study spectral properties of $\mathscr{K}$ and the operators $A^\dagger A$ and $A A^\dagger$, which depend in an essential way on the geometry of $J$ and $E$.  The case when $U:=J\cup E=\mathbb R$ was considered in \cite{BKT19}, where it was shown that the spectrum of $A^\dagger A$ and $A A^\dagger$ is the segment $[0,1]$, the spectrum is purely absolutely continuous, and its multiplicity equals to the number of double endpoints.  An explicit diagonalization of these two operators was also discovered, see \cite[Theorem 5.3]{BKT19}.  In a subsequent paper \cite{BKT20} the setting is exactly the same as in this paper, i.e. $U=J\cup E \subsetneq \mathbb R$. Using the results of \cite{BKT19}, the following Theorem (\cite[Theorem 2.1]{BKT20}) was established:


\begin{theorem}\label{thm: spec prop}
Let $\mathscr K = A \oplus A^\dagger: L^2(U) \to L^2(U)$ be the operator given by \eqref{K-def} and let $\Sp(\mathscr K)$ denote the spectrum of $\mathscr{K}$. Here $U=J\cup E$, and $J,E\subset\R$ are multi-intervals with disjoint interiors. One has:
\begin{enumerate}
\item $\Sp(\mathscr K)\subseteq [-1,1]$;
\item The points $\l=\pm 1$ and $\l=0$ are not eigenvalues of $\mathscr K$. The eigenvalues of $\mathscr K$, if they exist, are {symmetric with respect to $\l=0$ and have finite multiplicities}. Moreover, they can accumulate only at $\l=0$;
\item If there are $n\ge 1$ ($n\in\N$) double endpoints, then $\Sp_{ac}(\mathscr K) = [-1,1]$, and the multiplicity of $\Sp_{ac}(\mathscr K)$ equals $n$;
 \item If there are no double endpoints, then $\mathscr K$ is of trace class. In this case, $\Sp(\mathscr K)$ consists only of eigenvalues and $\l=0$, which is the accumulation point of the eigenvalues; 
\item The singular continuous component is empty, i.e., $\Sp_{sc}(\mathscr K) = \emptyset$.
\end{enumerate}
\end{theorem}
\noindent

A similar Theorem regarding the spectral properties of $A^\dagger A$ and $A A^\dagger$ was also established, see \cite[Theorem 2.4]{BKT20}.

The main goal of this paper is to show that when $U$ consists of a single interval, then the eigenvalues of $\mathscr K$ do not accumulate at zero. Combined with the above Theorem this implies that $\mathscr K$ (and, therefore, both $A^\dagger A$ and $A A^\dagger$) has finitely many eigenvalues. This is achieved by reducing the spectral analysis of $\mathscr K$ to solving a RHP, and finding its approximate solution using the Deift-Zhou steepest descent method. As a byproduct of our analysis, we find an approximation to the kernel of the unitary transformation that diagonalizes $\mathscr{K}$ (and, therefore, $A^\dagger A$, and $A A^\dagger$), and obtain a precise estimate of the exponential instability (or, the degree of ill-posedness) of inverting the operators $\mathscr{K}$, $A^\dagger A$, and $A A^\dagger$. 

More precisely, we show that as $\l\to0$ ($\ka=-\ln|\l/2|\to\infty$), to leading order solving the equation $\mathscr{K}f=\phi$ is equivalent to computing integrals of the form
\be\begin{split}
    \int^{\infty}_{*} e^{\ka}{\begin{Bmatrix}\cos\\ \sin
    \end{Bmatrix}}(\gg_{\text{im}}'(x_0)(x-x_0)\ka)\tilde G_k(\ka)d\ka,
\end{split}
\ee
where $\tilde G_k(\ka)$ can be computed from the right-hand side $\phi$ in a stable way (i.e., as a bounded operator between the appropriate $L^2$ spaces). This means that the degree of ill-posedness of inverting $\mathscr{K}$ is $\exp\left(\ka/|\gg_{\text{im}}'(x_0)|\right)$. Our result provides a relationship between the instability of inverting $\mathscr{K}$ and the geometry of the problem (location of the intervals $J$ and $E$), which is encapsulated in the function $\gg_{\text{im}}'(x_0)$. According to Proposition \ref{prop: g-function}, the function $\gg_{\text{im}}'$ is positive on $\mathring{E}$, negative on $\mathring{J}$ and has $O((z-b_j)^{-\hf})$ behavior at double points $b_j$. Similar results are valid for inverting the operators $A^\dagger A$ and $A A^\dagger$.

For an interval $I\subset \mathbb R$, let $L^2(I,\mathbb R^n)$ denote the Hilbert space of $n$-dimensional real vector-valued functions defined on $I$: $\vec h(\l)=(h_1(\l),\dots,h_n(\l))$, $\l\in I$, with the standard definition of the norm:
\be\label{hnorm}
\Vert \vec h \Vert_{L^2(I,\mathbb R^n)}^2:=\sum_{j=1}^n\int_I h_j^2(\l) d\l.
\ee

Let $H_p, H_{c},  H_{ac}\subset L^2(U)$ denote the subspaces of discontinuity (the span of all eigenfunctions), continuity, and absolute continuity, respectively, with respect to $\mathscr{K}$. By Theorem~\ref{thm: spec prop}, $H_c=H_{ac}$. If $H_p\not=\emptyset$, let $L'$ be the number of distinct positive eigenvalues of $\mathscr{K}$: $0<\l_1 < \dots < \l_{L'}$. It follows from statement (2) of Theorem~\ref{thm: spec prop} that $-\l_j$, $j=1,2,\dots,L'$, are the negative eigenvalues of $\mathscr{K}$. Let $\Id$ denote the identity operator.

\begin{theorem}[main result]\label{thm:main}
Assume there are $n\geq1$ double points and no gaps, i.e. $U=J\cup E=[a_1,a_2]$.  One has:
\begin{enumerate}
\item If $H_p\not=\emptyset$, then $2L:=\text{dim}(H_p)<\infty$ and $H_p\subset C^\infty(\mathring{J}\cup\mathring{E})$.
\item There exist $n$ bounded operators $\Q_j:\,L^2(U)\to L^2([-1,1],\R)$, $1\leq j\leq n$, and a projector $P:L^2(U)\to \mathbb R^{2L}$ such that the operator $\U:L^2(U)\to L^2([-1,1],\mathbb{R}^n)\times \mathbb{R}^{2L}$ given by 
\be\label{main-diag-oper}
\U f=(\U_{ac}f, Pf),\ \U_{ac}:=(\Q_1,\dots,\Q_n):\,L^2(U)\to L^2([-1,1],\R^n),
\ee
is unitary, and 
\be\begin{split}\label{diag-two-comps}
   \U \mathscr{K}\U^\dagger=&\lambda\Id \text{ on }L^2([-1,1],\mathbb R^n),\\
   \U \mathscr{K}\U^\dagger=&D \text{ on }\mathbb R^{2L}.
\end{split}
\ee
Here $D$ is a diagonal matrix, consisting of the eigenvalues of $\mathscr{K}$ repeated according to their multiplicity. The first line is understood in the sense of operator equality, where $\lambda$ is a multiplication operator $\mathbb R^n\to \mathbb R^n$. The operators $\Q_j$ are not unique. If $H_p=\emptyset$, then $P$ is absent from \eqref{main-diag-oper} (i.e., $T=T_{ac}$), and the second line is absent from \eqref{diag-two-comps}.
\item The operators $\Q_j$ can be found so that their kernels $Q_j(x;\l)$ satisfy 
\be\label{Qj smoothness}
Q_j\in C^{\infty}\left((\mathring{J}\cup\mathring{E})\times \left((-1,1)\setminus\Xi\right)\right), 1\leq j\leq n,
\ee
where $\Xi$ is a finite set, and, for any closed interval $I\subset \mathring{J}\cup \mathring{E}$, $n\ge 0$, and $\e>0$, 
\be\label{E phi 3 II}
\sup_{x\in I,\e<|\l|<1}\left|(d/dx)^{n} Q_j(x;\l)\right|<\infty.
\ee
\item Explicit expressions that approximate $|\l|^{1/2}Q_j(x;\l)$, $k=1,\ldots,n$, with accuracy $\BigO{\varkappa^{-1}}$, where $\ka=-\ln|{\l}/{2}|$, are given in Proposition~\ref{Eprime simple} below. 
\end{enumerate}
\end{theorem}

For an interval $I\subset U$, let $\pi_I^\dagger:L^2(I)\to L^2(U)$ denote the extension by zero of a function defined on $I$ to all of $U$.
\begin{corollary}\label{cor:main}
Assume there are $n\geq1$ double points and no gaps, i.e. $J\cup E=[a_1,a_2]$. The following assertions hold if $I=J$, $B=A^\dagger A$ and if $I=E$, $B=AA^\dagger$. 
There exists a projector $P_I:L^2(I)\to \mathbb R^L$ such that the operator $\U_I:L^2(I)\to L^2([0,1],\mathbb R^n)\times R^L$ given by 
\be\label{main-diag-K2}
\U_I f=(\U_{ac,I}f, P_I f),\ \U_{ac,I}:=\U_{ac}\pi_I^\dagger:\,L^2(I)\to L^2([0,1],\mathbb R^n),
\ee
where $\U_{ac}$ is the same as in Theorem~\ref{thm:main}, is unitary, and 
\be\begin{split}\label{diag-two-K2}
   \U_I B\U_I^\dagger=&\lambda^2\Id \text{ on }L^2([0,1],\mathbb R^n),\\
   \U_I B\U_I^\dagger=&D \text{ on }\mathbb R^L.
\end{split}
\ee
Here $D$ is a diagonal matrix, consisting of the squares of positive eigenvalues of $\mathscr{K}$ repeated according to their multiplicity. The first line is understood in the sense of operator equality, where $\lambda^2$ is a multiplication operator $\mathbb R^n\to\mathbb R^n$. If $H_p=\emptyset$, then $P_I$ is absent from \eqref{main-diag-K2} (i.e., $\U_I=\U_{ac,I}$), and the second line is absent from \eqref{diag-two-K2}. 
\end{corollary}

The paper is organized as follows. In Section \ref{sec: K + RHP}, we introduce a RHP whose solution is denoted by $\Gamma(z;\lambda)$ and show that the kernel of the resolvent of $\mathscr{K}, AA^\dagger, A^\dagger A$ is expressed in terms of $\Gamma(z;\lambda)$.  {We also prove Corollary~\ref{cor:main}, assuming Theorem~\ref{thm:main} holds.}
Section \ref{sec: Gamma} is dedicated to computing the small-$\lambda$ asymptotics of $\Gamma(z;\lambda)$ in various regions of the complex plane.  In Section \ref{sec:jump}, we compute the small-$\lambda$ asymptotics of the jump of the kernel of the resolvent $\mathscr{R}$ in the $\lambda$-plane and express said jump as a quadratic form. In Section \ref{sec: exact to approx}, we finish the proof of Theorem~\ref{thm:main}, obtain small-$\l$ approximation to the kernels $|\l|^{1/2}Q_j(x;\l)$, and compute the degree of ill-posedness of inverting $\mathscr{K}$.  Lastly, supplementary material can be found in the Appendices.  A beautiful proof that certain complicated matrix is positive definite, which is based on the calculus of residues, deserves a special mention.  

\begin{remark}[Notation]
Throughout this text we will encounter several functions which are multi-valued on oriented curves (or union of oriented curves) in the complex plane.  Let $\Sigma\subset\mathbb{C}$ be an oriented curve with a $+$ and $-$ side, determined by the orientation, and suppose $f:\mathbb{C}\setminus\Sigma\to\mathbb{C}$ is analytic.  For any $w\in\Sigma$, we let $f(w_+)$, $f(w_-)$ denote the limiting value of $f(z)$ as $z\notin\Sigma$ approaches $w\in\Sigma$, non-tangentially, from the $+$, $-$ side of $\Sigma$, respectively.
\end{remark}

\section{The resolvent and resolution of the identity of $\mathscr{K}$}\label{sec: K + RHP}

We begin with the RHP for $\Gamma(z;\lambda)$ and show its relation to the operator $\mathscr{K}$.
\begin{problem}\label{RHPGamma}
Find a $2\times2$ matrix-function $\Gamma(z;\lambda)$ which, for any $\lambda\in\mathbb{C}\setminus[-1,1]$, satisfies:
\begin{enumerate}
    \item $\Gamma(z;\lambda)$ is analytic for $z\in\overline{\mathbb{C}}\setminus U$,
    \item $\Gamma(z;\lambda)$ has the jump condition
    \begin{align}\label{gamma jump}
        \Gamma(z_+;\lambda)=\Gamma(z_-;\lambda)\left(\1-\frac{2i}{\lambda}f(z)g^t(z)\right), \qquad z\in \mathring{J}\cup\mathring{E},
    \end{align}
    where $\1$ is the identity matrix, $^t$ denotes transpose, and
    \begin{align}\label{vec fg}
        f^t(z):=\begin{bmatrix} \chi_E(z) & \chi_J(z) \end{bmatrix}, \qquad g^t(z):=\begin{bmatrix} -\chi_J(z) & \chi_E(z) \end{bmatrix},
    \end{align}
    \item $\Gamma(\infty;\lambda)=\1$,
    \item $\Gamma(z_\pm;\lambda)\in L^2_{\mathrm{loc}}(U)$.
\end{enumerate}
\end{problem}
Notice that \eqref{gamma jump} is equivalent to
\begin{align}
    \Gamma(z_+;\lambda)&=\Gamma(z_-;\lambda)\begin{cases}
     \begin{bmatrix} 1 & -\frac{2i}{\lambda} \\ 0 & 1 \end{bmatrix}, & z\in \mathring{E}, \\
     \begin{bmatrix} 1 & 0 \\ \frac{2i}{\lambda} & 1 \end{bmatrix}, & z\in \mathring{J}.
    \end{cases}
\end{align}
Condition (4) of RHP \ref{RHPGamma} guarantees that if a solution exists (a solution does indeed exist, see below \eqref{res kernel}), it must be unique. For convenience of matrix calculations, throughout the paper we use the Pauli matrices
$$
\s_1= \begin{bmatrix}
0 & 1\\
1&0
\end{bmatrix},~~~
\s_2= \begin{bmatrix}
0 & -i\\
i&0
\end{bmatrix},~~~
\s_3= \begin{bmatrix}
1 & 0\\
0&-1
\end{bmatrix}.
$$
\begin{remark}\label{rem-sym}
Due to the uniqueness of the solution to RHP \ref{RHPGamma}, we can easily discover the symmetries
\begin{align}
    \Gamma(z;\lambda)=\overline{\Gamma(\overline{z};\overline{\lambda})}, \qquad \Gamma(z;\lambda)=\sigma_3\Gamma(z;-\lambda)\sigma_3,\label{sig2 conj}
\end{align}
because $\Gamma(z;\lambda)$, $\overline{\Gamma(\overline{z};\overline{\lambda})}$, and $\sigma_3\Gamma(z;-\lambda)\sigma_3$ solve RHP \ref{RHPGamma}.  Moreover, if we wish to solve the RHP satisfying properties (1), (3), (4) from RHP \ref{RHPGamma} and the jump condition 
\begin{align*}
    \tilde{\Gamma}(z_+;\lambda)=\tilde{\Gamma}(z_-;\lambda)\left(\1+\frac{2i}{\lambda}g(z)f^t(z)\right), \qquad z\in \mathring{E}\cup\mathring{J},
\end{align*}
(i.e. the jump matrix on $\mathring{E},\mathring{J}$ from property (2) of RHP \ref{RHPGamma} have been interchanged) then
\begin{align*}
    \tilde{\Gamma}(z;\lambda)=\sigma_2\Gamma(z;\lambda)\sigma_2=\Gamma^{-t}(z;\lambda),
\end{align*}
where $^{-t}$ denote transpose followed by inverse. 
\end{remark}
Let $\mathscr{R}(\lambda):L^2(U)\to L^2(U)$ denote the resolvent of $\mathscr{K}$, defined by
\begin{align}\label{res of K}
    \Id+\mathscr{R}(\lambda)=\left(\Id-\frac{1}{\lambda}\mathscr{K}\right)^{-1}.
\end{align}

\br \label{rem-std}
The standard form of the resolvent of an operator $T$ is $\mathcal R(\lambda;T)=(\lambda{\Id}-T)^{-1}$ (according to \cite[p. 920, Section X.6]{DS57}).  Thus, the `normalized' resolvent of \eqref{res of K} satisfies
\begin{align*}
    \lambda^{-1}({\Id}+\mathscr{R}(\lambda))=\mathcal R(\lambda;T).
\end{align*}
\er
It was shown in \cite[Theorem 4.7, Section 4.2]{BKT20} that
\begin{enumerate}
    \item $\mathscr{R}(\lambda)$ is an integral operator with kernel $R:U\times U\times(\overline{\mathbb{C}}\setminus[-1,1])\to\mathbb{C}$ defined as
    \begin{align}\label{res kernel}
        R(x,y;\lambda)=\frac{g^t(y)\Gamma^{-1}(y;\lambda)\Gamma(x;\lambda)f(x)}{\pi\lambda(x-y)},
    \end{align}
    \item RHP \ref{RHPGamma} is solvable if and only if $\Id-\frac{1}{\lambda}\mathscr{K}$ has a bounded inverse { if and only if $\lambda\notin[-1,1]$},
    \item The solution of RHP \ref{RHPGamma} $\Gamma(z;\lambda)$ can be continued in $\lambda$ onto the upper, lower shores of $(-1,1)\setminus\{0\}$ and is denoted $\Gamma(z;\lambda_+), \Gamma(z;\lambda_-)$, respectively.
\end{enumerate}
The resolvent kernel $R(x,y;\lambda)$ has the symmetries $R(x,y;\lambda)=R(y,x;\lambda)$, $\overline{R(x,y;\lambda)}=R(x,y;\overline{\lambda})$ and can be expressed as 
\begin{multline}\label{full-resol}
    R(x,y;\lambda)=\chi_E(x)\chi_E(y)R_{EE}(x,y;\lambda)+\chi_E(x)\chi_J(y)R_{EJ}(x,y;\lambda) \\
    +\chi_J(x)\chi_E(y)R_{JE}(x,y;\lambda)+\chi_J(x)\chi_J(y)R_{JJ}(x,y;\lambda),
\end{multline}
where
\begin{align}\label{ker RE RJ}
    \begin{split}
    R_{EE}(x,y;\lambda)&:=\frac{\begin{vmatrix} \Gamma_{11}(y;\lambda) & \Gamma_{11}(x;\lambda) \\ \Gamma_{21}(y;\lambda) & \Gamma_{21}(x;\lambda) \end{vmatrix}}{\pi\lambda(x-y)}, \qquad R_{EJ}(x,y;\lambda):=\frac{\begin{vmatrix} \Gamma_{12}(y;\lambda) & \Gamma_{11}(x;\lambda) \\ \Gamma_{22}(y;\lambda) & \Gamma_{21}(x;\lambda) \end{vmatrix}}{\pi\lambda(x-y)}, \\
    R_{JE}(x,y;\lambda)&:=\frac{\begin{vmatrix} \Gamma_{11}(y;\lambda) & \Gamma_{12}(x;\lambda) \\ \Gamma_{21}(y;\lambda) & \Gamma_{22}(x;\lambda) \end{vmatrix}}{\pi\lambda(x-y)}, \qquad R_{JJ}(x,y;\lambda):=\frac{\begin{vmatrix} \Gamma_{12}(y;\lambda) & \Gamma_{12}(x;\lambda) \\ \Gamma_{22}(y;\lambda) & \Gamma_{22}(x;\lambda) \end{vmatrix}}{\pi\lambda(x-y)}.
    \end{split}
\end{align}
The expressions \eqref{full-resol}, \eqref{ker RE RJ} are simply obtained by expanding \eqref{res kernel}, and the symmetry properties are obtained via Remark \ref{rem-sym}.  Let $\G_j(z;\lambda)$ denote the $j$th column of the matrix $\Gamma(z;\lambda)$, $j=1,2$.
\br \label{rem-anal-columns} 
Note that $\Gamma_2(z_{+};\lambda)=\Gamma_2(z_{-};\lambda)$ on $J$. So, $\Gamma_2(z;\lambda)$ is analytic on and around $J$, whereas $\Gamma_1(z;\lambda)$ is analytic on and around $E$.
\er

Since the second column $\Gamma_2(z;\lambda)$ is analytic on $J$, and the first column $\Gamma_1(z;\lambda)$ is analytic on $E$, the symmetry condition from Remark \ref{rem-sym} yields
\be\label{GammaLambda symmetry1}
\G_2(y;\l)=\overline{\G_2(y;\bar\l)}, \qquad \G_1(x;\l)=\overline{\G_1(x;\bar\l)}
\ee
when $y\in J, x\in E$, so that the jump $\De_\l \G_j(z;\l)$ of $\G_j(z;\l)$ ($j=1,2$) over the segment of continuous spectrum on $\R$ when $y\in J, x\in E$ becomes 
\be\label{jump-Gam_2}
    \De_\l\G_2(y;\l)=-2i\Im\G_2(y;\l_-), \qquad \De_\l\G_1(x;\l)=-2i\Im\G_1(x;\l_-),
\ee
where $\Delta_xf(x):=f_+(x)-f_-(x)$.  Also notice that \eqref{GammaLambda symmetry1} implies that
\begin{align}\label{GammaLambda symmetry2}
    \begin{split}
    \Re[\Gamma_2(y;\lambda_+)]&=\Re[\Gamma_2(y;\lambda_-)], \qquad \Im[\Gamma_2(y;\lambda_+)]=-\Im[\Gamma_2(y;\lambda_-)], ~~~ y\in J, \\
    \Re[\Gamma_1(x;\lambda_+)]&=\Re[\Gamma_1(x;\lambda_-)], \qquad \Im[\Gamma_1(x;\lambda_+)]=-\Im[\Gamma_1(x;\lambda_-)], ~~ x\in E,
    \end{split}
\end{align}
for $\lambda\in(-1,1)\setminus\Sigma$, where 
\be
\Sigma:=\{0,\pm\l_1,\dots,\pm\l_{L'}\}.
\ee
Thus, $\Sigma$ is the set consisting of all the eigenvalues of $\mathscr{K}$ and $0$.  Applying the same considerations to the four kernels of \eqref{ker RE RJ}, we obtain
\begin{align}\label{jump-kerr RE RJ}
\begin{split}
    \De_{\lambda}R_{EE}(x,y;\lambda)=\frac{2i\Im\begin{vmatrix} \Gamma_{11}(y;\lambda_-) & \Gamma_{11}(x;\lambda_-) \\ \Gamma_{21}(y;\lambda_-) & \Gamma_{21}(x;\lambda_-) \end{vmatrix}}{-\pi\lambda(x-y)}, \quad \De_{\lambda}R_{EJ}(x,y;\lambda)&=\frac{2i\Im\begin{vmatrix} \Gamma_{12}(y;\lambda_-) & \Gamma_{11}(x;\lambda_-) \\ \Gamma_{22}(y;\lambda_-) & \Gamma_{21}(x;\lambda_-) \end{vmatrix}}{-\pi\lambda(x-y)}, \\
    \De_{\lambda}R_{JE}(x,y;\lambda)=\frac{2i\Im\begin{vmatrix} \Gamma_{11}(y;\lambda_-) & \Gamma_{12}(x;\lambda_-) \\ \Gamma_{21}(y;\lambda_-) & \Gamma_{22}(x;\lambda_-) \end{vmatrix}}{-\pi\lambda(x-y)}, \quad \De_{\lambda}R_{JJ}(x,y;\lambda)&=\frac{2i\Im\begin{vmatrix} \Gamma_{12}(y;\lambda_-) & \Gamma_{12}(x;\lambda_-) \\ \Gamma_{22}(y;\lambda_-) & \Gamma_{22}(x;\lambda_-) \end{vmatrix}}{-\pi\lambda(x-y)},
\end{split}
\end{align}
and thus we have the relation
\begin{multline}\label{jump-kerr}
    \Delta_\lambda R(x,y;\lambda)=\chi_E(x)\chi_E(y)\Delta_\lambda R_{EE}(x,y;\lambda)+\chi_E(x)\chi_J(y)\Delta_\lambda R_{EJ}(x,y;\lambda) \\
    +\chi_J(x)\chi_E(y)\Delta_\lambda R_{JE}(x,y;\lambda)+\chi_J(x)\chi_J(y)\Delta_\lambda R_{JJ}(x,y;\lambda).
\end{multline}
For future reference, notice that the numerators in \eqref{jump-kerr RE RJ} can be represented as 
\begin{align}
    &2i\le(\det\le[\Re\G_i(y;\lambda),\Im\G_j(x;\lambda_-)\ri]-\det\le[\Re\G_j(x;\lambda),\Im\G_i(y;\lambda_-)\ri]\ri), \label{jump-num}
\end{align}
where $(i,j)=(1,1), (2,1), (1,2), (2,2)$ corresponds to  $\De_{\lambda}R_{EE}(x,y;\lambda)$, $\De_{\lambda}R_{EJ}(x,y;\lambda)$, $\De_{\lambda}R_{JE}(x,y;\lambda)$, and $\De_{\lambda}R_{JJ}(x,y;\lambda)$, respectively. This follows immediately by writing $\G_j=\Re\G_j+i\Im\G_j$, $j=1,2$, and simplifying. 

\begin{proposition}\label{prop: DeltaAbsLambdaRJRE ids}
For $\lambda\in(-1,1)\setminus\Sigma$, we have the identities
\begin{align*}
    \Delta_\lambda R_{EE}(x,y;\lambda)&=\frac{-2i}{\pi\lambda(x-y)}\Im\begin{vmatrix} \Gamma_{11}(y;|\lambda|_-) & \Gamma_{11}(x;|\lambda|_-) \\ \Gamma_{21}(y;|\lambda|_-) & \Gamma_{21}(x;|\lambda|_-) \end{vmatrix}, & & x\in E, ~y\in E, \\
    \Delta_\lambda R_{EJ}(x,y;\lambda)&=\frac{-2i}{\pi|\lambda|(x-y)}\Im\begin{vmatrix} \Gamma_{12}(y;|\lambda|_-) & \Gamma_{11}(x;|\lambda|_-) \\ \Gamma_{22}(y;|\lambda|_-) & \Gamma_{21}(x;|\lambda|_-) \end{vmatrix}, & & x\in E, ~y\in J, \\
    \Delta_\lambda R_{JE}(x,y;\lambda)&=\frac{-2i}{\pi|\lambda|(x-y)}\Im\begin{vmatrix} \Gamma_{11}(y;|\lambda|_-) & \Gamma_{12}(x;|\lambda|_-) \\ \Gamma_{21}(y;|\lambda|_-) & \Gamma_{22}(x;|\lambda|_-) \end{vmatrix}, & & x\in J, ~y\in E, \\
    \Delta_\lambda R_{JJ}(x,y;\lambda)&=\frac{-2i}{\pi\lambda(x-y)}\Im\begin{vmatrix} \Gamma_{12}(y;|\lambda|_-) & \Gamma_{12}(x;|\lambda|_-) \\ \Gamma_{22}(y;|\lambda|_-) & \Gamma_{22}(x;|\lambda|_-) \end{vmatrix}, & & x\in J, ~y\in J,
\end{align*}
where $|\lambda|_-:=(-\lambda)_-$ when $\lambda<0$.  In particular, $\De_{\lambda}R_{EE}(x,y;\lambda)$ and $\De_{\lambda}R_{JJ}(x,y;\lambda)$ are odd in $\l$, while $\De_{\lambda}R_{EJ}(x,y;\lambda)$ and $\De_{\lambda}R_{JE}(x,y;\lambda)$ are even in $\l$. 
\end{proposition}

\begin{proof}
For $\lambda>0$ there is nothing to prove, so we only need to consider $\lambda<0$.  As is easily seen, if $A$ and $B$ are two functions defined in a neighborhood of some $\l_0\in(-1,0)$ and $-\l_0$, respectively, and {$A(\tilde\l)\equiv B(-\tilde\l)$ for any $\tilde\l\in\mathbb C\setminus [-1,1]$ close to $\l_0$}, then $\Delta_\l A(\l)|_{\l=\l_0}=-\Delta_\l B(\l)|_{\l=|\l_0|}$. Also, \eqref{sig2 conj} implies that $\Gamma_{ij}(x;-\l)=\Gamma_{ij}(x;\l)$ if $i=j$, and $\Gamma_{ij}(x;-\l)=-\Gamma_{ij}(x;\l)$ if $i\not=j$. Thus, replacing $\l$ with $-\l$ changes the sign of the numerators in the formulas for $R_{EE}$ and $R_{JJ}$, while the numerators in the formulas for $R_{EJ}$ and $R_{JE}$ do not change sign. Combining these two facts with \eqref{jump-kerr RE RJ} we prove the Proposition.
\end{proof}


Knowledge of the jump of the resolvent unlocks the \textit{resolution of the identity} (also known as the spectral family) $\mathcal{E}(\l):L^2(U)\to L^2(U)$ for $\mathscr{K}$, which satisfies for any $\Delta=(a,b)$:
\begin{align}\label{complete res of id}
(\mathcal{E}(\Delta)f)(x):=\frac{-1}{2\pi i}\lim_{\delta\to0+}\lim_{\e\to0+}
\int_{a+\delta}^{b-\delta}\left(\frac{\mathscr{R}[f](x;t+i\e)}{t+i\e}-\frac{\mathscr{R}[f](x;t-i\e)}{t-i\e}\right)dt.
\end{align}
Here we used \cite[p. 921]{DS57} for resolvents in standard form (see Remark \ref{rem-std}). Let $\E_{ac}(\l)$ be the family of projectors, which is obtained from $\E(\l)$ by replacing the projection onto $H_p$ with projection to zero. In other words, $\E(\De)=\E_{ac}(\De)$ for any interval that contains no eigenvalues of $\mathscr{K}$, and $\E_{ac}((\l_j-\e,\l_j+\e))\to {\bf 0}$ as $\e\to 0+$ for each eigenvalue $\l_j$. Pick any interval $\De$. Our definition implies that $\E_{ac}(\De)f=\E(\De)f$ for any $f\in H_{ac}$, and $\E_{ac}(\De)f=0$ for any $f\in H_p$. We use the subscript ``$ac$'', because $\E_{ac}$ is closely connected with $\mathscr{K}_{ac}$, the absolutely continuous part of $\mathscr{K}$. Let $E(x,y;\l)$ and $E_{ac}(x,y;\l)$ denote the kernels of $\E(\l)$ and $\E_{ac}(\l)$, respectively. Throughout the paper we use the notation $E'(x,y;\l)=\partial E(x,y;\l)/\partial\l$, and similarly for $E_{ac}$.

\begin{proposition}\label{prop: smoothness of res of id}
\noindent
\begin{enumerate}
\item One has
\be\label{Epr decomp}
E'(x,y;\l)=E_{ac}'(x,y;\l)+\sum_j P_j(x,y)\delta(\l-\l_j),
\ee
where $P_j\in C^\infty\bigl((\mathring{J}\cup \mathring{E})\times (\mathring{J}\cup \mathring{E})\bigr)$ is the kernel of the projector onto the eigenspace corresponding to $\l_j$, and the sum is over all the distinct eigenvalues $\l_j$ of $\mathscr{K}$;
\item $E_{ac}'(x,y;\l)$ is analytic in $(\mathring{J}\cup \mathring{E})\times (\mathring{J}\cup \mathring{E})\times ((-1,1)\setminus \{0\})$;
\item For any closed interval $I\subset \mathring{J}\cup \mathring{E}$, $n_1,n_2\ge 0$, and $\e>0$, one has
\be\label{resol-prop bdd}
\sup_{x,y\in I,\e<|\l|<1}\left|(d/dx)^{n_1}(d/dy)^{n_2}E_{ac}'(x,y;\l)\right|<\infty.
\ee
\end{enumerate}
\end{proposition}

\begin{proof}
Let $\l_0$ be a pole of $R(x,y;\l)$ (recall that $R$ is defined in \eqref{res kernel}). It follows from \cite[Section 5.1]{BKT20} that 
$R(x,y;\l_\pm)=R_{\l_0}^\pm(x,y)/(\l-\l_0)+O(1)$ near $\l_0$. By the symmetry with respect to complex conjugation (see \cite[eq. (3.1)]{Howland1969} and  \cite[Remark 4.8]{BKT20}), $\overline{R_{\l_0}^+(x,y)}=R_{\l_0}^-(x,y)$. Next we use the same argument as in \cite[Proof of Theorem 2]{Howland1969} to show that $R_{\l_0}^\pm(x,y)$ are real-valued and equal each other. We have
\be
E'(x,y;\l)=-\frac1{2\pi i\l}\Delta_\l R(x,y;\l)=-\frac1{\pi\l}\frac{\Im R_{\l_0}^+}{\l-\l_0}+O(1),\ \l\to\l_0.
\ee
Since $(\E'(\l)f,f)\ge0$ for any $f\in L^2(U)$ on each side of $\l_0$, this implies that $\Im R_{\l_0}^\pm(x,y)\equiv 0$. The residue of $R$ at the pole gives the corresponding projector (see \cite[Proof of Theorem 5.5]{BKT20}). Applying the same argument at all poles and using that $E(x,y;\l)$ is the kernel of the spectral family of $\mathscr{K}$ proves \eqref{Epr decomp}.

The fact that $E_{ac}'$ is analytic for any $\l\in(-1,1)\setminus\{0\}$ follows from \cite[eq. (4.20), proof of Lemma~4.14]{BKT20}. The latter asserts that $R(x,y;\l)$ is a meromorphic function of $\l$, $\Delta_\l R(x,y;\l)=R(x,y;\l_+)-R(x,y;\l_-)$, and the kernels $R(x,y;\l_+)$ and $R(x,y;\l_-)$ admit analytic continuation in $\l$ into the lower and upper half-planes, respectively. The analyticity in $x$ and $y$ follows from the identities of Proposition \ref{prop: DeltaAbsLambdaRJRE ids} and Remark~\ref{rem-anal-columns}.

To prove the Proposition it remains only to establish that the derivatives of $E_{ac}'(x,y;\l)$ with respect to $x$ and $y$ remain bounded as $\l\to\pm1$. To show this fact, introduce the function 
\be
\label{rhosurf}
\rho(\lambda) =-\frac 1 2 +   \frac 1{ i\pi} \ln \le(\frac {1 - \sqrt{1-\lambda^2}}\lambda \ri), \qquad \l(\rho)=-\frac{1}{\sin(\pi\rho)},
\ee
 that provides a conformal 
map between the slit $\l$-plane $\overline{\C}\setminus [-1,1]$ and the vertical strip $\Re \rho<\hf$.  $\l(\rho)$  is the inverse function. The other strips $|\Re \rho-k|<\frac 1 2$, $k\in \Z$, in the $\rho$-plane are mapped to the same slit of the $\l$-plane by  $\l(\rho)$ and represent the various sheets of the branched map $\l(\rho)$. Then the desired result is a consequence of the following facts, established  
 in  \cite{BKT20}: 
  (1) $R(x,y;\l(\rho))$ is meromorphic in the $\rho\in\mathbb C$ plane, 
  and; (2) $\l=\pm1$ (i.e., $\rho=\pm1/2$) are not poles, see Figure 3, Theorem 4.7, and Lemma 4.14 in \cite{BKT20}. 
\end{proof}

\begin{corollary}\label{prop: deriv res of id}
The kernel $E_{ac}'(x,y;\lambda)$, where $\l\in [-1,1]\setminus\Sigma$, is given by
\begin{multline}\label{deriv res of id kernel}
    E_{ac}'(x,y;\lambda)=\frac{-1}{2\pi i \l}\left(\chi_E(x)\chi_E(y)\Delta_\l R_{EE}(x,y;\l)+\chi_E(x)\chi_J(y)\Delta_\l R_{EJ}(x,y;\l) \right. \\
    \left. +\chi_J(x)\chi_E(y)\Delta_\l R_{JE}(x,y;\l)+\chi_J(x)\chi_J(y)\Delta_{\l}R_{JJ}(x,y;\l)\right).
\end{multline}
\end{corollary}

\begin{proof} By Proposition~\ref{prop: smoothness of res of id}, all we need to do to find $E_{ac}'(x,y;\l)$, $\l\in[-1,1]\setminus\Sigma$, is to compute the jump $\Delta_\l R(x,y;\l)$, and \eqref{jump-kerr}
completes the proof.
\end{proof}

Let $\mathscr{R}_E(\lambda^2)$, $\mathscr{R}_{J}(\lambda^2)$ denote the resolvents of $AA^\dagger$, $A^\dagger A$, in the sense of \eqref{res of K}, respectively.

\begin{proposition}\label{prop:Ksq}
We have the relations
\begin{align}
    \Id+\mathscr{R}_{E}(\lambda^2)&=\left(\Id-\frac{1}{\lambda^2}AA^\dagger\right)^{-1}=\Id+\pi_E\mathscr{R}(\lambda)\pi_E, \label{Res_E} \\
    \Id+\mathscr{R}_{J}(\lambda^2)&=\left(\Id-\frac{1}{\lambda^2}A^{\dagger}A\right)^{-1}=\Id+\pi_J\mathscr{R}(\lambda)\pi_J, \label{Res_J}
\end{align}
where $\mathscr{R}$ is given by \eqref{res of K}, and $\pi_E:L^2(U)\to L^2(E)$, $\pi_J:L^2(U)\to L^2(J)$ are the natural projections. Moreover, the kernels of $\mathscr{R}_E$, $\mathscr{R}_J$, are given by $R_{EE}$, $R_{JJ}$, respectively, see \eqref{ker RE RJ}.
\end{proposition}

\begin{proof}
From \eqref{res of K} we have
\begin{align}\label{res sum}
    \mathscr{R}(\lambda)=\sum_{j=1}^\infty\left(\frac{\mathscr{K}}{\lambda}\right)^j.
\end{align}
Using the matrix form of $\mathscr{K}$ \eqref{K-def}, it is clear that the odd powers in \eqref{res sum} are block off-diagonal and the even powers in \eqref{res sum} are block diagonal.  The results \eqref{Res_E}, \eqref{Res_J} follow by comparing the expansions of $(\Id-\lambda^{-2}AA^{\dagger})^{-1}$, $\pi_E\mathscr{R}(\lambda)\pi_E$ and $(\Id-\lambda^{-2}A^{\dagger}A)^{-1}$, $\pi_J\mathscr{R}(\lambda)\pi_J$.
\end{proof}

Combining Theorem~\ref{thm:main} and Proposition~\ref{prop:Ksq} proves Corollary~\ref{cor:main}.

\section{Small $\lambda$ asymptotics of $\Gamma(z;\lambda)$}\label{sec: Gamma}

\begin{figure}
\begin{center}
\begin{tikzpicture}[scale=1.25]

\draw[black,thick] (-5,0) -- (-2,0);
\draw[black,thick] (-1,0) -- (1,0);
\draw[black,thick] (3,0) -- (5,0);

\filldraw[black] (-5,0) circle [radius=1.5pt] node[anchor=north] {$a_1$};
\filldraw[black] (-4,0) circle [radius=1.5pt] node[anchor=north] {$b_1$};
\filldraw[black] (-3,0) circle [radius=1.5pt] node[anchor=north] {$b_2$};
\filldraw[black] (-2,0) circle [radius=1.5pt] node[anchor=north] {$a_2$};
\filldraw[black] (-1,0) circle [radius=1.5pt] node[anchor=north] {$a_3$};
\filldraw[black] (1,0) circle [radius=1.5pt] node[anchor=north] {$a_4$};
\filldraw[black] (3,0) circle [radius=1.5pt] node[anchor=north] {$a_5$};
\filldraw[black] (4,0) circle [radius=1.5pt] node[anchor=north] {$b_3$};
\filldraw[black] (5,0) circle [radius=1.5pt] node[anchor=north] {$a_6$};

\node[above] at (-4.5,0) {$E$};
\node[above] at (-3.5,0) {$J$};
\node[above] at (-2.5,0) {$E$};
\node[above] at (0,0) {$J$};
\node[above] at (3.5,0) {$J$};
\node[above] at (4.5,0) {$E$};

\end{tikzpicture}
\end{center}
\caption{An example of the sets $J$ and $E$ with $n=3$ double points and $2g+2=6$ endpoints.}\label{fig: J and E}
\end{figure}
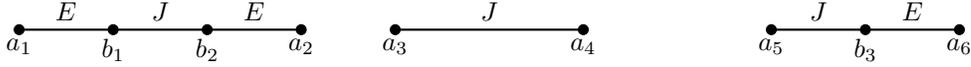

We follow the steepest descent method of Deift-Zhou, which was first introduced in \cite{DZ93}.  First, we remind the reader of the notation.  The set $U=E\cup J$ consists of $g+1$ disjoint intervals.  Endpoints of $J$ or $E$ which:
\begin{itemize}
    \item belong to both $J,E$ are called \textit{double} endpoints or just double points,
    \item belong to only $J$ or only $E$ are called \textit{simple} endpoints or just endpoints.
\end{itemize} 
We let $a_1,a_2,\ldots,a_{2g+2}\in\mathbb{R}$, $g\in\{0\}\cup\mathbb{N}$, satisfying $a_1<a_2<\cdots<a_{2g+2}$ denote the simple endpoints of $J$ and $E$; they are the endpoints of $g+1$ disjoint intervals. We let $b_1,b_2,\ldots,b_n\in\mathbb{R}$, $n\in\mathbb{N}$, satisfying $b_1<b_2<\cdots<b_n$, denote the double endpoints of $J$ and $E$; they lie in the interior of the above mentioned $g+1$ disjoint intervals.  See Figure \ref{fig: J and E} for an example.

\subsection{Transformations $\Gamma\to Y\to Z$}

To begin, we seek the $\gg$-function $\gg(z)$ which has the following properties:
\begin{itemize}
\item The function $\gg(z)$ satisfies the jump conditions
\begin{align}
    \gg(z_+)+\gg(z_-)&=-1, & & z\in \mathring{J}, \label{gJJump} \\
    \gg(z_+)+\gg(z_-)&=1, & & z\in \mathring{E}, \label{gEJump} \\
    \gg(z_+)-\gg(z_-)&=i\Om_j, & & z\in(a_{2j},a_{2j+1}), ~ j=1,\dots,g, \label{gGapJump}
\end{align}
where the constants $\Omega_j$ are to be determined.
\item $\gg(z)$ is analytic on $\overline{\mathbb{C}}\setminus U$,
\item $\gg(z)\in L^2_{loc}(U)$.
\end{itemize}
We introduce the characteristic function $\chi$ and radical $R:\mathbb{C}\setminus(J\cup E)\to\mathbb{C}$
\be
\chi(z)= \le\{
\begin{array}{cc}
1, & \text{if  $z\in \mathring{E}$} \\
-1, &\text{if  $z\in \mathring{J}$} 
\end{array}
\ri. , 
\qquad
R(\z)=\prod_{j=1}^{2g+2}(\z-a_j)^{\hf}, \label{radical}
\ee
with the branch of $R$ satisfying $R(z)\sim z^{g+1}$ as $z\ra \infty$.
\begin{proposition}\label{prop: g-function}
The $\gg$-function is given by
\be\label{gg}
\gg(z)=\frac{R(z)}{2\pi i}\le(\int_{U}\frac{\chi(\z) d\z}{(\z-z)R_+(\z)}+
\sum_{j=1}^{g}\int_{a_{2j}}^{a_{2j+1}}\frac{
{i\Om_j}
d\z}{(\z-z)R(\z)}\ri),
\ee
where the constants $\Om_j\in\mathbb{R}$ are (uniquely) chosen so that $\gg(z)$ is analytic at $z=\infty$.  Moreover,
\begin{enumerate}
    \item $\gg(z)$ is Schwarz symmetric (i.e. $\gg(z)=\overline{\gg(\overline{z})}$),
    \item $|\Re[\gg(z)]|<\frac{1}{2}$ for $z\in\overline{\mathbb{C}}\setminus U$,
    \item $\Re[\gg(z_\pm)]=\frac{1}{2}$ and $\Im[\gg'(z_{+})]>0$ for $z\in \mathring{E}$,
    \item $\Re[\gg(z_\pm)]=-\frac{1}{2}$ and $\Im[\gg'(z_{+})]<0$ for $z\in \mathring{J}$.
\end{enumerate}
\end{proposition}

\begin{proof}
The formula \eqref{gg} follows from the Sokhotski-Plemelj formula.  The Schwarz symmetry can now easily be verified via \eqref{gg} because $R(z)$ is Schwarz symmetric.  Using the Schwarz symmetry, we have
\begin{align*}
    \gg(z_+)=\overline{\gg(\overline{z_+})}=\overline{\gg(z_-)}, \qquad z\in\mathring{J}\cup\mathring{E}
\end{align*}
which immediately implies that
\begin{align*}
    \Re[\gg(z_+)]=\Re[\gg(z_-)], \qquad \Im[\gg(z_+)]=-\Im[\gg(z_-)], \qquad z\in\mathring{J}\cup\mathring{E}.
\end{align*}
It now follows from \eqref{gJJump}, \eqref{gEJump} that $\Re[\gg(z_\pm)]=-\frac{1}{2}$ for $z\in \mathring{J}$ and $\Re[\gg(z_\pm)]=\frac{1}{2}$ for $z\in \mathring{E}$.  Since $\Re[\gg(z)]$ is harmonic on $\overline{\mathbb{C}}\setminus(J\cup E)$, its maximum and minimum value can only be attained on $J\cup E$.  Lastly, consider $z=x+iy\in\mathring{E}$: since $\Re[\gg(z_\pm)]=\frac{1}{2}$ is maximized, we have $\frac{\partial}{\partial y}\Re[\gg(z_\pm)]<0$.  Thus, according to the Cauchy-Riemann equations, $0<\frac{\partial}{\partial x}\Im[\gg(z_+)]=\Im[\gg'(z_{+})]$, as desired.  The same idea can be applied when $z\in\mathring{J}$ to obtain $\Im[\gg'(z_{+})]<0$.
\end{proof}

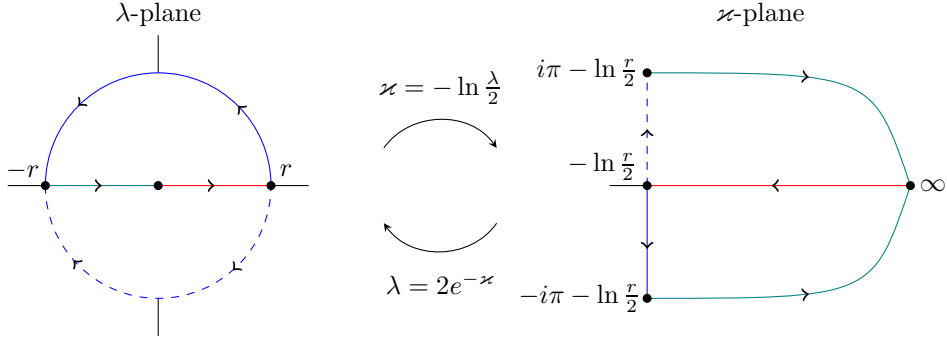
\begin{figure}
\begin{center}
\begin{tikzpicture}[scale=1.00]

\draw (-6,0) -- (-5.5,0);
\draw (-2.5,0) -- (-2,0);
\draw (-4,-2) -- (-4,-1.5);
\draw (-4,1.5) -- (-4,2);
\node[above] at (-4,2) {$\lambda$-plane};

\draw[blue,postaction = {decorate, decoration = {markings, mark = at position .25 with {\arrow[black,thick]{>};}}},postaction = {decorate, decoration = {markings, mark = at position .75 with {\arrow[black,thick]{>};}}}] (-2.5,0) arc (0:180:1.5cm);

\draw[blue,dashed,postaction = {decorate, decoration = {markings, mark = at position .25 with {\arrow[black,thick]{<};}}},postaction = {decorate, decoration = {markings, mark = at position .75 with {\arrow[black,thick]{<};}}}] (-5.5,0) arc (180:360:1.5cm);

\draw[red,postaction = {decorate, decoration = {markings, mark = at position .5 with {\arrow[black,thick]{>};}}}] (-4,0) -- (-2.5,0);

\draw[teal,postaction = {decorate, decoration = {markings, mark = at position .5 with {\arrow[black,thick]{>};}}}] (-5.5,0) -- (-4,0);

\filldraw[black] (-2.5,0) circle [radius=1.5pt];
\node at (-2.3,0.2) {$r$};

\filldraw[black] (-5.5,0) circle [radius=1.5pt];
\node at (-5.8,0.2) {$-r$};

\filldraw[black] (-4,0) circle [radius=1.5pt];

\draw (2,0) -- (2.5,0);
\draw[red,postaction = {decorate, decoration = {markings, mark = at position .5 with {\arrow[black,thick]{<};}}}] (2.5,0) -- (6,0);
\draw[blue,postaction = {decorate, decoration = {markings, mark = at position .5 with {\arrow[black,thick]{<};}}}] (2.5,-1.5) -- (2.5,0);
\draw[blue,dashed,postaction = {decorate, decoration = {markings, mark = at position .5 with {\arrow[black,thick]{>};}}}] (2.5,0) -- (2.5,1.5);
\node[above] at (4,2) {$\varkappa$-plane};

\draw[teal,postaction = {decorate, decoration = {markings, mark = at position .5 with {\arrow[black,thick]{>};}}}] (2.5,1.5) .. controls (5.5,1.5) .. (6,0);
\draw[teal,postaction = {decorate, decoration = {markings, mark = at position .5 with {\arrow[black,thick]{>};}}}] (2.5,-1.5) .. controls (5.5,-1.5) .. (6,0);

\filldraw[black] (2.5,0) circle [radius=1.5pt];
\node at (1.9,0.25) {$-\ln\frac{r}{2}$};

\filldraw[black] (2.5,1.5) circle [radius=1.5pt];
\node at (1.7,1.5) {$i\pi-\ln\frac{r}{2}$};

\filldraw[black] (2.5,-1.5) circle [radius=1.5pt];
\node at (1.6,-1.5) {$-i\pi-\ln\frac{r}{2}$};

\filldraw[black] (6,0) circle [radius=1.5pt];
\node at (6.3,0) {$\infty$};


\draw[-stealth] (-1,0.5) .. controls (-0.65,1) and (0.15,1) .. (0.5,0.5);
\node at (-0.25,1.3) {$\varkappa=-\ln\frac{\lambda}{2}$};

\draw[stealth-] (-1,-0.5) .. controls (-0.65,-1) and (0.15,-1) .. (0.5,-0.5);
\node at (-0.25,-1.3) {$\lambda=2e^{-\varkappa}$};

\end{tikzpicture}
\end{center}
\caption{The image of the ball $\{\lambda\in\mathbb{C}:|\lambda|\leq r\}$, $r>0$, under the map $\varkappa(\lambda)=-\ln\frac{\lambda}{2}$, where we understand that the green `curves' in the $\varkappa$-plane represent $|\Im\varkappa|=\pi$.}\label{fig: kappa}
\end{figure}

We introduce the important change of spectral variable given by
\begin{align}\label{def: kappa}
    \ka=\varkappa(\lambda):=-\ln\frac{\l}{2},
\end{align}
where the standard branch of the logarithm is taken (i.e. $|\Im\varkappa|\leq\pi$).  See Figure \ref{fig: kappa} for the image of a ball under this map.  Using this new spectral variable and the $\gg$-function, we use the transformation 
\begin{equation}\label{y-def}
Y(z;\varkappa)=e^{-\varkappa \gg(\infty)\sigma_3}\Gamma(z;2e^{-\ka}) e^{\varkappa \gg(z)\sigma_3},
\end{equation}
which reduces RHP \ref{RHPGamma} to the following RHP.

\begin{problem}\label{RHPY}
Find a $2\times 2$ matrix-function $Y(z;\ka)$ which, for any {fixed} $\lambda\in\mathbb{C}\setminus[-1,1]$, satisfies:
\begin{enumerate}
\item $Y(z;\ka)$ is analytic on $\overline{\C}\setminus[a_1,a_{2g+2}]$,
\item $Y(z;\ka)$ satisfies the jump conditions
\bea\label{jumpY}
\begin{split}
Y(z_+;\varkappa)& =Y(z_-;\varkappa)
\begin{bmatrix}
 e^{\ka (\gg(z_+) -\gg(z_-)) } & 0 \\ 
{i} e^{\ka( \gg(z_+) +\gg(z_-) +1) } & e^{-\ka( \gg(z_+) -\gg(z_-)) } 
\end{bmatrix}, & & z\in \mathring{J},
 \\
Y(z_+;\varkappa)&=Y(z_-;\varkappa)\begin{bmatrix} e^{\ka (\gg(z_+) -\gg _-(z)) } & 
 -i e^{-\ka(\mathcal \gg(z_+) + \gg(z_-) -1)  } \\ 0 & e^{-\ka(\gg(z_+) -\gg(z_-)) }\end{bmatrix}, & & z\in \mathring{E}, 
 \\
 Y(z_+;\varkappa)&=Y(z_-;\varkappa)e^{i\ka\Om_j \sigma_3}, & & z\in(a_{2j},a_{2j+1}),~ j =1,\dots, g,
\end{split}
\eea
\item  non-tangential boundary values of $Y(z;\ka)$  from the upper/lower half-planes  belong to $L^2_{loc}(U)$,
\item $Y(z;\varkappa)=\1+\BigO{z^{-1}}$ as $z\ra\infty$.
\end{enumerate}
\end{problem}
Notice that the jumps of $Y(z;\varkappa)$ on $\mathring{J}, \mathring{E}$ can be factorized as
\begin{align}
\begin{split}
    \begin{bmatrix}
 e^{\ka (\gg(z_+) -\gg(z_-)) } & 0 \\ 
{i} e^{\ka( \gg(z_+) +\gg(z_-) +1) } & e^{-\ka( \gg(z_+) -\gg(z_-)) } 
\end{bmatrix}&=\begin{bmatrix} 1 & -ie^{-\varkappa(2\gg(z_-)+1)} \\ 0 & 1 \end{bmatrix}(i\sigma_1)\begin{bmatrix} 1 & -ie^{-\varkappa(2\gg(z_+)+1)} \\ 0 & 1 \end{bmatrix}, \\
    \begin{bmatrix} e^{\ka (\gg(z_+) -\gg _-(z)) } & 
 - i e^{-\ka(\mathcal \gg(z_+) + \gg(z_-) -1)  } \\ 0 & e^{-\ka(\gg(z_+) -\gg(z_-)) }\end{bmatrix}&=\begin{bmatrix} 1 & 0 \\ ie^{\varkappa(2\gg(z_-)-1)} & 1 \end{bmatrix}(-i\sigma_1)\begin{bmatrix} 1 & 0 \\ ie^{\varkappa(2\gg(z_+)-1)} & 1 \end{bmatrix},
\end{split}
\end{align}
respectively, see \eqref{gJJump}, \eqref{gEJump}.  We now follow the standard procedure of opening lenses around each subinterval of $J,E$.  We call these sets ${\mathcal L_{E,J}^{(\pm)}}$, which are the regions inside the lenses around intervals $E,J$ and in the upper or lower half planes, respectively, see Figure \ref{fig: lenses}.  Now introduce the new unknown matrix
\bea
Z(z;\ka) =
\le\{\begin{array}{ll}
 Y(z;\ka), & \text{ outside the lenses,}  \\
\ds  Y(z;\ka) \le[\begin{matrix}
1 & 0\\
\mp {i}{ {\rm e}^{ \ka (2\gg(z) -1)}} & 1
\end{matrix}\ri], & z\in {\mathcal L_E^{(\pm)}},\\
\ds  Y(z;\ka) \le[\begin{matrix}
1 & {\pm i}  { {\rm e}^{ -\ka (2\gg(z) +1)}} \\
0& 1
\end{matrix}\ri], & z\in  \mathcal L_J^{(\pm)}.
 \end{array}
 \ri.
 \label{mat-Z}
\eea
We verify that $Z(z;\varkappa)$ is the solution to the following RHP:

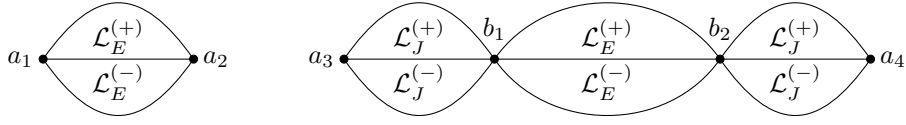
\begin{figure}
\begin{center}
\begin{tikzpicture}[scale=1.00]

\filldraw[black] (-4,0) circle [radius=1.5pt];
\filldraw[black] (-2,0) circle [radius=1.5pt];
\filldraw[black] (0,0) circle [radius=1.5pt];
\filldraw[black] (2,0) circle [radius=1.5pt];
\filldraw[black] (5,0) circle [radius=1.5pt];
\filldraw[black] (7,0) circle [radius=1.5pt];

\draw (-4,0) -- (-2,0);
\draw (0,0) -- (7,0);

\node at (-4.3,0) {$a_1$};
\node at (-1.7,0) {$a_2$};
\node at (-0.3,0) {$a_3$};
\node at (2,0.4) {$b_1$};
\node at (5,0.4) {$b_2$};
\node at (7.3,0) {$a_4$};

\node at (-3,0.3) {$\mathcal{L}_E^{(+)}$};
\node at (-3,-0.3) {$\mathcal{L}_E^{(-)}$};
\node at (1,0.3) {$\mathcal{L}_J^{(+)}$};
\node at (1,-0.3) {$\mathcal{L}_J^{(-)}$};
\node at (3.5,0.3) {$\mathcal{L}_E^{(+)}$};
\node at (3.5,-0.3) {$\mathcal{L}_E^{(-)}$};
\node at (6,0.3) {$\mathcal{L}_J^{(+)}$};
\node at (6,-0.3) {$\mathcal{L}_J^{(-)}$};

\draw[black] (-4,0) .. controls (-3.3,1) and (-2.7,1) .. (-2,0);
\draw[black] (-4,0) .. controls (-3.3,-1) and (-2.7,-1) .. (-2,0);
\draw[black] (0,0) .. controls (0.7,1) and (1.3,1) .. (2,0);
\draw[black] (0,0) .. controls (0.7,-1) and (1.3,-1) .. (2,0);
\draw[black] (2,0) .. controls (2.7,1) and (4.3,1) .. (5,0);
\draw[black] (2,0) .. controls (2.7,-1) and (4.3,-1) .. (5,0);
\draw[black] (5,0) .. controls (5.7,1) and (6.3,1) .. (7,0);
\draw[black] (5,0) .. controls (5.7,-1) and (6.3,-1) .. (7,0);

\end{tikzpicture}
\end{center}
\vspace{0pt}
\caption{The lenses $\mathcal{L}_{J,E}^{(\pm)}$ with $g=1$ and $n=2$ double points.} \label{fig: lenses}
\vspace{-0.0cm}
\end{figure}

\begin{problem}\label{RHPZ}
Find a $2\times2$ matrix-function $Z(z;\varkappa)$ which, for any {fixed} $\lambda\in\mathbb{C}\setminus[-1,1]$, satisfies:
\begin{enumerate}
\item $Z(z;\varkappa)$ is analytic in $\overline{\mathbb{C}}\setminus\left([a_1,a_{2g+2}]\cup\partial\mathcal{L}_J^{(\pm)}\cup\partial\mathcal{L}_E^{(\pm)}\right)$,
\item $Z(z;\varkappa)$ satisfies the jump conditions
\bea\label{Z jump}
\begin{split}
    Z(z_+;\varkappa)&=Z(z_-;\varkappa)(i\sigma_1), & & z\in \mathring{J}, \\
    Z(z_+;\varkappa)&=Z(z_-;\varkappa)(-i\sigma_1), & & z\in \mathring{E}, \\
    Z(z_+;\varkappa)&=Z(z_-;\varkappa)e^{i\varkappa\Omega_j\sigma_3}, & & z\in(a_{2j},a_{2j+1}), ~ j=1,\ldots,g, \\
    Z(z_+;\varkappa)&=Z(z_-;\varkappa)\begin{bmatrix} 1 & 0 \\ ie^{\varkappa(2\gg(z)-1)} & 1 \end{bmatrix}, & & z\in \partial\mathcal{L}_E^{(\pm)}, \\
    Z(z_+;\varkappa)&=Z(z_-;\varkappa)\begin{bmatrix} 1 & -ie^{-\varkappa(2\gg(z)+1)} \\ 0 & 1 \end{bmatrix}, & & z\in \partial\mathcal{L}_J^{(\pm)},
\end{split}
\eea
\item non-tangential boundary values of $Z(z;\varkappa)$ from the upper/lower half-planes belong to $L^2_{loc}(U)$,
\item $Z(z;\varkappa)=\1+\BigO{z^{-1}}$ as $z\ra\infty$.
\end{enumerate}
\end{problem}
Notice that the jumps of $Z(z;\varkappa)$ on $\partial\mathcal{L}_{E,J}^{(\pm)}$ are exponentially small provided that $z$ is a fixed distance from any endpoint of $J$ and $E$.  The large $\varkappa$ approximation of $Z(z;\ka)$ outside small discs around all endpoints (i.e. ignoring the jumps on $\partial\mathcal{L}_{E,J}^{(\pm)}$) is given by the outer parametrix (solution of the model RHP), which we denote $\Psi(z;\ka)$ and is constructed in the next Subsection. The large $\varkappa$ approximation of $Z(z;\ka)$ near the endpoints is described by local parametrices which are constructed in Subsection \ref{subsec: local parametrices}.

\subsection{The solution of the model problem $\Psi$}\label{ssect-sol_mod_gen}

We obtain the model RHP by ignoring the jumps of RHP \ref{RHPZ} for $z\in\partial\mathcal{L}_{E,J}^{(\pm)}$ and by prescribing  certain endpoint behavior when $z=a_j$, $j=1,\ldots,2g+2$, and $z=b_k$, $k=1,\ldots,n$.

\begin{problem}\label{modelRHP}
Find a $2\times2$ matrix-function $\Psi(z;\ka)$ which satisfies:
\begin{enumerate}
    \item $\Psi(z;\varkappa)$ is analytic on $\overline{\mathbb{C}}\setminus[a_1,a_{2g+2}]$,
    \item $\Psi(z;\varkappa)$ has the jump conditions
    \bea\label{jumpPsi}
    \begin{split}
    & \Psi(z_+;\varkappa)=\Psi(z_-;\varkappa) (-i\s_1), & & z\in \mathring{E}, \\
    & \Psi(z_+;\varkappa)=\Psi(z_-;\varkappa)(i\s_1), & & z\in \mathring{J}, \\
    & \Psi(z_+;\varkappa)=\Psi(z_-;\varkappa)e^{i\ka\Om_j\sigma_3}, & & z\in (a_{2j},a_{2j+1}), ~ j=1,\ldots,g,
    \end{split}
    \eea
    \item $\Psi(z;\varkappa)=\1+\BigO{z^{-1}}$ as $z\to\infty$,
    \item $\Psi(z;\varkappa)=\BigO{|z-a_j|^{-\frac{1}{4}}}$ as $z\to a_j$, $j=1,\ldots,2g+2$,
    \item $\Psi(z;\varkappa)=\BigO{|z-b_k|^{-\frac{1}{2}}}$ as $z\to b_k$, $k=1,\ldots,n$,
    \item $\Psi(z_\pm;\varkappa)$ has $L^2$ behavior for $z\in U\setminus\{b_1,\ldots,b_n\}$.
\end{enumerate}
\end{problem}

\begin{figure}
\begin{center}
\begin{tikzpicture}[scale=1.25]

\draw[black,thick] (-5,0) -- (-2,0);
\draw[black,thick] (-1,0) -- (1,0);
\draw[black,thick] (3,0) -- (5,0);

\filldraw[black] (-5,0) circle [radius=1.5pt] node[anchor=north] {$a_1$};
\filldraw[black] (-2,0) circle [radius=1.5pt] node[anchor=north] {$a_2$};
\filldraw[black] (-1,0) circle [radius=1.5pt] node[anchor=north] {$a_3$};
\filldraw[black] (1,0) circle [radius=1.5pt] node[anchor=north] {$a_4$};
\filldraw[black] (3,0) circle [radius=1.5pt] node[anchor=north] {$a_5$};
\filldraw[black] (5,0) circle [radius=1.5pt] node[anchor=north] {$a_6$};

\node[above] at (-3.5,0) {$\tilde{E}$};
\node[above] at (0,0) {$\tilde{J}$};
\node[above] at (4,0) {$\tilde{J}$};

\end{tikzpicture}
\end{center}
\caption{The sets $\tilde{J}$ and $\tilde{E}$ corresponding to the setup in Figure \ref{fig: J and E}.}\label{fig: tilde J and tilde E}
\end{figure}
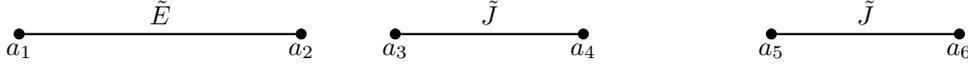

Note that solution of the RHP \ref{modelRHP} is not unique because $\Psi(z)\not\in L_{loc}^2$ near any double point $b_k$.  To properly fix a solution to RHP \ref{modelRHP}, we introduce the auxiliary $\tgg$-function and another model RHP.  First, we fix $\epsilon>0$ such that
\begin{align}
    \epsilon<\min_{\substack{j=1,\ldots,g+1 \\ k=1,\ldots,n}}|a_{2j-1}-b_k|,
\end{align}
define the sets
\begin{align}
    K_J=\left\{j\in\{1,\ldots,g+1\}:[a_{2j-1},a_{2j-1}+\epsilon]\subset J\right\}, \cr
    K_E=\left\{j\in\{1,\ldots,g+1\}:[a_{2j-1},a_{2j-1}+\epsilon]\subset E\right\},
\end{align}
and define the intervals
\begin{align}
    \tilde{J}=\bigcup_{j\in K_J}[a_{2j-1},a_{2j}], \qquad \tilde{E}=\bigcup_{j\in K_E}[a_{2j-1},a_{2j}].
\end{align}
See Figure \ref{fig: tilde J and tilde E} for an example of the sets $\tilde{J}$, $\tilde{E}$ corresponding to the sets $J$, $E$ as described in Figure \ref{fig: J and E}.  Notice that $\tilde{J}\cup\tilde{E}=J\cup E$ but $\tilde{J},\tilde{E}$ have no double points, i.e. $\tilde{J}\cap\tilde{E}=\emptyset$.  If $n=0$, then $J,E$ have no double points so $\tilde{J}=J$ and $\tilde{E}=E$.  If $g=0$, then $\tilde{J}=[a_1,a_2]$ and $\tilde{E}=\emptyset$ if $[a_1,b_1]\subset J$ or $\tilde{E}=[a_1,a_2]$ and $\tilde{J}=\emptyset$ if $[a_1,b_1]\subset E$.  We now define the $\tgg$-function to be the solution of the following scalar RHP:
\bi\label{RHPtgg}
\item $\tgg(z;\varkappa)$ satisfies the following jump conditions:
\begin{align}
    \tgg(z_+;\varkappa)+\tgg(z_-;\varkappa)&=\begin{cases}
    0, & z\in (\mathring{E}\cap\mathring{\tilde{E}})\cup(\mathring{J}\cap\mathring{\tilde{J}}), \\
    i\pi\cdot\text{sgn}\Im\varkappa, & z\in \mathring{E}\cap\mathring{\tilde{J}}, \\
    -i\pi\cdot\text{sgn}\Im\varkappa, & z\in \mathring{J}\cap\mathring{\tilde{E}},
    \end{cases} \label{tggJumps1} \\
    \tgg(z_+;\varkappa)-\tgg(z_-;\varkappa)&=W_j, ~~~ z\in(a_{2j},a_{2j+1}), ~ j=1,\dots,g, \label{tggJumps2}
\end{align}
where the constants $W_j$ are to be determined.
\item  $\tgg(z;\varkappa)$ is analytic on $\overline{\C}\setminus{[a_1,a_{2g+2}]}$ and is an $L^2_{loc}$ function on the jump contour.
\ei
If $\varkappa\in\mathbb{R}$, then $\mathrm{sgn}\Im\varkappa$ is to be understood in the sense of continuation from the upper or lower half plane.  It can be verified that
\be\label{tgg}
\tgg(z;\varkappa)=\frac{R(z)}{2\pi i}\le(\int_{J\cap\tilde{E}}\frac{ -i\pi\cdot\text{sgn}\Im\varkappa}{(\z-z)R_+(\z)}d\z+\int_{\tilde{J}\cap E}\frac{i\pi\cdot\text{sgn}\Im\varkappa }{(\z-z)R_+(\z)}d\z+
\sum_{j=1}^{g}\int_{a_{2j}}^{a_{2j+1}}\frac{
{W_j}
d\z}{(\z-z)R(\z)}\ri),
\ee
where $R(z)$ was defined in \eqref{radical} and the constants $W_j\in\R$ are (uniquely) chosen (in the standard way), so that $\tgg(z;\varkappa)$ is analytic at $z=\infty$.  According to \eqref{tgg}, 
\be\label{tgg-loc}
\tgg(z;\varkappa) = \pm \frac{1}{2}\text{sgn}(\Im z)\text{sgn}(\Im\varkappa)\ln(z-b_k)+\BigO{1}, \qquad z\to b_k,
\ee
where the sign `$-$' if we have $J$ to the left of $b_k$ and $E$ to the right and the sign is `$+$' in the opposite case.

We define 
\be\label{Psi_tilde}
\tilde \Psi(z;\ka)= e^{-\tgg(\infty;\varkappa)\s_3} \Psi(z;\ka) e^{\tgg(z;\varkappa)\s_3},
\ee
which is the solution to the following RHP:

\begin{problem}
\label{modelRHPtilde}
Find a $2\times2$ matrix-function $\tilde\Psi(z;\ka)$ which satisfies:
\begin{enumerate}
    \item $\tilde{\Psi}(z;\varkappa)$ is analytic on $\overline{\mathbb{C}}\setminus[a_1,a_{2g+2}]$,
    \item $\tilde{\Psi}(z;\varkappa)$ has the jump conditions
    \bea
    \begin{split}
        &\tilde{\Psi}(z_+;\varkappa)=\tilde{\Psi}(z_-;\varkappa)(-i\sigma_1), & & z\in\mathring{\tilde{E}}, \\
        &\tilde{\Psi}(z_+;\varkappa)=\tilde{\Psi}(z_-;\varkappa)(i\sigma_1), & & z\in\mathring{\tilde{J}}, \\
        &\tilde{\Psi}(z_+;\varkappa)=\tilde{\Psi}(z_-;\varkappa)e^{(i\varkappa\Omega_j+W_j)\sigma_3}, & & z\in(a_{2j},a_{2j+1}), ~ j=1,\ldots,g,
    \end{split}
    \eea
    \item $\tilde{\Psi}(z;\varkappa)=\1+\BigO{z^{-1}}$ as $z\to\infty$,
    \item $\tilde{\Psi}(z;\varkappa)=\BigO{|z-a_j|^{-\frac{1}{4}}}$ as $z\to a_j$, $j=1,2,\ldots,2g+2$.
\end{enumerate}
\end{problem}


\begin{remark}\label{rem-poles-in-kappa}
As is well known \cite{DIZ97}, a solution of RHP \ref{modelRHPtilde}, if one exists, is unique and can be expressed in terms of  Riemann $\Theta$-functions.  However, there can be (depending on the configuration of $\tilde{E}, \tilde{J}$) a countable, nowhere dense set  $\{\varkappa_{j}+i\pi k\}_{j,k\in\mathbb{Z}}$, $\varkappa_j\in\mathbb{R}$, of values of $\varkappa$, for which RHP \ref{modelRHPtilde} has no solution ($\tilde{\Psi}(z;\varkappa)$ has poles at $\varkappa=\varkappa_{j}+i\pi k$).  Let us discuss a few scenarios.  Suppose
\begin{itemize}
    \item $\tilde{E}=[a_1,a_2]\cup[a_{2g+1},a_{2g+2}]$ and $\tilde{J}=[a_3,a_4]\cup\cdots\cup[a_{2g-1},a_{2g}]$, $g\geq2$; the asymptotics of $\varkappa_j$, together with the asymptotics of the corresponding eigenvectors of the associated operator $\mathscr{K}$ from \eqref{operator K}, was established in \cite[eq. (7.19), (7.30)]{BKT16}.
    
    \item $\tilde{E}=[a_1,a_2]$ and $\tilde{J}=[a_3,a_4]$ (i.e. $g=1$); it can be shown that
    \begin{align*}
        \varkappa_j=\frac{i\pi}{\tau}\left(j-\frac{1}{2}\right), \qquad \mathrm{where} ~ \tau=i\frac{\int_{a_3}^{a_4}\frac{ds}{|R(s)|}}{\int_{a_2}^{a_3}\frac{ds}{|R(s)|}}\in i\mathbb{R_+}.
    \end{align*}
    This is possible {to obtain} because the zeros of the Jacobi $\Theta$-function (i.e. Riemann $\Theta$-function for $g=1$) are explicitly known \cite[eq. 8.189.4]{GRtable}.
    
    \item $\tilde{E}=[a_1,a_2]$ and $\tilde{J}=\emptyset$ (this is the scenario of Section \ref{sec:jump}); there is no set of `bad' $\varkappa_j$ because RHP \ref{modelRHPtilde} is independent of $\varkappa$ and its solution is given by $\tilde \Psi(z)=\le(\frac{z-a_1}{z-a_2}\ri)^{\frac{\s_1}{4}}$.
\end{itemize}
\end{remark}


\bl \label{lem-mod-sol}
There exists a solution $\Psi(z;\ka)$ to RHP \ref{modelRHP} such that the matrix function $\tilde\Psi(z;\ka)$ given by \eqref{Psi_tilde} 
solves RHP \ref{modelRHPtilde}.
\el

\begin{proof}
Analyticity of $ \tilde\Psi(z;\ka)$ and its asymptotics at $z=\infty$ are clear.  We only need to check the jump conditions, which are easily verified via \eqref{jumpPsi} and the jump conditions for $\tgg(z)$.
\end{proof}
We fix the solution to the model RHP \ref{modelRHP} as  
\be\label{Psi_tilde_1}
\Psi(z;\ka)= e^{\tgg(\infty;\varkappa)\s_3} \tilde \Psi(z;\ka) e^{-\tgg(z;\varkappa)\s_3}.
\ee

\br\label{rem-anal-param}
Note that $\det \Psi\equiv 1$ and $\tilde \Psi(z;\ka)$ is analytic (and invertible) near $z=b_k$ on any shore of the cut $E\cup J$.  Therefore, the RHP \ref{modelRHPtilde} has a unique solution $\tilde \Psi(z;\ka)$, that, in general, can be constructed in terms of Riemann $\Theta$-functions. 
\er

\subsection{Local parametrices}\label{subsec: local parametrices}

Construction of the local parametrices at the endpoints $a_j$, $j=1,\dots,2g+2$ is essentially the same as in \cite[Section 4.3]{BKT16}.  Here we consider a parametrix at $b_k\in(a_{2j-1}, a_{2j})$.  Let us define the approximate solution to RHP \ref{RHPZ} as
\begin{align}
    \wh{Z}(z;\varkappa)=\begin{cases}
     \tilde{\mathcal{P}}_j(z;\varkappa), & z\in\mathcal{D}_{a_j}, ~~~ j=1,\ldots,2g+2, \\
     \mathcal{P}_k(z;\varkappa), & z\in\mathcal{D}_{b_k}, ~~~ k=1,\ldots,n, \\
     \Psi(z;\varkappa), & \mbox{elsewhere},
    \end{cases}\label{hatZ explicit}
\end{align}
where $\Psi(z;\varkappa)$ a the solution of RHP \ref{modelRHP} given by \eqref{Psi_tilde_1}, and $\tilde{\mathcal{P}}_j, \mathcal{P}_k$ are the parametrices.  The parametrices $\tilde{\mathcal{P}}_j$ are given by the standard Bessel type parametrix (see \cite[eq. 4.29]{BKT16}) so we omit the details.  However, the construction of the parametrices $\mathcal{P}_k$, in terms of hypergeometric functions, is new.  First, we define the local coordinate $\zeta_k(z)$ by relating the $\gg$-function and the 4-point $\gg_4$-function, see \eqref{g4}.  Explicitly,
\begin{align}\label{g_4=gg}
    \gg_4(\zeta_k(z))=\begin{cases}
     \gg(z), & \text{for}~(b_k-\delta,b_k)\subset E, \\
     -\gg(z), & \text{for}~(b_k-\delta,b_k)\subset J,
    \end{cases}
\end{align}
where $\delta>0$ is sufficiently small.  The following Proposition guarantees that such a $\zeta_k(z)$ exists.
\begin{proposition}\label{prop: zeta}
Let functions $\wh \phi(z), \wh \psi(z)$ be analytic in a disc $\Do$ centered at the origin and let $\phi(z)=\frac{\ln z}{i\pi}+\wh \phi(z)$,  $\psi(z)=\frac{\ln z}{i\pi}+\wh \psi(z)$.  Then, there exists a function $\z(z)=cz(1+y(z))$ analytic in $\tilde{\mathcal{D}}$, a disk centered at $z=0$ which is a subset of $\mathcal{D}$, where $c\neq0$ and $y(0)=0$, such that
\be\label{near1}
\phi(\z(z))=\psi(z).
\ee
\end{proposition}

\begin{proof} 
Substituting $\z(z)$ in \eqref{near1} and taking $c=e^{i\pi(\wh \psi(0)-\wh\phi(0))}$,
we obtain the equation
\be
F(z,y)=\frac{1}{i\pi} \ln(1+y)+\wh \phi(cz(1+y))-\wh \psi(z) -\wh \psi(0)+\wh\phi(0)=0,
\ee
which is true for $(z,y)=(0,0)$. Since
\be
\frac{\partial F}{\partial y}=\frac{1}{i\pi(1+y)}+cz\wh \phi'(cz(1+y))\neq 0
\ee
at $(z,y)=(0,0)$, the conclusion follows from the Implicit Function Theorem.
\end{proof}
We denote by $\Gamma_4(z;\lambda)$ the solution of RHP \ref{RHPGamma} in the scenario with 1 double point $b_1=0$, $E=[-a,0]$ and $J=[0,a]$, where $a>0$.  It was shown in \cite[eq. 17]{BBKT19} that $\Gamma_4(z;\lambda)$ can be explicitly expressed in terms of hypergeometric functions.  We also let $\Psi_4(z;\varkappa)$, $\tgg_4(\zeta_k)$ (see \eqref{Psi_4}, \eqref{tg4}, respectively) denote the solution of RHP \ref{modelRHP} in the same scenario.  When $(b_k-\delta,b_k)\subset E$, the parametrices $\mathcal{P}_{k}$ are given by

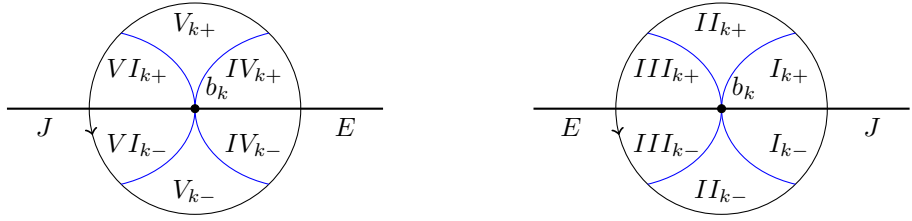
\begin{figure}
\begin{center}
\begin{tikzpicture}[scale=1.00]

\draw[black,thick] (1,0) -- (6,0);

\draw[black,postaction = {decorate, decoration = {markings, mark = at position .54 with {\arrow[black,thick]{>};}}}] (3.5,0) circle [radius=40pt];

\draw[blue] (3.5,0) arc [start angle=0, end angle=72.5, x radius=40pt, y radius=30pt];
\draw[blue] (3.5,0) arc [start angle=0, end angle=-72.5, x radius=40pt, y radius=30pt];
\draw[blue] (3.5,0) arc [start angle=180, end angle=252.5, x radius=40pt, y radius=30pt];
\draw[blue] (3.5,0) arc [start angle=180, end angle=107.5, x radius=40pt, y radius=30pt];

\filldraw[black] (3.5,0) circle [radius=1.5pt];

\node[above] at (3.8,0) {$b_k$};

\node[below] at (1.5,0) {$E$};
\node[below] at (5.5,0) {$J$};

\node[left] at (4.8,0.5) {$I_{k+}$};
\node[below] at (3.5,1.4) {$II_{k+}$};
\node[right] at (2.2,0.5) {$III_{k+}$};
\node[right] at (2.2,-0.5) {$III_{k-}$};
\node[above] at (3.5,-1.4) {$II_{k-}$};
\node[left] at (4.8,-0.5) {$I_{k-}$};

\draw[black,thick] (-6,0) -- (-1,0);

\draw[black,postaction = {decorate, decoration = {markings, mark = at position .54 with {\arrow[black,thick]{>};}}}] (-3.5,0) circle [radius=40pt];

\draw[blue] (-3.5,0) arc [start angle=0, end angle=72.5, x radius=40pt, y radius=30pt];
\draw[blue] (-3.5,0) arc [start angle=0, end angle=-72.5, x radius=40pt, y radius=30pt];
\draw[blue] (-3.5,0) arc [start angle=180, end angle=252.5, x radius=40pt, y radius=30pt];
\draw[blue] (-3.5,0) arc [start angle=180, end angle=107.5, x radius=40pt, y radius=30pt];

\filldraw[black] (-3.5,0) circle [radius=1.5pt];

\node[above] at (-3.2,0) {$b_k$};

\node[below] at (-5.5,0) {$J$};
\node[below] at (-1.5,0) {$E$};

\node[left] at (-2.2,0.5) {$IV_{k+}$};
\node[below] at (-3.5,1.4) {$V_{k+}$};
\node[right] at (-4.8,0.5) {$VI_{k+}$};
\node[right] at (-4.8,-0.5) {$VI_{k-}$};
\node[above] at (-3.5,-1.4) {$V_{k-}$};
\node[left] at (-2.2,-0.5) {$IV_{k-}$};

\end{tikzpicture}
\end{center}
\vspace{-10pt}
\caption{Two possible configurations of the set $\mathcal{D}_{b_k}$ and the regions $I_{k\pm},\ldots,VI_{k\pm}$.} \label{fig: para}
\vspace{-5pt}
\end{figure}

\be\label{ParamIeLeft}
\Pot_k(z;\ka)=
\begin{cases}
\Psi(z;\varkappa)\Psi^{-1}_4(\zeta_k;\varkappa)\Gamma_4(\zeta_k;2e^{-\varkappa})e^{\varkappa \gg_4(\zeta_k)\sigma_3}\begin{bmatrix} 1 & \pm ie^{-\varkappa(2\gg_4(\zeta_k)+1)} \\ 0 & 1 \end{bmatrix}, &\text{for $z\in I_{k\pm}$}, \\
\Psi(z;\varkappa)\Psi^{-1}_4(\zeta_k;\varkappa)\Gamma_4(\zeta_k;2e^{-\varkappa})e^{\varkappa\gg_4(\zeta_k)\sigma_3}, &\text{for $z\in II_{k\pm}$}, \\
\Psi(z;\varkappa)\Psi^{-1}_4(\zeta_k;\varkappa)\Gamma_4(\zeta_k;2e^{-\varkappa})e^{\varkappa\gg_4(\zeta_k)\sigma_3}\begin{bmatrix} 1 & 0 \\ \mp ie^{\varkappa(2\gg_4(\zeta_k)-1)} & 1 \end{bmatrix}, &\text{for $z\in III_{k\pm}$},
\end{cases}
\ee
where the regions $I_{k\pm},II_{k\pm},III_{k\pm}$ are described in Figure \ref{fig: para}, right panel.  When $(b_k-\delta,b_k)\subset J$, the parametrices $\mathcal{P}_{k}$ are given by
\be\label{ParamIeRight}
\Pot_k(z;\ka)=
\begin{cases}
\Psi(z;\varkappa)\s_2\Psi^{-1}_4(\zeta_k;\varkappa)\Gamma_4(\zeta_k;2e^{-\varkappa})e^{\varkappa\gg_4(\zeta_k)\sigma_3}\begin{bmatrix} 1 & 0 \\ \mp ie^{\varkappa(2\gg_4(\zeta_k)-1)} & 1 \end{bmatrix}\s_2, &\text{for $z\in VI_{k\pm}$}, \\
\Psi(z;\varkappa)\s_2\Psi^{-1}_4(\zeta_k;\varkappa)\Gamma_4(\zeta_k;2e^{-\varkappa})e^{\varkappa\gg_4(\zeta_k)\sigma_3}\s_2, &\text{for $z\in V_{k\pm}$}, \\
\Psi(z;\varkappa)\s_2\Psi^{-1}_4(\zeta_k;\varkappa)\Gamma_4(\zeta_k;2e^{-\varkappa})e^{\varkappa\gg_4(\zeta_k)\sigma_3}\begin{bmatrix} 1 & \pm ie^{-\varkappa(2\gg_4(\zeta_k)+1)} \\ 0 & 1 \end{bmatrix}\s_2, &\text{for $z\in IV_{k\pm}$},
\end{cases}
\ee
where the regions $IV_{k\pm},V_{k\pm},VI_{k\pm}$ are described in Figure \ref{fig: para}, left panel.


\begin{remark}\label{rem: S_e}
For a given $\epsilon>0$ we define  $\mathcal{S}_\epsilon\subset \C$ as
\begin{align}
    \mathcal{S}_\epsilon:=\left\{\varkappa\in\mathbb{C}:|\Im\varkappa|\leq\pi, \Re\varkappa\geq\ln(2), \text{ and } ||\Psi(z;\varkappa)^{\pm1}||<\frac{N(z)}{\epsilon}\right\},
\end{align}
where $\Psi(z;\varkappa)$ is the solution of RHP \ref{modelRHP}, given by \eqref{Psi_tilde_1}, $\displaystyle ||\Psi||:=\max_{i,j\in\{1,2\}}|\Psi_{i,j}|$ and $K(z)$ is independent of $\varkappa$.  The image of $\mathcal{S}_\epsilon$ under the map $\lambda=2e^{-\varkappa}$ is a subset of the unit ball $|\lambda|\leq1$, see Figure \ref{fig: kappa}.  In the case where $\tilde{E}=[a_1,a_2]\cup[a_{2g+1},a_{2g+2}], \tilde{J}=[a_3,a_4]\cup\cdots\cup[a_{2g-1},a_{2g}]$, $g\geq2$, studied in \cite{BKT16}, see also
Remark \ref{rem-poles-in-kappa}, the set $\mathcal{S}_\epsilon$ with a sufficiently small $\epsilon$ is the horizontal strip $|\Im\varkappa|\leq\pi$ without small closed neighborhood surrounding each pole $\varkappa_j, \varkappa_j\pm i\pi$ of the solution $\tilde\Psi(z;\varkappa)$ of RHP \ref{modelRHPtilde}.  A precise description of the set $\mathcal{S}_\epsilon$ for a general configuration of the intervals $\tilde{E}, \tilde{J}$ would require further study.  This study is not carried out in this paper since in Sections \ref{sec:jump}-\ref{sec: exact to approx} below we are interested only in the one interval case, for which $\mathcal{S}_\epsilon=\{\varkappa\in\mathbb{C}:|\Im\varkappa|\leq\pi, \Re\varkappa\geq\ln(2)\}$, see Remark \ref{rem-poles-in-kappa}.  For these reasons, we assume throughout the remainder of this Section that $\mathcal{S}_\epsilon$ is non-empty.
\end{remark}

We conclude this subsection by demonstrating the key properties of the parametrices.

\begin{lemma}\label{lem-param}
Let $\varkappa\in\mathcal{S}_\epsilon$.  The parametrix $\mathcal{P}_{k}(z;\varkappa)$ satisfies the following properties:
\begin{enumerate}
    \item $Z(z;\ka)\Pot_k^{-1}(z;\ka)$ is single-valued for $z\in\mathcal{D}_{b_k}$,
    \item $\Pot_k(z;\ka)\in L^2_{loc}$ on the jump contours $(\partial\mathcal{L}_J^{(\pm)}\cup\partial\mathcal{L}_E^{(\pm)}\cup E\cup J)\cap\mathcal{D}_{b_k}$ {provided that $\Re\varkappa\notin[\ln(2),\infty)$ (if and only if $\lambda\notin[-1,1]$)},
    \item for $z\in\partial\mathcal{D}_{b_k}$, we have the uniform approximation $\Psi(z;\ka)\Pot_k^{-1}(z;\ka)=\1+\BigO{\varkappa^{-1}}$ as $\ka\ra\infty$.
\end{enumerate}
\end{lemma}

\begin{proof}
We prove this statement assuming that $(b_k-\delta,b_k)\subset E$, i.e. using only \eqref{ParamIeLeft}.  The proof for \eqref{ParamIeRight} will be similar with the use of Remark \ref{rem-sym}.  From the jump properties \eqref{gamma jump}, \eqref{gJJump}, \eqref{gEJump}, \eqref{jumpPsi}, it is easily verified that $\mathcal{P}_{k}(z;\varkappa)$ has the exact same jumps as $Z(z;\varkappa)$ for $z\in\mathcal{D}_{b_k}$. \\

To prove the second requirement we notice that since $\Gamma_4(z;\lambda)$ is an $L^2_{loc}$ matrix valued function provided that $\lambda\notin[-1,1]$, $\mathcal{P}_{k}(z;\varkappa)$ will also be $L^2_{loc}$ provided that $\Psi(z;\ka)\Psi^{-1}_4(\z_k;\ka)$ is bounded on $\Do_{b_k}$.  Note that, according to \eqref{tgg-loc}, $\tgg(z;\varkappa)-\tgg_4(\z_k;\varkappa)$ does not have logarithmic singularity at $z=b_k$, so that $e^{\tgg(z;\varkappa)-\tgg_4(\z_k;\varkappa)}$ is bounded in a neighborhood of $b_k$ (recall that the $\tgg_4$ function is the $\tgg$ function with one double point $b_1=0$ and $[-a,0]=E$, $[0,a]=J$, see \eqref{g4}).  Taking into the account the fact that $\tilde\Psi^{\pm 1}$ from \eqref{Psi_tilde_1} is bounded in $\Do_{b_k}$, we obtain
\be
    \Psi(z;\ka)\Psi^{-1}_4(\z_k;\ka)= e^{\tgg(\infty;\varkappa)\s_3} \tilde \Psi(z;\ka) e^{-(\tgg(z;\varkappa)-\tgg_4(\z_k;\varkappa))\s_3}\le(\frac{z+a}{z-a}\ri)^{-\frac{\s_1}{4}}e^{-\tgg_4(\infty;\varkappa)\s_3},
\ee
which is bounded in $\Do_{b_k}$. \\

We prove the third requirement in regions $II_{k\pm}$ only, as the proof is similar for the other regions.  According to Theorem \ref{GammaAsmptotics},
\be
    \Psi^{-1}_4(\z_k;\ka)\G_4(\z_k;\ka)e^{\ka\gg_4(\z_k)\s_3}=\1 + \BigO{\ka^{-1}}~~~~~~~~~~~~{\rm as}~~~\ka\ra\infty,
\ee
and thus
\begin{equation}\label{P Psi inv asymp}
    \Pot_k(z;\ka)\Psi^{-1}(z;\ka)=\Psi(z;\ka)(\1 + \BigO{\ka^{-1}})\Psi^{-1}(z;\ka)=\1 + \BigO{\ka^{-1}},
\end{equation}
since $\Psi^{\pm 1}(z;\ka)$ are bounded on $\partial\Do_{b_k}$ when $\ka\in\mathcal{S}_\epsilon$ (see Remark \ref{rem: S_e}).  The proof is now complete.
\end{proof}
We note that the parametrices $\tilde{\mathcal{P}}_j(z;\varkappa)$, $j=1,\ldots,2g+2$, also satisfy properties (1)-(3) of Lemma \ref{lem-param}, see \cite[Prop. 4.14]{BKT16} for more details.

\subsection{Small-norm problem}

\begin{figure}
\begin{center}
\begin{tikzpicture}[scale=1.00]

\filldraw[black] (-4,0) circle [radius=1.5pt];
\filldraw[black] (-2,0) circle [radius=1.5pt];
\filldraw[black] (0,0) circle [radius=1.5pt];
\filldraw[black] (2,0) circle [radius=1.5pt];
\filldraw[black] (5,0) circle [radius=1.5pt];
\filldraw[black] (7,0) circle [radius=1.5pt];

\draw[black,postaction = {decorate, decoration = {markings, mark = at position .54 with {\arrow[black,thick]{>};}}}] (-4,0) circle [radius=0.6cm];
\draw[black,postaction = {decorate, decoration = {markings, mark = at position .54 with {\arrow[black,thick]{>};}}}] (-2,0) circle [radius=0.6cm];
\draw[black,postaction = {decorate, decoration = {markings, mark = at position .54 with {\arrow[black,thick]{>};}}}] (-0,0) circle [radius=0.6cm];
\draw[black,postaction = {decorate, decoration = {markings, mark = at position .54 with {\arrow[black,thick]{>};}}}] (2,0) circle [radius=0.6cm];
\draw[black,postaction = {decorate, decoration = {markings, mark = at position .54 with {\arrow[black,thick]{>};}}}] (5,0) circle [radius=0.6cm];
\draw[black,postaction = {decorate, decoration = {markings, mark = at position .54 with {\arrow[black,thick]{>};}}}] (7,0) circle [radius=0.6cm];

\node[above] at (-4,0) {$a_1$};
\node[above] at (-2,0) {$a_2$};
\node[above] at (0,0) {$a_3$};
\node[above] at (2,0) {$b_1$};
\node[above] at (5,0) {$b_2$};
\node[above] at (7,0) {$a_4$};

\draw[black,postaction = {decorate, decoration = {markings, mark = at position .54 with {\arrow[black,thick]{>};}}}] (-3.586,0.424) .. controls (-3.3,0.75) and (-2.7,0.75) .. (-2.424,0.424);
\draw[black,postaction = {decorate, decoration = {markings, mark = at position .54 with {\arrow[black,thick]{>};}}}] (-3.586,-0.424) .. controls (-3.3,-0.75) and (-2.7,-0.75) .. (-2.424,-0.424);
\draw[black,postaction = {decorate, decoration = {markings, mark = at position .54 with {\arrow[black,thick]{>};}}}] (0.424,0.424) .. controls (0.7,0.75) and (1.3,0.75) .. (1.586,0.424);
\draw[black,postaction = {decorate, decoration = {markings, mark = at position .54 with {\arrow[black,thick]{>};}}}] (0.424,-0.424) .. controls (0.7,-0.75) and (1.3,-0.75) .. (1.586,-0.424);
\draw[black,postaction = {decorate, decoration = {markings, mark = at position .54 with {\arrow[black,thick]{>};}}}] (2.424,0.424) .. controls (2.7,0.75) and (4.3,0.75) .. (4.586,0.424);
\draw[black,postaction = {decorate, decoration = {markings, mark = at position .54 with {\arrow[black,thick]{>};}}}] (2.424,-0.424) .. controls (2.7,-0.75) and (4.3,-0.75) .. (4.586,-0.424);
\draw[black,postaction = {decorate, decoration = {markings, mark = at position .54 with {\arrow[black,thick]{>};}}}] (5.424,0.424) .. controls (5.7,0.75) and (6.3,0.75) .. (6.586,0.424);
\draw[black,postaction = {decorate, decoration = {markings, mark = at position .54 with {\arrow[black,thick]{>};}}}] (5.424,-0.424) .. controls (5.7,-0.75) and (6.3,-0.75) .. (6.586,-0.424);

\end{tikzpicture}
\end{center}
\vspace{0pt}
\caption{The contour $\Sigma_T$ with $g=1$ and $n=2$ double points.} \label{fig: T jumps}
\vspace{-0.0cm}
\end{figure}
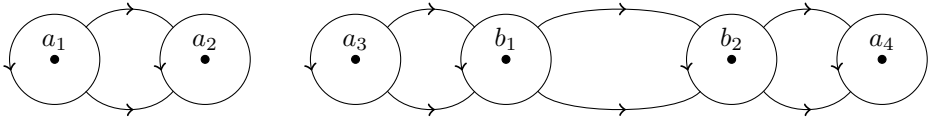

The error matrix $T:\mathbb{C}^{2\times2}\setminus\Sigma_T\to\mathbb{C}^{2\times 2}$ is defined as
\begin{align}\label{err def}
    T(z;\varkappa)=Z(z;\varkappa)\wh Z^{-1}(z;\varkappa),    
\end{align}
where $\Sigma_T:=\partial\mathcal{D}\cup\left((\partial\mathcal{L}_J^{(\pm)}\cup\mathcal{L}_E^{(\pm)})\setminus\mathcal{D}\right)$ and $\mathcal{D}:=\left(\cup_{j=1}^{2g+2}\mathcal{D}_{a_j}\right)\cup\left(\cup_{k=1}^n\mathcal{D}_{b_k}\right)$, see Figure \ref{fig: T jumps} for an example.  {We emphasize that $T(z;\varkappa)$ for $\lambda\in[-1,1]$ (see Figure \ref{fig: kappa} for corresponding $\varkappa$ values) is to be understood in terms of continuation from the upper or lower half plane.}  Using the definition of $\hat{Z}$ \eqref{hatZ explicit} and the jump properties of $Z$ \eqref{Z jump}, we see that
\begin{align}\label{T jumps}
    T(z_+;\varkappa)=T(z_-;\varkappa)\begin{cases}
     \Psi(z;\varkappa)\tilde{\mathcal{P}}^{-1}_j(z), & z\in\partial\mathcal{D}_{a_j}, ~ j=1,\ldots,g, \\
     \Psi(z;\varkappa)\mathcal{P}^{-1}_k(z), & z\in\partial\mathcal{D}_{b_k}, ~ k=1,\ldots,n, \\
     \Psi(z;\varkappa)\begin{bmatrix} 1 & -ie^{-\varkappa(2\gg(z)+1)} \\ 0 & 1 \end{bmatrix}\Psi^{-1}(z;\varkappa), & z\in\partial\mathcal{L}^{(\pm)}_E\setminus\mathcal{D}, \\
     \Psi(z;\varkappa)\begin{bmatrix} 1 & 0 \\ ie^{\varkappa(2\gg(z)-1)} & 1 \end{bmatrix}\Psi^{-1}(z;\varkappa), & z\in\partial\mathcal{L}^{(\pm)}_J\setminus\mathcal{D}, \\
     \1, & z\in \mathring{J}\cup \mathring{E}.
    \end{cases}
\end{align}
While proving Lemma \ref{lem-param}, we obtained the approximation
\begin{align*}
    \Psi(z;\varkappa)\tilde{\mathcal{P}}^{-1}_j(z)&=\1+\BigO{\varkappa^{-1}}, \qquad \mbox{uniformly for } z\in\partial\mathcal{D}_{a_j}, \\
    \Psi(z;\varkappa)\mathcal{P}^{-1}_k(z)&=\1+\BigO{\varkappa^{-1}}, \qquad \mbox{uniformly for } z\in\partial\mathcal{D}_{b_k},
\end{align*}
as $\varkappa\to\infty$ with $\varkappa\in\mathcal{S}_\epsilon$ for $j=1,\ldots,2g+2$, $k=1,\ldots,n$, see \eqref{P Psi inv asymp}.  Also note that $\Psi(z;\varkappa)$ is uniformly bounded for $z\in\partial\mathcal{L}^{(\pm)}_{J,E}\setminus\mathcal{D}$, $\varkappa\in\mathcal{S}_\epsilon$.  Thus, as $\varkappa\to\infty$ with $\varkappa\in\mathcal{S}_\epsilon$, we can write \eqref{T jumps} as
\begin{align}\label{T jumps asymp}
    T(z_+;\varkappa)=T(z_-;\varkappa)\begin{cases}
    \1, & z\in J\cup E, \\
    \1+\BigO{\varkappa^{-1}}, & z\in\partial\mathcal{D}, \\
    \1+\BigO{e^{-c\varkappa}}, & z\in\left(\partial\mathcal{L}_J^{(\pm)}\cup\partial\mathcal{L}_E^{(\pm)}\right)\setminus\mathcal{D},
    \end{cases}
\end{align}
where $c>0$ is sufficiently small.  In other words, $T(z;\varkappa)$ solves a small-norm problem, therefore it is well known (see, for example, \cite[p. 1529--1532]{DKMVZ99}) that
\begin{align}\label{E uni err}
    T(z;\varkappa)=\1+\BigO{\varkappa^{-1}} \qquad \mbox{as} ~ \varkappa\to\infty
\end{align}
uniformly for $z$ in compact subsets of $\mathbb{C}\setminus\Sigma_T$, provided $\varkappa\in\mathcal{S}_\epsilon$.  We conclude this section with a statement for the small $\lambda$ asymptotics of $\Gamma(z;\lambda)$.

\begin{theorem}\label{thm: gamma asymptotics}
Let $\varkappa\to\infty$ with $\varkappa\in\mathcal{S}_\epsilon$.
\begin{enumerate}
    \item If $z$ is in a compact subset of $\mathbb{C}\setminus(E\cup J)$, then we have the uniform approximation
\begin{align}
    \Gamma(z;2e^{-\varkappa})=e^{\varkappa \gg(\infty)\sigma_3}\left(\1+\BigO{\varkappa^{-1}}\right)\Psi(z;\varkappa)e^{-\varkappa \gg(z)\sigma_3}.
\end{align}
\item If $z_\pm$ is in a compact subset of $\mathring{E}\cup\mathring{J}$, then we have the uniform approximation
\begin{multline}
    \Gamma(z_\pm;2e^{-\varkappa})=e^{\varkappa \gg(\infty)\sigma_3}\left(\1+\BigO{\varkappa^{-1}}\right)\Psi(z_\pm;\varkappa) \\
    \times\begin{bmatrix} 1 & \mp ie^{-\varkappa(2\gg(z_\pm)+1)}\chi_{J}(z) \\ \pm ie^{\varkappa(2\gg(z_\pm)-1)}\chi_{E}(z) & 1 \end{bmatrix}e^{-\varkappa\gg(z_\pm)\sigma_3}.
\end{multline}
\item If $z_\pm\in(b_k-\delta,b_k)\cup(b_k,b_k+\delta)$ and $\delta>0$ is sufficiently small, then we have the uniform approximation
\begin{align*}
    \Gamma(z_\pm;2e^{-\varkappa})=e^{\varkappa \gg(\infty)\sigma_3}\left(\1+\BigO{\varkappa^{-1}}\right)\Psi(z_\pm;\varkappa)\begin{cases}
     \sigma_2\Psi_4^{-1}(\zeta_{k\pm};\varkappa)\Gamma_4(\zeta_{k\pm};2e^{-\varkappa})\sigma_2, &(b_k-\delta,b_k)\subset J, \\
     \Psi_4^{-1}(\zeta_{k\pm};\varkappa)\Gamma_4(\zeta_{k\pm};2e^{-\varkappa}), &(b_k-\delta,b_k)\subset E.
    \end{cases}
\end{align*}

\end{enumerate}
\end{theorem}

\begin{proof}
We begin by reversing the transformations of this section \eqref{y-def}, \eqref{mat-Z}, \eqref{err def} to obtain
\begin{align}
    \Gamma(z;\lambda)=e^{\varkappa\gg(\infty)\sigma_3}T(z;\varkappa)\wh{Z}(z;\varkappa)\begin{cases}
    e^{-\varkappa\gg(z)\sigma_3}, & z ~ \mbox{outside lenses}, \\
    \begin{bmatrix} 1 & \mp {i}{ e^{-\ka (2\gg(z) +1)}} \\ 0 & 1 \end{bmatrix}e^{-\varkappa\gg(z)\sigma_3}, & z\in\mathcal{L}_J^{(\pm)}, \\
    \begin{bmatrix} 1 & 0 \\ \pm {i}{ e^{ \ka (2\gg(z) -1)}} & 1 \end{bmatrix}e^{-\varkappa\gg(z)\sigma_3}, & z\in\mathcal{L}_E^{(\pm)}.
    \end{cases}
\end{align}
Suppose $z$ is in a compact subset of $\mathbb{C}\setminus(E\cup J)$.  We then construct the disks $\mathcal{D}_{a_j}$, $\mathcal{D}_{b_k}$ and lenses $\mathcal{L}_J^{(\pm)}$, $\mathcal{L}_E^{(\pm)}$ so that $z$ is outside of the disks and lenses.  Using the definition of $\wh{Z}(z;\varkappa)$ \eqref{hatZ explicit} and applying \eqref{E uni err}, we obtain the first result.  Now consider $z$ in a compact subset of $\mathring{E}\cup\mathring{J}$.  The disks $\mathcal{D}_{a_j}$, $\mathcal{D}_{b_k}$ are chosen so that $z$ lies outside the disks.  Again we turn to \eqref{hatZ explicit}, \eqref{E uni err} to obtain the second result.  Lastly, let $\delta>0$ be such that $(b_k,b_k+\delta)\subset\mathcal{D}_{b_k}$ for $k=1,\ldots,n$.  After applying \eqref{hatZ explicit}, \eqref{ParamIeLeft}, \eqref{E uni err}, we obtain the final result.
\end{proof}

We now consider the special scenario where $g=0$ and there are $n\geq1$ double points, i.e. $E\cup J=[a_1,a_2]$.  In this setting, the matrix $\Psi(z;\varkappa)$ can be expressed in terms of elementary functions.  Explicitly,
\begin{align}\label{Psi simple}
    \Psi(z;\ka)= e^{\tgg(\infty;\varkappa)\s_3}\le(\frac{z-a_1}{z-a_2}\ri)^{\frac{\s_1}4} e^{-\tgg(z;\varkappa)\s_3}.
\end{align}
The following result will be useful later.

\begin{lemma}\label{lemma: gamma no pole}
Let $J\cup E=[a_1,a_2]$ with $n\geq1$ double points, $z\in\mathring{J}\cup\mathring{E}$ and $|\tilde{\varkappa}|$ be sufficiently large so that the results of Theorem \ref{thm: gamma asymptotics} apply.  Then, for any $\lambda\in(-1,1)\setminus\{0\}$ such that $|\varkappa|>|\tilde\varkappa|$, $\Gamma(z;\lambda)$ is pole-free (in $\lambda$).
\end{lemma}

\begin{proof}
Assume there exists a $\frac{\lambda_0}{2}=e^{-\varkappa_0}$ ($\lambda_0>0$) with $\varkappa_0>|\tilde{\varkappa}|$ such that $\Gamma(z;\lambda_0)$ has a pole.  Then, according to \cite[eq. 5.1]{BKT20}, $\Gamma(z;\lambda)$ has the Laurant expansion around $\lambda=\lambda_0$ given by
\begin{align}\label{gamma pole}
    \Gamma(z;\lambda)=\frac{\Gamma^{0}(z)}{\lambda-\lambda_0}+\Gamma^1(z)+\BigO{\lambda-\lambda_0},
\end{align}
see \cite[Prop. 5.2]{BKT20} for more details on $\Gamma^0(z)$.  According to \eqref{gamma pole}, $\Gamma(z;\lambda)$ becomes unbounded as $\lambda\to\lambda_0$.  Since $\varkappa_0>|\tilde{\varkappa}|$, Theorem \ref{thm: gamma asymptotics} applies.  But after noticing that $\gg(z), \Psi(z;\varkappa), \tgg(z;\varkappa)$ are bounded (see \eqref{gg}, \eqref{tgg}, \eqref{Psi simple}), Theorem \ref{thm: gamma asymptotics} implies that $\Gamma(z;\lambda)$ is bounded as $\lambda\to\lambda_0$, a contradiction.  Thus, $\Gamma(z;\lambda)$ is pole-free for $\varkappa>|\tilde{\varkappa}|$ when $\lambda>0$.  For $\lambda<0$, we use the symmetry $\Gamma(z;\lambda)=\sigma_3\Gamma(z;-\lambda)\sigma_3$, see Remark \ref{rem-sym}, to obtain the result.
\end{proof}

\section{Asymptotics of the jump of the kernel of $\mathscr{R}$}\label{sec:jump}

We consider the case of just one interval $[a_1,a_2]=J\cup E$ with $n$ double points $b_1<b_2<\dots<b_n$ in $(a_1,a_2)$ and $[a_1,b_1]\subset E$.  Explicitly,
\begin{align*}
    E&=\begin{cases}
     [a_1,b_1]\cup[b_2,b_3]\cup\cdots\cup[b_n,a_2], &\text{$n$ even}, \\
     [a_1,b_1]\cup[b_2,b_3]\cup\cdots\cup[b_{n-1},b_n], &\text{$n$ odd},
    \end{cases}
\end{align*}
and $J=[a_1,a_2]\setminus(E\setminus\{b_1,\ldots,b_n\})$.  The objective of this Section is to obtain the large $\varkappa$ asymptotics of the kernels $Q_j$, see Theorem \ref{thm:main} items (3) and (4).  This goal is achieved in Proposition \ref{Eprime simple}.  The first step in proving this Proposition is to obtain the leading order asymptotics of the jump of the kernel of the resolvent, which was computed in terms of the matrix $\Gamma(z;\lambda)$ in Proposition \ref{prop: DeltaAbsLambdaRJRE ids}.  Recall first that the leading order asymptotics of $\Gamma(z;\lambda)$ was obtained in Theorem \ref{thm: gamma asymptotics} and second that since we are using Proposition \ref{prop: DeltaAbsLambdaRJRE ids} we only need to consider $\Im\varkappa\geq0$ (equivalently $\Im\lambda\leq0$).  Thus we assume throughout this Section that $\Im\varkappa\geq0$ and $\Im\varkappa=0$ is to be understood as the continuation onto the real-axis from the upper-half plane.  An important piece of the leading order asymptotics of $\Gamma(z;\lambda)$ in Theorem \ref{thm: gamma asymptotics} is the solution of the model RHP \ref{modelRHP} $\Psi(z;\varkappa)$.  Due to the simplicity of our intervals $J, E$, we can express $\Psi(z;\varkappa)$ in terms of elementary functions.  We begin by verifying that the solution of RHP \ref{modelRHPtilde} is $\tilde \Psi(z;\ka)=\tilde{\Psi}(z)=\le(\frac{z-a_1}{z-a_2}\ri)^{\frac{\s_1}{4}}$.  On the other hand, similarly to Section \ref{sect-modsol2}, we can write
\be\label{Psi_S}
    \Psi(z;\ka)=\Psi(z)= e^{\tgg(\infty)\s_3}\le(\frac{z-a_1}{z-a_2}\ri)^{\frac{\s_1}4} e^{-\tgg(z)\s_3}=S(z)\left(\frac{z-a_1}{(z-a_2)^{(-1)^n}}\right)^{\frac{\sigma_1}{4}}\tilde{B}_n(z)^{\sigma_1},
\ee 
where
\begin{align}\label{tilde Bn}
    \tilde{B}_n:\begin{cases}
     \mathbb{C}\setminus \left((-\infty,a_1)\cup E\right)\to\mathbb{C}, & n ~ \text{odd}, \\
     \mathbb{C}\setminus J\to\mathbb{C}, & n ~ \text{even},
    \end{cases} \hspace{0.75cm} \tilde{B}_n(z):=\prod_{j=1}^n\left(z-b_j\right)^\frac{(-1)^j}{2},
\end{align}
the branches for each root in $\tilde{B}_n(z)$ and $\left(\frac{z-a_1}{(z-a_2)^{(-1)^n}}\right)^{\frac{1}{4}}$ are $(-\infty,c_j)$ with positive orientation for each $c_j\in\{a_1,a_2,b_1,\ldots,b_n\}$ and
\be\label{S(z)}                             
S(z)=\1+\sum_{j=1}^n \frac{B_j}{z-b_j}= \1  +\sum_{j=1}^n (M_j+iN_j)\frac{\1 +(-1)^j\s_1}{z-b_j}.
\ee                                        
Here 
$M_j, N_j$ are real diagonal matrices and the latter equation follows from the requirement that $\Psi(z)=\BigO{(z-b_j)^{-\hf}}$ near any $b_j$.  We can now see that $\Psi(z)$ can be written as the sum of a Schwarz symmetric function and an anti-Schwarz symmetric function as
\begin{align}\label{Psi Sch + aSch}
    \Psi(z)=\wh{\Psi}(z)+i\breve{\Psi}(z),
\end{align}
where $\wh{\Psi}(z)$, $\breve{\Psi}(z)$ are the Schwarz symmetric functions
\begin{align}
    \wh{\Psi}(z)&:=\left(\1+\sum_{j=1}^n\frac{M_j}{z-b_j}(\1+(-1)^j\sigma_1)\right)\left(\frac{z-a_1}{(z-a_2)^{(-1)^n}}\right)^{\frac{\sigma_1}{4}}\tilde{B}_n(z)^{\sigma_1}, \label{wh Psi} \\
    \breve{\Psi}(z)&:=\sum_{j=1}^n\frac{N_j}{z-b_j}(\1+(-1)^j\sigma_1)\left(\frac{z-a_1}{(z-a_2)^{(-1)^n}}\right)^{\frac{\sigma_1}{4}}\tilde{B}_n(z)^{\sigma_1}. \label{breve Psi}
\end{align}
We are now prepared to extract the leading order asymptotics of $\Delta_{\lambda}R_{EE}$, $\Delta_{\lambda}R_{EJ}$, $\Delta_{\lambda}R_{JE}$, and $\Delta_{\lambda}R_{JJ}$.

\begin{proposition}\label{prop: Delta R asymp}
Let $\lambda\in [-1,1]\setminus\Sigma$, $M_j$ denote the $j$th column of the matrix $M$, and, with abuse of notation, $\varkappa=\varkappa(|\lambda|)$.  As $\lambda\to0$,
\begin{enumerate}
    \item if $x,y$ are in a compact subset of $\mathring{E}$, we have the uniform approximation 
    \begin{align}\label{Delta REE asymp}
        \Delta_{\lambda}R_{EE}(x,y;\lambda)=\tilde{R}_{EE}(x,y;\lambda)+\BigO{\varkappa^{-1}},
    \end{align}
    where
    \begin{multline}\label{tREE}
        \tilde{R}_{EE}(x,y;\lambda)= \\
        \frac{\begin{vmatrix} \Re[e^{-i\varkappa\Im\gg(y)}\wh{\Psi}_{1}(y)] & \Re[e^{-i\varkappa\Im\gg(x)}\breve{\Psi}_{1}(x)] \end{vmatrix}_+-\begin{vmatrix} \Re[e^{-i\varkappa\Im\gg(x)}\wh{\Psi}_{1}(x)] & \Re[e^{-i\varkappa\Im\gg(y)}\breve{\Psi}_{1}(y)] \end{vmatrix}_+}{\mathrm{sgn}(\lambda)\frac{i\pi}{4}(x-y)},
    \end{multline}
    and we understand that $x,y$ are taken on the upper shore of $\mathring{E}$.
    
    \item if $x,y$ are in a compact subset of $\mathring{E}, \mathring{J}$, respectively, we have the uniform approximation 
    \begin{align}\label{Delta REJ asymp}
        \Delta_{\lambda}R_{EJ}(x,y;\lambda)=\tilde{R}_{EJ}(x,y;\lambda)+\BigO{\varkappa^{-1}},
    \end{align}
    where
    \begin{multline}\label{tREJ}
        \tilde{R}_{EJ}(x,y;\lambda)= \\
        \frac{\begin{vmatrix} \Re[e^{i\varkappa\Im\gg(y)}\wh{\Psi}_{2}(y)] & \Re[e^{-i\varkappa\Im\gg(x)}\breve{\Psi}_{1}(x)] \end{vmatrix}_+-\begin{vmatrix} \Re[e^{-i\varkappa\Im\gg(x)}\wh{\Psi}_{1}(x)] & \Re[e^{i\varkappa\Im\gg(y)}\breve{\Psi}_{2}(y)] \end{vmatrix}_+}{\frac{i\pi}{4}(x-y)},
    \end{multline}
    and we understand that $x,y$ are taken on the upper shore of $\mathring{E}, \mathring{J}$, respectively.
    
    \item if $x,y$ are in a compact subset of $\mathring{J}, \mathring{E}$, respectively, we have the uniform approximation 
    \begin{align}\label{Delta RJE asymp}
        \Delta_{\lambda}R_{JE}(x,y;\lambda)=\tilde{R}_{JE}(x,y;\lambda)+\BigO{\varkappa^{-1}},
    \end{align}
    where
    \begin{multline}\label{tRJE}
        \tilde{R}_{JE}(x,y;\lambda)= \\
        \frac{\begin{vmatrix} \Re[e^{-i\varkappa\Im\gg(y)}\wh{\Psi}_{1}(y)] & \Re[e^{i\varkappa\Im\gg(x)}\breve{\Psi}_{2}(x)] \end{vmatrix}_+-\begin{vmatrix} \Re[e^{i\varkappa\Im\gg(x)}\wh{\Psi}_{2}(x)] & \Re[e^{-i\varkappa\Im\gg(y)}\breve{\Psi}_{1}(y)] \end{vmatrix}_+}{\frac{i\pi}{4}(x-y)},
    \end{multline}
    and we understand that $x,y$ are taken on the upper shore of $\mathring{J}, \mathring{E}$, respectively.
    
    \item if $x,y$ are in a compact subset of $\mathring{J}$, we have the uniform approximation
    \begin{align}\label{Delta RJJ asymp}
        \Delta_{\lambda}R_{JJ}(x,y;\lambda)=\tilde{R}_{JJ}(x,y;\lambda)+\BigO{\varkappa^{-1}},
    \end{align}
    where
    \begin{multline}\label{tRJJ}
        \tilde{R}_{JJ}(x,y;\lambda)= \\
        \frac{\begin{vmatrix} \Re[e^{i\varkappa\Im\gg(y)}\wh{\Psi}_{2}(y)] & \Re[e^{i\varkappa\Im\gg(x)}\breve{\Psi}_{2}(x)] \end{vmatrix}_+-\begin{vmatrix} \Re[e^{i\varkappa\Im\gg(x)}\wh{\Psi}_{2}(x)] & \Re[e^{i\varkappa\Im\gg(y)}\breve{\Psi}_{2}(y)] \end{vmatrix}_+}{\mathrm{sgn}(\lambda)\frac{i\pi}{4}(x-y)},
    \end{multline}
    and we understand that $x,y$ are taken on the upper shore of $\mathring{J}$.
\end{enumerate}
\end{proposition}

\begin{proof}
We will prove \eqref{Delta REJ asymp}, \eqref{tREJ} since the proof of the other identities similar.  Using \eqref{jump-num} and Proposition \ref{prop: DeltaAbsLambdaRJRE ids}, we have
\begin{align}\label{in proof: REJ}
    \Delta_{\lambda}R_{EJ}(x,y;\lambda)=\frac{2i\left(\det\left[\Re\Gamma_2(y;|\lambda|), \Im\Gamma_1(x;|\lambda|_-)\right]-\det\left[\Re\Gamma_1(x;|\lambda|), \Im\Gamma_2(y;|\lambda|_-)\right]\right)}{-\pi|\lambda|(x-y)}.
\end{align}
From Theorem \ref{thm: gamma asymptotics} and the jump properties of $\gg(z)$, $\Psi(z)$, we have (taking $\varkappa=\varkappa(|\lambda|)$)
\begin{align}\label{G2G1 asymp}
\begin{split}
    \Gamma_2(z;|\lambda|_-)&=e^{\varkappa\gg(\infty)\sigma_3}\left(\1+\BigO{\varkappa^{-1}}\right)\begin{bmatrix} e^{\varkappa\gg(z_+)}\Psi_{2+}(z)+e^{\varkappa\gg(z_-)}\Psi_{2-}(z) \end{bmatrix}, \qquad z\in\mathring{J}, \\
    \Gamma_1(z;|\lambda|_-)&=e^{\varkappa\gg(\infty)\sigma_3}\left(\1+\BigO{\varkappa^{-1}}\right)\begin{bmatrix} e^{-\varkappa\gg(z_+)}\Psi_{1+}(z) + e^{-\varkappa\gg(z_-)}\Psi_{1-}(z) \end{bmatrix}, \qquad z\in\mathring{E}.
\end{split}
\end{align}
Recall that if $f(z)$ is a Schwarz symmetric function with a jump on an interval $I\subset\mathbb{R}$, then $f_+(z)+f_-(z)=2\Re[f_+(z)]=2\Re[f_-(z)]$ for $z\in I$.  Using \eqref{Psi Sch + aSch}, we have
\begin{align}\label{ReIm parts}
    e^{(-1)^j\varkappa\gg(z_+)}\Psi_{j+}(z)+e^{(-1)^j\varkappa\gg(z_-)}\Psi_{j-}(z)=2\Re\left[e^{(-1)^j\varkappa\gg(z_+)}\wh{\Psi}_{j_+}(z)\right]+2i\Re\left[e^{(-1)^j\varkappa\gg(z_+)}\breve{\Psi}_{j+}(z)\right].
\end{align}
The results now follow by using \eqref{G2G1 asymp}, \eqref{ReIm parts} in \eqref{in proof: REJ} and using the fact that $\Re[\gg(z_\pm)]=\frac{1}{2}$ for $z\in\mathring{E}$ and $\Re[\gg(z_\pm)]=-\frac{1}{2}$ for $z\in \mathring{J}$, see Proposition \ref{prop: g-function}.
\end{proof}

To simplify the leading order kernels obtained in Proposition \ref{prop: Delta R asymp}, we need to compute the matrices $M_j, N_j$ appearing in \eqref{S(z)}.  This is achieved by understanding the behavior of $\tgg(z)$ as $z\to b_j$, $j=1,\ldots,n$ and as $z\to\infty$.  Recall from \eqref{tgg} that $\tgg(z)$ (for $\Im\varkappa\geq0$) has the integral representation
\begin{align*}
    \tgg(z)=\frac{R(z)}{2\pi i}\int_{J}\frac{ -i\pi}{(\z-z)R_+(\z)}d\z.
\end{align*}
We can evaluate this integral explicitly.  Define the function $s:\mathbb{C}\setminus(-\infty,a_2]\times(a_1,a_2)\to\mathbb{C}$ as
\begin{align}\label{szb}
    s(z,b):=i\sqrt{(z-a_1)(a_2-b)}+\sqrt{(z-a_2)(b-a_1)},
\end{align}
where $\arg(z-a_j)=0$ for $z-a_j>0$.  For $z\in(-\infty,a_2)$,
\begin{align*}
    s(z_+,b)=\begin{cases}
    \frac{(b-z)(a_2-a_1)}{s(z_-,b)}, & z\in(a_1,a_2), \\
    -s(z_-,b), & z\in(-\infty,a_1).
    \end{cases}
\end{align*}
We also verify that $\tilde{B}_n(z)$, defined in \eqref{tilde Bn}, has the jumps
\begin{align*}
    \tilde{B}_{n}(z_+)=-\tilde{B}_{n}(z_-), \qquad z\in\begin{cases}
     (-\infty,a_1)\cup \mathring{E}, & n ~ \text{odd}, \\
     \mathring{J}, & n ~ \text{even}.
    \end{cases}
\end{align*}
From the jump conditions of $s(z,b)$ and $\tilde{B}_n(z)$, we have the identity
\begin{align}\label{tgg explicit}
    \tilde{\gg}(z)=\ln\left((a_2-a_1)^{\frac{(-1)^n-1}{4}}\tilde{B}_n(z)\prod_{k=1}^{n} s(z,b_k)^{(-1)^{k+1}}\right),
\end{align}
where the logarithm $\ln(\cdot)$ has the branch $(-\infty,0)$ with positive orientation, because the right hand side of \eqref{tgg explicit} solves the RHP for $\tgg(z)$, see \eqref{tggJumps1} and the surrounding text.  Alternatively, \eqref{tgg explicit} can be obtained by applying \cite[eq. 2.266]{GRtable} to \eqref{tgg}.  It is convenient to introduce the angles $\nu_k$, $k=1,\ldots,n$, defined as the acute angle adjacent to the side $\sqrt{b_k-a_1}$ in the right triangle with hypotenuse $\sqrt{a_2-a_1}$ and sides $\sqrt{b_k-a_1}, \sqrt{a_2-b_k}$, see Figure \ref{fig: nu angle}.
\begin{figure}
\begin{center}
\begin{tikzpicture}[scale=1.00]



\draw (-2,0) -- (2,0);
\draw (-2,0) -- (2,2);
\draw (2,0) -- (2,2);

\node[below] at (0,0) {$\sqrt{b_k-a_1}$};
\node[above] at (0,1.3) {$\sqrt{a_2-a_1}$};
\node at (2.8,1) {$\sqrt{a_2-b_k}$};
\node at (-1.1,0.175) {$\nu_k$};

\draw[thick] (-1.4,0) arc (0:28:0.6);

\end{tikzpicture}
\end{center}
\vspace{0pt}
\caption{The angle $\nu_k$.} \label{fig: nu angle}
\vspace{-0.0cm}
\end{figure}
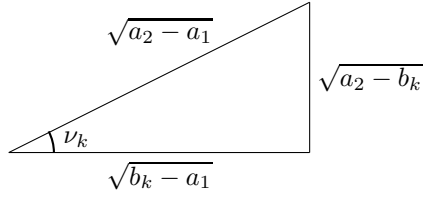
With this notation, we can see that for $z\in(a_1,a_2)$,
\begin{align*}
    s(z_\pm,b_k)=i(a_2-a_1)\sin(\nu_k\pm\nu(z)),
\end{align*}
where $\nu(z)=\sin^{-1}\left(\frac{\sqrt{a_2-z}}{\sqrt{a_2-a_1}}\right)$ and $\nu(b_k)=\nu_k$.  As $z\to b_j$ on the upper shore and interior of $J$,
\begin{align}
    &\tilde{B}_n(z)=i^{\frac{(-1)^n+1}{2}}\hat{k}_j|z-b_j|^{\frac{(-1)^j}{2}}+\BigO{(z-b_j)^{\frac{(-1)^j}{2}+1}}, \qquad \hat{k}_j:=\prod_{\substack{s=1 \\ s\neq j}}^n|b_j-b_s|^{\frac{(-1)^s}{2}}, \label{hat kj} \\
    &\prod_{k=1}^{n} s(z,b_k)^{(-1)^{k+1}}=(i(a_2-a_1))^\frac{1-(-1)^n}{2}\hat{s}_j+o(1), \qquad \hat{s}_j:=\prod_{k=1}^{n} \sin(\nu_j+\nu_k)^{(-1)^{k+1}}, \label{hat sj}
\end{align}
and thus
\begin{align}\label{etgg}
    e^{\tilde{\gg}(z)}=i(a_2-a_1)^\frac{1-(-1)^n}{4}\hat{s}_j\hat{k}_j|z-b_j|^{\frac{(-1)^j}{2}}+\BigO{(z-b_j)^{\frac{(-1)^j}{2}+1}}.
\end{align}
To understand the behavior of $\tgg(z)$ as $z\to\infty$, we first notice that as $z\to+\infty$,
\begin{align*}
    \frac{s(z,b_k)}{\sqrt{z-b_k}}=\sqrt{a_2-a_1}e^{i\nu_k}\left(1+\frac{i(a_2-a_1)\sin(2\nu_k)}{4z}+\BigO{z^{-2}}\right),
\end{align*}
for $k=1,\ldots,n$.  We see that 
\begin{align}\label{tgg inf}
    \tgg(\infty)=i\sum_{k=1}^n(-1)^{k+1}\nu_k,
\end{align}
and as $z\to\infty$,
\begin{align}
    e^{\tilde{\gg}(z)}=e^{\tilde{\gg}(\infty)}\left(1+\frac{i(a_2-a_1)}{4z}\sum_{k=1}^n(-1)^{k+1}\sin(2\nu_k)+\BigO{z^{-2}}\right).
\end{align}
We now state a preparatory Lemma.

\bl\label{lem-s1}
For any $M,N\in \C\setminus\{0\}$ and $a\in\C$ we have
\be
M^{\s_1}e^{a\s_3}N^{\s_1}=(MN)^{\s_1}\cosh a +\s_3\left(\frac NM\right)^{\s_1}\sinh a.
\ee
\el

\begin{proof}
Using the identities
\begin{align*}
    \beta^{\sigma_1}=\frac{1}{2}\begin{bmatrix} \beta+\beta^{-1} & \beta-\beta^{-1} \\ \beta-\beta^{-1} & \beta+\beta^{-1} \end{bmatrix}=\cosh(\ln\beta)\1+\sinh(\ln\beta)\sigma_1, \qquad \beta\in\mathbb{C}\setminus\{0\}
\end{align*}
and
\begin{align*}
    2\cosh(a)\cosh(b)&=\cosh(a+b)+\cosh(a-b), \\
    2\sinh(a)\sinh(b)&=\cosh(a+b)-\cosh(a-b), \\
    2\cosh(a)\sinh(b)&=\sinh(a+b)-\sinh(a-b),
\end{align*}
the Lemma follows from a straightforward calculation.
\end{proof}

\bl\label{lem-resM} 
The matrices $M_j, N_j$ of \eqref{S(z)} have the form $M_{j}=m_j\1$ and $N_{j}=n_j\s_3$ for $j=1,\ldots,n$, where
\be\label{m_n}
m_j=k_j\begin{cases}
-\sin(\nu_j-\alpha), & j ~ \text{odd}, \\
\cos(\nu_j+\a), & j ~ \text{even},
\end{cases} \hspace{1cm}
n_j=k_j\begin{cases}
-\cos(\nu_j-\alpha), & j ~ \text{odd}, \\
\sin(\nu_j+\a), & j ~ \text{even},
\end{cases}
\ee
$\a=\Im\tgg(\infty)$ and 
\begin{align}\label{k_j}
    k_j=\begin{cases}
    \frac{1}{2}(a_2-a_1)\cos(\nu_j)(\sin\nu_j)^{\frac{1-(-1)^n}{2}}\prod_{\substack{l=1 \\ l\neq j}}^n\left(\sin|\nu_l-\nu_j|\right)^{(-1)^{l}}, & j ~ \text{odd}, \\
    \frac{1}{2}(a_2-a_1)(\sin\nu_j)^{\frac{1+(-1)^n}{2}}\prod_{\substack{l=1 \\ l\neq j}}^n\left(\sin|\nu_l-\nu_j|\right)^{(-1)^{l+1}}, & j ~ \text{even}.
 \end{cases}
\end{align}
\el
                                                                               
\begin{proof}
According to \eqref{Psi_S},
\begin{align}\label{M0}
S(z)= e^{\tgg(\infty)\s_3}A^{\frac{\s_1}{4}}e^{-\tgg(z)\s_3}B^{-\frac{\s_1}{4}},
\end{align}
where (recall $\tilde{B}_n(z)$ was defined in \eqref{tilde Bn})
\begin{align*}
    A=A(z)=\frac{z-a_1}{z-a_2}, ~~~ B=B_n(z)=\frac{(z-a_1)}{(z-a_2)^{(-1)^n}}\tilde{B}_n(z)^4.
\end{align*}
Then, using Lemma \ref{lem-s1}, we have
\be\label{M1}
M(z)= e^{\tgg(\infty)\s_3}\le[\le(\frac{A}{B}\ri)^{\frac{\s_1}{4}}\cosh \tgg(z)-\s_3\le(\frac{1}{AB}\ri)^{\frac{\s_1}{4}}\sinh \tgg(z)\ri].
\ee
Using \eqref{etgg} we obtain
\be\label{chsh}
\begin{split}
\cosh \tgg(z_+)= \frac{i}{2}|z-b_j|^{-\frac{1}{2}}(-1)^{j+1}(\hat{k}_j\hat{s}_j(a_2-a_1))^\frac{1-(-1)^n}{4})^{(-1)^{j+1}}+\BigO{1},\cr
\sinh \tgg(z_+)=\frac{i}{2}|z-b_j|^{-\frac{1}{2}}(\hat{k}_j\hat{s}_j(a_2-a_1))^\frac{1-(-1)^n}{4})^{(-1)^{j+1}}+\BigO{1},
 \end{split}
\ee
as $z\ra b_{j}$ on the upper shore and interior of $J$, $j=1,\dots,n$.  It is straightforward to verify that $\arg \le(\frac AB\ri)^\qt =-\frac \pi 2$ and
$\arg \le(\frac 1{AB}\ri)^\qt =0$ on the upper shore and interior of $J$ so that
\be\label{A_B}
\le(\frac AB\ri)^\qt =-i\le|\frac AB\ri|^\qt,~~~~
\le(\frac 1{AB}\ri)^\qt=\le|\frac 1{AB}\ri|^\qt.
\ee
Moreover, according to \eqref{M0}, \eqref{A_B}, 
\bea\label{A_B1}
\begin{split}
&&2\le(\frac AB\ri)^{\frac{\s_1}{4}} =(-1)^{j+1}i\left(\hat{k}_j\right)^{(-1)^{j+1}}\tilde{A}_{j,1}(\1+(-1)^j\s_1)|z-b_{j}|^{-\hf}+\BigO{1},\cr
&&2\le(\frac 1{AB}\ri)^{\frac{\s_1}{4}}=\left(\hat{k}_j\right)^{(-1)^{j+1}}\tilde{A}_{j,2}(\1+(-1)^j\s_1)|z-b_{j}|^{-\hf}+\BigO{1},
\end{split}
\eea
where constants $\hat{k}_j>0$ are defined in \eqref{hat kj} and
\begin{align}
    \tilde{A}_{j,1}&=\begin{cases}
    1, & n ~ \text{even}, \\
    (a_2-b_j)^\frac{(-1)^{j+1}}{2}, & n ~ \text{odd},
    \end{cases} \\
    \tilde{A}_{j,2}&=\begin{cases}
    \left(\frac{b_j-a_1}{a_2-b_j}\right)^\frac{(-1)^j}{2}, & n ~ \text{even}, \\
    (b_j-a_1)^\frac{(-1)^{j+1}}{2}, & n ~ \text{odd}.
    \end{cases}
\end{align}
Substituting \eqref{chsh}-\eqref{A_B1} into \eqref{M1}, letting $\alpha=\Im[\tgg(\infty)]$, and comparing the $\BigO{(z-b_j)^{-1}}$ term (leading order obtained by substituting \eqref{chsh}-\eqref{A_B1} into \eqref{M1} is of the form $\frac{c}{|z-b_j|}$, be careful with absolute value) with that of 
\eqref{S(z)}, we obtain 
\begin{align*}
M_j+iN_j&=\frac{(-1)^j}{4}(\hat{k}_j^2\hat{s}_j(a_2-a_1)^\frac{1-(-1)^n}{4})^{(-1)^{j+1}}e^{\tgg(\infty)\sigma_3}\le(\1\tilde{A}_{j,1} + i\s_3\tilde{A}_{j,2}\ri) \\
&=\frac{(-1)^j}{4}(\hat{k}_j^2\hat{s}_j(a_2-a_1)^\frac{1-(-1)^n}{4})^{(-1)^{j+1}}\le[\le(\tilde{A}_{j,1}\cos\a-\tilde{A}_{j,2}\sin\a\ri)\1+\le(\tilde{A}_{j,2}\cos\a +\tilde{A}_{j,1}\sin\a\ri)i\s_3\ri].
\end{align*}
Using the fact that $\sqrt{b_j-a_1}=\sqrt{a_2-a_1}\cos(\nu_j)$ and $\sqrt{a_2-b_j}=\sqrt{a_2-a_1}\sin(\nu_j)$, we have
\begin{align}
    D_j=k_j\begin{cases}
    -\sin(\nu_j-\alpha)\1-\cos(\nu_j-\alpha)i\sigma_3, & j ~ \text{odd}, \\
    \cos(\nu_j+\a)\1+\sin(\nu_j+\a)i\s_3, & j ~ \text{even},
    \end{cases}
\end{align}
where
\be
k_j=\begin{cases}
    \frac{\hat{k}_j^2\hat{s}_j}{4}(a_2-a_1)^\frac{1-(-1)^n}{2}(\sin\nu_j)^\frac{1+(-1)^n}{-2}, & j ~ \text{odd}, \\
    \frac{\sec\nu_j}{4\hat{k}_j^2\hat{s}_j}(a_2-a_1)^\frac{(-1)^n-1}{2}(\sin\nu_j)^\frac{(-1)^n-1}{2}, & j ~ \text{even}.
\end{cases}
 %
 \ee
These  equations imply \eqref{m_n}.  The relation $b_j-b_l=(a_2-a_1)\sin(\nu_l+\nu_j)\sin(\nu_l-\nu_k)$ is used to complete the proof.
\end{proof}

We now show that $\tilde{R}_{JJ}, \tilde{R}_{JE}, \tilde{R}_{EJ}, \tilde{R}_{EE}$ can be expressed as quadratic forms.  Let $N\in\mathbb{N}$ be such that $n=2N$ if $n$ is even and $n=2N+1$ if $n$ is odd and 
\begin{align*}
    \tilde{N}:=\begin{cases}
     N, & n ~ \text{even}, \\
     N+1, & n ~ \text{odd}.
    \end{cases}
\end{align*}
We define the $n\times n$ block diagonal matrix
\begin{align}\label{M-def}
    \mathbb{M}:=\begin{bmatrix} \mathbb{M}_1 & 0 \\ 0 & \mathbb{M}_2 \end{bmatrix},
\end{align}
where $\mathbb{M}_1, \mathbb{M}_2$ are symmetric blocks of size $\tilde{N}\times\tilde{N}$, $N\times N$, respectively, with diagonal entries
\begin{align}
    \left[\mathbb{M}_1\right]_{ll}&=k_{2l-1}\cos(\nu_{2l-1}-\alpha)-2k_{2l-1}\sum_{\substack{j=1 \\ j\neq l}}^{\tilde{N}}\frac{k_{2j-1}\sin(\nu_{2l-1}-\nu_{2j-1})}{b_{2j-1}-b_{2l-1}}, \label{M1 diag} \\
    \left[\mathbb{M}_2\right]_{ll}&=k_{2l}\sin(\nu_{2l}+\alpha)-2k_{2l}\sum_{\substack{j=1 \\ j\neq l}}^{N}\frac{k_{2j}\sin(\nu_{2l}-\nu_{2j})}{b_{2j}-b_{2l}}, \label{M2 diag}
\end{align}
and off-diagonal entries
\begin{align}\label{offdiag-ent}
    \left[\mathbb{M}_1\right]_{jl}=\frac{2k_{2j-1}k_{2l-1}\sin(\nu_{2j-1}-\nu_{2l-1})}{b_{2l-1}-b_{2j-1}}, \qquad \left[\mathbb{M}_2\right]_{jl}=\frac{2k_{2j}k_{2l}\sin(\nu_{2j}-\nu_{2l})}{b_{2l}-b_{2j}},
\end{align}
where $l\neq j$.  We show that $\mathbb{M}$ is positive definite in Proposition \ref{prop: M pos def}.

\bt\label{theo-dim_n}
Let $\lambda\in[-1,1]\setminus\{0\}$ and, with abuse of notation, $\varkappa=\varkappa(|\lambda|)$.  We define $\vec{f}, \vec{h}:(\mathring{J}\cup\mathring{E})\times[-1,1]\setminus\{0\}\to\mathbb{R}^n$ as
\begin{align}
    \vec{f}^{\hspace{0.4mm}t}(x;\varkappa)&=\begin{bmatrix} \frac{\phi_-(x;\varkappa)}{x-b_1} & \cdots & \frac{\phi_-(x;\varkappa)}{x-b_{2\tilde{N}-1}} & \frac{\phi_+(x;\varkappa)}{x-b_2} & \cdots &  \frac{\phi_+(x;\varkappa)}{x-b_{2N}} \end{bmatrix}, \\
    \vec{h}^{\hspace{0.4mm}t}(x;\varkappa)&=\begin{bmatrix} \frac{\tilde{\phi}_-(x;\varkappa)}{x-b_1} & \cdots & \frac{\tilde{\phi}_-(x;\varkappa)}{x-b_{2\tilde{N}-1}} & \frac{\tilde{\phi}_+(x;\varkappa)}{x-b_2} & \cdots &  \frac{\tilde{\phi}_+(x;\varkappa)}{x-b_{2N}} \end{bmatrix},
\end{align}
where
\begin{align}\label{phi pm}
    \phi_{\pm}(x;\varkappa)=\Re\le(e^{i\ka \Im\gg(x_+)}B_{n}^{\pm\qt}(x_+)\ri), ~~~ \tilde{\phi}_\pm(x;\varkappa)=\pm\Re\left(e^{-i\varkappa\Im\gg(x_+)}B_{n}^{\pm\frac{1}{4}}(x_+)\right).
\end{align}
We mention that $\pm$ occurring in $\phi_\pm$, $\tilde{\phi}_\pm$ is convenient notation and should not be interpreted as a boundary value.
\begin{enumerate}
    \item For $x,y\in \mathring{E}$,
    \begin{align}\label{REE quad-form}
        \tilde{R}_{EE}(x,y;\lambda)=\mathrm{sgn}(\lambda)\frac{4}{i\pi}\vec{h}^{\hspace{0.4mm}t}(y;\varkappa)\mathbb{M}\vec{h}(x;\varkappa).
    \end{align}

    \item For $x\in\mathring{E}$ and $y\in\mathring{J}$,
    \begin{align}\label{REJ quad-form}
        \tilde{R}_{EJ}(x,y;\lambda)=\frac{4}{i\pi}\vec{h}^{\hspace{0.4mm}t}(y;\varkappa)\mathbb{M}\vec{f}(x;\varkappa).
    \end{align}
    
    \item For $x\in\mathring{J}$ and $y\in\mathring{E}$,
    \begin{align}\label{RJE quad-form}
        \tilde{R}_{JE}(x,y;\lambda)=\frac{4}{i\pi}\vec{f}^{\hspace{0.4mm}t}(y;\varkappa)\mathbb{M}\vec{h}(x;\varkappa).
    \end{align}
    
    \item For $x,y\in \mathring{J}$,
    \begin{equation}\label{RJJ quad-form}
        \tilde{R}_{JJ}(x,y;\lambda) = \mathrm{sgn}(\lambda)\frac{4}{i\pi}\vec{f}^{\hspace{0.4mm}t}(y;\varkappa)\mathbb{M}\vec{f}(x;\varkappa).
    \end{equation}
\end{enumerate}
\et

\begin{proof}
We assume all multi-valued functions are evaluated on the upper shore of $\mathring{J}\cup \mathring{E}$.  We prove this Theorem by evaluating the determinants in \eqref{tREE}, \eqref{tREJ}, \eqref{tRJE}, \eqref{tRJJ}.  Taking $x,y\in \mathring{J}$ and using \eqref{wh Psi}, \eqref{breve Psi} we have
\begin{align}
    \Re[e^{i\varkappa\Im\gg(y)}\wh{\Psi}_{2}(y)]&=\frac{1}{2}\phi_+(y;\varkappa)\begin{bmatrix} 1 \\ 1 \end{bmatrix}+\frac{1}{2}\phi_-(y;\varkappa)\begin{bmatrix} -1 \\ 1 \end{bmatrix}+\sum_{j=1}^n\frac{m_j}{y-b_j}\phi_{(-1)^j}(y;\varkappa)\begin{bmatrix} (-1)^j \\ 1 \end{bmatrix} \label{Re hat Psi2} \\
    \Re[e^{i\varkappa\Im\gg(x)}\breve{\Psi}_{2}(x)]&=\sum_{j=1}^n\frac{n_j}{x-b_j}\phi_{(-1)^j}(x;\varkappa)\begin{bmatrix} (-1)^j \\ -1 \end{bmatrix}, \label{Re tilde Psi2}
\end{align}
where $\phi_{(-1)^j}=\phi_+$ for $j$ even and $\phi_{(-1)^j}=\phi_-$ for $j$ odd.  Now evaluating the first determinant in \eqref{tRJJ}, we find
\begin{multline} \label{1st_det}
\begin{vmatrix} \Re[e^{i\varkappa\Im\gg(y)}\wh{\Psi}_{2}(y)] & \Re[e^{i\varkappa\Im\gg(x)}\breve{\Psi}_{2}(x)] \end{vmatrix}=\sum_{j=1}^n\frac{n_j\phi_{(-1)^j}(y;\ka)\phi_{(-1)^j}(x;\ka)}{(-1)^{j+1}(x-b_j)} \\
+2\sum_{\substack{j,l=1 \\ j=l\mod2}}^{n}\frac{m_jn_l\phi_{(-1)^j}(y;\ka)\phi_{(-1)^j}(x;\ka)}{(-1)^{j+1}(y-b_j)(x-b_l)}.
\end{multline}
Using \eqref{1st_det}, the numerator of \eqref{tRJJ} is
\begin{multline}\label{det-double}
\begin{vmatrix} \Re[e^{i\varkappa\Im\gg(y)}\wh{\Psi}_{2}(y)] & \Re[e^{i\varkappa\Im\gg(x)}\breve{\Psi}_{2}(x)] \end{vmatrix}-\begin{vmatrix} \Re[e^{i\varkappa\Im\gg(x)}\wh{\Psi}_{2}(x)] & \Re[e^{i\varkappa\Im\gg(y)}\breve{\Psi}_{2}(y)] \end{vmatrix}= \\
(x-y)\Bigg(\sum_{j=1}^n\frac{n_j\phi_{(-1)^j}(y;\ka)\phi_{(-1)^j}(x;\ka)}{(-1)^{j}(y-b_j)(x-b_j)}-2\sum_{\substack{j,l=1,j<l \\ j=l {\mod 2}}}^n k_jk_l\frac{\sin(\nu_j-\nu_l)(b_l-b_j)\phi_{(-1)^j}(y;\ka)\phi_{(-1)^j}(x;\ka)}{(y-b_j)(y-b_l)(x-b_j)(x-b_l)}\Bigg),
\end{multline}
where we have used $n_{l}m_j-n_jm_{l}=(-1)^{j+1}k_jk_l\sin(\nu_j-\nu_l)$, where $l>j$.  Now using
\begin{align*}
    \frac{b_l-b_j}{(y-b_j)(x-b_l)(y-b_l)(x-b_j)}=\frac{1}{b_l-b_j}\left(\frac{1}{y-b_j}-\frac{1}{y-b_l}\right)\left(\frac{1}{x-b_j}-\frac{1}{x-b_l}\right),
\end{align*}
we see that \eqref{det-double} is equivalent to
\begin{multline}\label{det-single}
    (x-y)\Bigg[\sum_{j=1}^n\frac{\phi_{(-1)^j}(y;\ka)\phi_{(-1)^j}(x;\ka)}{(y-b_j)(x-b_j)}\Bigg((-1)^{j}n_j-2k_j\sum_{\substack{l=1,j\neq l \\ l=j\mod2}}^{n}\frac{k_l\sin(\nu_j-\nu_l)}{b_l-b_j}\Bigg) \\
    +2\sum_{\substack{j,l=1, j<l \\ j=l\mod 2}}^n k_jk_l\frac{\sin(\nu_j-\nu_l)\phi_{(-1)^j}(y;\ka)\phi_{(-1)^j}(x;\ka)}{b_l-b_j}\left(\frac{1}{(y-b_j)(x-b_l)}+\frac{1}{(y-b_l)(x-b_j)}\right)\Bigg].
\end{multline}
The statement \eqref{RJJ quad-form} follows from \eqref{m_n}, \eqref{tRJJ}, \eqref{det-single}.  Now taking $x,y\in\mathring{E}$ and evaluating the determinant in \eqref{tREE} using \eqref{wh Psi}, \eqref{breve Psi}, we see that
\begin{align}
    \Re[e^{-i\varkappa\Im\gg(y)}\hat{\Psi}_1(y)]&=\frac{1}{2}\tilde{\phi}_+(y;\varkappa)\begin{bmatrix} 1 \\ 1 \end{bmatrix}+\frac{1}{2}\tilde{\phi}_-(y;\varkappa)\begin{bmatrix} -1 \\ 1 \end{bmatrix}+\sum_{j=1}^n\frac{m_j\tilde{\phi}_{(-1)^j}(y;\varkappa)}{y-b_j}\begin{bmatrix} (-1)^j \\ 1 \end{bmatrix}, \label{Re hat Psi1} \\
    \Re[e^{-i\varkappa\Im\gg(x)}\breve{\Psi}_1(x)]&=\sum_{j=1}^n\frac{n_j\tilde{\phi}_{(-1)^j}(x;\varkappa)}{x-b_j}\begin{bmatrix} (-1)^j \\ -1 \end{bmatrix}, \label{Re tilde Psi1}
\end{align}
where $\tilde{\phi}_{(-1)^j}=\tilde{\phi}_+$ for $j$ even and $\tilde{\phi}_{(-1)^j}=\tilde{\phi}_-$ for $j$ odd.  It is clear that \eqref{Re hat Psi1}, \eqref{Re tilde Psi1} are of the same form as \eqref{Re hat Psi2}, \eqref{Re tilde Psi2} and thus the result \eqref{RJJ quad-form} follows immediately.  The results for $\tilde{R}_{EJ}$ and $\tilde{R}_{JE}$ are obtained similarly.
\end{proof}

\begin{remark}\label{rem: phi tphi}
To compute $\phi_\pm,\tilde{\phi}_\pm$ more explicitly, we first verify
\begin{align}\label{arg Bn}
    \mathrm{arg}(B_{n}^{\pm\frac{1}{4}}(x_+))=\pm\begin{cases}
     \frac{\pi}{4}, & x\in \mathring{J}, \\
     -\frac{\pi}{4}, & x\in \mathring{E}.
    \end{cases}
\end{align}
Using that fact, we have
\begin{align*}
    \phi_{\pm}(x;\ka)&=\Re\le(e^{i\ka \Im\gg(x_+)}B_{n}^{\pm\qt}(x_+)\ri)=|B_{n}^{\pm\frac{1}{4}}(x_+)|\cos\left(\varkappa\Im\gg(x_+)\pm\frac{\pi}{4}\right), & & x\in \mathring{J}, \\
    \tilde{\phi}_\pm(x;\varkappa)&=\pm\Re\left(e^{-i\varkappa\Im\gg(x_+)}B_{n}^{\pm\frac{1}{4}}(x_+)\right)=\pm|B_{n}^{\pm\frac{1}{4}}(x_+)|\cos\left(\varkappa\Im\gg(x_+)\pm\frac{\pi}{4}\right), & & x\in\mathring{E}.
\end{align*}
\end{remark}

Notice that $\vec{f}, \vec{h}$ can be decomposed as
\begin{align}\label{fh decomp}
\begin{split}
    \vec{f}(x;\varkappa)&=\phi_-(x;\varkappa)\begin{bmatrix} \vec{v}_-(x) \\ \vec{0} \end{bmatrix}+\phi_+(x;\varkappa)\begin{bmatrix} \vec{0} \\ \vec{v}_+(x) \end{bmatrix}, \\
    \vec{h}(x;\varkappa)&=\tilde{\phi}_-(x;\varkappa)\begin{bmatrix} \vec{v}_-(x) \\ \vec{0} \end{bmatrix}+\tilde{\phi}_+(x;\varkappa)\begin{bmatrix} \vec{0} \\ \vec{v}_+(x) \end{bmatrix},
\end{split}
\end{align}
where $\vec{0}$ denotes the (appropriately sized) vector of all $0$ and
\begin{align}
    \vec{v}_-(x):=\begin{bmatrix} \frac{1}{x-b_1} \\ \vdots \\ \frac{1}{x-b_{2\tilde{N}-1}} \end{bmatrix}, ~~~ \vec{v}_+(x):=\begin{bmatrix} \frac{1}{x-b_2} \\ \vdots \\ \frac{1}{x-b_{2N}} \end{bmatrix}.
\end{align}
Recall from \eqref{M-def} and Proposition \ref{prop: M pos def} that $\mathbb{M}$ is a block diagonal matrix with positive definite blocks $\mathbb{M}_1$, $\mathbb{M}_2$.  Let $C_-$, $C_+$, and $C$ denote the upper triangular Cholesky factor {with a positive diagonal} of $\mathbb{M}_1$, $\mathbb{M}_2$, and $\mathbb{M}$, respectively.  Then,
\begin{align}\label{M Choleksy}
    \mathbb{M}=\begin{bmatrix} C_-^tC_- & 0 \\ 0 & C_+^tC_+ \end{bmatrix}=\begin{bmatrix} C_- & 0 \\ 0 & C_+ \end{bmatrix}^t\begin{bmatrix} C_- & 0 \\ 0 & C_+ \end{bmatrix}=:C^tC,
\end{align}
where $C$ is the Cholesky factor of $\mathbb{M}$.  It is to be noted that the upper triangular Cholesky factor {with a positive diagonal} of a positive definite matrix is unique {\cite[Lemma~12.1.6]{MSch02}}.  We conclude this Section with the following Proposition.

\begin{proposition}\label{Eprime simple}
Suppose $\lambda\in(-1,1)\setminus\{0\}$, $x,y$ are in a compact subset of $\mathring{J}\cup\mathring{E}$ and, with abuse of notation, $\varkappa=\varkappa(|\lambda|)$.  Then, as $\lambda\to0$, we have the identity
\begin{align}\label{Eprime approx}
    |\l|E_{ac}'(x,y;\lambda)=\sum_{j=1}^{n}G_j(x;\lambda)G_j(y;\lambda)+\BigO{\varkappa^{-1}},
\end{align}
where $E_{ac}'(x,y;\lambda)$ is the kernel of $\frac{d}{d\lambda}\mathcal{E}_{ac}[f](x;\lambda)$, see \eqref{complete res of id} and the text that follows it,
\begin{align}
    G_j(x;\lambda):=\chi_E(x)\hat{h}_j(x;\lambda)+\chi_J(x)\hat{f}_j(x;\lambda),
\end{align}
and
\begin{align}
    \hat{h}_j(x;\lambda)&=A_j(x)\cos\left(\varkappa\Im\gg(x_+)+s(j)\frac{\pi}{4}\right)s(j), \label{h prop} \\
    \hat{f}_j(x;\lambda)&=A_j(x)\cos\left(\varkappa\Im\gg(x_+)+s(j)\frac{\pi}{4}\right)\mathrm{sgn}(\lambda), \label{f prop}
\end{align}
with $s(j)=-1$ for $j=1,\ldots,\tilde{N}$, $s(j)=1$ for $j=\tilde{N}+1,\ldots,n$,
\begin{align}\label{A prop}
    A_j(x)=\frac{\sqrt{2}}{\pi}\begin{cases}
    |B_{n}^{-\frac{1}{4}}(x_+)|[C_-\vec{v}_-(x)]_j, & j=1,\ldots,\tilde{N}, \\
    |B_{n}^{\frac{1}{4}}(x_+)|[C_+\vec{v}_+(x)]_{j-\tilde{N}}, & j=\tilde{N}+1,\ldots,n.
    \end{cases}
\end{align}
\end{proposition}

\begin{proof}
From \eqref{deriv res of id kernel}, Proposition \ref{prop: Delta R asymp} and Theorem \ref{theo-dim_n} we find that
\begin{multline}
    |\lambda|E_{ac}'(x,y;\lambda)=\frac{2}{\pi^2}\left(\vec{h}^{\hspace{0.4mm}t}(y;\varkappa)\mathbb{M}\vec{h}(x;\varkappa)\chi_E(x)\chi_E(y)+\mathrm{sgn}(\lambda)\vec{h}^{\hspace{0.4mm}t}(y;\varkappa)\mathbb{M}\vec{f}(x;\varkappa)\chi_E(y)\chi_J(x) \right. \\
    \left.+\mathrm{sgn}(\lambda)\vec{f}^{\hspace{0.4mm}t}(y;\varkappa)\mathbb{M}\vec{h}(x;\varkappa)\chi_J(y)\chi_E(x)+\vec{f}^{\hspace{0.4mm}t}(y;\varkappa)\mathbb{M}\vec{f}(x;\varkappa)\chi_J(x)\chi_J(y)\right)+\BigO{\varkappa^{-1}}.
\end{multline}
Using the Cholesky decomposition $\mathbb{M}=C^tC$ and letting
\begin{align}\label{hf def}
    \hat{h}_j(x;\lambda):=\frac{\sqrt{2}}{\pi}[C\vec{h}(x;\varkappa)]_j, \qquad \hat{f}_j(x;\lambda):=\mathrm{sgn}(\lambda)\frac{\sqrt{2}}{\pi}[C\vec{f}(x;\varkappa)]_j,
\end{align}
we have
\begin{align}\label{Eac h f}
    |\l|E_{ac}'(x,y;\lambda)=\sum_{j=1}^n(\chi_E(x)\hat{h}_j(x;\lambda)+\chi_J(x)\hat{f}_j(x))(\chi_E(y)\hat{h}_j(y;\lambda)+\chi_J(y)\hat{f}_j(y))+\BigO{\varkappa^{-1}}.
\end{align}
Remark \ref{rem: phi tphi}, \eqref{fh decomp}, \eqref{M Choleksy} are now used to show that $\hat{h}_j(x;\lambda)$ and $\hat{f}_j(x;\lambda)$ defined in \eqref{hf def} satisfy \eqref{h prop}, \eqref{f prop}, and \eqref{A prop}, thereby completing the proof.
\end{proof}

\section{Relating the exact resolution of the identity and its approximation}\label{sec: exact to approx}

\subsection{Diagonalization of $\mathscr{K}$}
In our case only one interval $U=J\cup E=[a_1,a_2]$ is present. Lemma \ref{lemma: gamma no pole} implies that the eigenvalues of $\mathscr{K}$ do not accumulate at zero, so their number is, at most, finite. By the results in the proof of Theorem~5.5 in \cite{BKT20} (see eq. (5.24) and the text following it), all the eigenfunctions of $\mathscr{K}$ are smooth. By statement (2) of Theorem~\ref{thm: spec prop}, all the eigenspaces are finite-dimensional, and statement (1) of Theorem~\ref{thm:main} is proven.

Let $\mathscr{K}_{ac}$ denote the absolutely continuous part of $\mathscr{K}$, i.e. the part of $\mathscr{K}$ in the subspace $H_{ac}\subset L^2(U)$ of absolute continuity with respect to $\mathscr{K}$. 
Let $g_1,\dots,g_n\in H_{ac}$ be a generating basis for $\mathscr{K}_{ac}$.  Consider the matrix distribution function
\be\label{S-fn}
S(\l)=(\E(\l) g_j,g_k)_{j,k=1}^n.
\ee
In \cite{BKT20} the authors construct a self-adjoint operator $B:L^2([-1,1],\mathbb R^n)\to L^2([-1,1],\mathbb R^n)$, which is unitarily equivalent to $\mathscr{K}_{ac}$, i.e. $B=W^\dagger\mathscr{K}_{ac}W$ for a unitary $W:L^2([-1,1],\mathbb R^n)\to H_{ac}$, see \cite[eq. (6.20)]{BKT20}. 
Moreover, $B\tilde h=\l\tilde h$, where $\tilde h$ is a $\mathbb R^n$-valued function of $\l$. Clearly, the spectral family of $B$ satisfies $\V(\l) \tilde h=\chi_\l \tilde h$, where $\chi_\l$ is the characteristic function of $(-\infty,\l]$ (see, e.g. \cite[Theorem 7.18]{weidm80}).  Thus, we can find a convenient generating basis for $B$. Set $\tilde h_j(\l):=\vec e_j$, $\l\in[-1,1]$, where $\vec e_j$ is the $j$-th standard basis vector in $\mathbb R^n$. Then the collection $\tilde h_1,\dots,\tilde h_n\in L^2([-1,1],\mathbb R^n)$ is a generating basis with the property  
\be\label{gen_set}
    \frac{d}{d\lambda}(\V(\l) \tilde h_j,\vec h_k)=\delta_{jk},\ 1\leq j,k\leq n,\ \l\in [-1,1].
\ee
As is easily seen, $\V(\l)=W^\dagger\E(\l)W$, the collection $g_j=W\tilde h_j$, $1\leq j\leq n$, is a generating basis for $\mathscr{K}_{ac}$, and 
\be\label{S-id}
    S'_{jk}(\l)=\frac{d}{d\lambda}(\E(\l) g_j,g_k)=\delta_{jk},\ 1\leq j,k\leq n,\ \l\in [-1,1].
\ee

By \cite[Theorem p. 291]{ag80}, there is a unitary $\U_{ac}:\,H_{ac}\to L^2([-1,1],\mathbb R^n)$, $\U_{ac}\phi=\tilde\phi$, so that
\be
\begin{split}\label{isometry}
    &\phi=\sum_{j=1}^n \int_{-1}^1 \tilde \phi_j(\l) d\E(\l) g_j, ~~~ \phi\in H_{ac},\quad \tilde \phi=(\tilde \phi_1,\dots,\tilde \phi_n)\in L^2([-1,1],\mathbb R^n).
\end{split}
\ee
Moreover, $\U_{ac}\mathscr{K}_{ac}\U_{ac}^\dagger\tilde\phi=\l\tilde\phi$, $\tilde\phi\in L^2([-1,1],\mathbb R^n)$.
By looking at
\be\label{isometry-distr}
(\phi,f)=\sum_{j=1}^n \int_{-1}^1 \tilde \phi_j(\l) d(\E(\l) g_j,f),\ \phi,f\in H_{ac},
\ee
we conclude by duality that
\be\label{isometry-3-alt}
\U_{ac}f=\left((\E'(\l) g_1,f),\dots,(\E'(\l) g_n,f)\right),\ f\in H_{ac}.
\ee
Fix any $g\in H_{ac}$. Since $(\E(\l) g,f)=(\E(\l) g,f_{ac})$ for any $f\in L^2(U)$, the above formula naturally extends $\U_{ac}$ to all of $L^2(U)$. In particular, $\U_{ac}f=0$ if $f\in H_p$. The convergence of the integral in \eqref{isometry-distr} follows from the inequality
\be\label{E-ineq}
|d(\E(\l) g,f)|\leq \left[(d\E(\l) g,g)d(\E(\l) f,f)\right]^{1/2},\ g,f\in H_{ac},
\ee
see \cite[p. 356]{Kato}. Here we used \eqref{S-id} and that $\tilde\phi\in L^2([-1,1],\mathbb R^n)$.

Of interest to us are the following operators that make up the extended $\U_{ac}$:
\be\label{isometry-part}
\Q_j: L^2(U)\to L^2([-1,1]),\ \Q_j f:=(\E'(\l) g_j,f),\ 1\leq j\leq n,
\ee
which are bounded, because $\U_{ac}$ is bounded. 
Each operator $\Q_j$ can be viewed as a continuous map $C_0^\infty(\mathring{J}\cup \mathring{E})\to \mathcal D'([-1,1])$. Denote by $Q_j(x;\l):=\E'(\l) g_j$ the Schwartz kernel of $\Q_j$ (see \cite[Section 5.2]{hor1}). 
For any interval $\Delta\subset [-1,1]\setminus \{0\}$,
\be\begin{split}\label{dot-prod}
(\E_{ac}(\Delta)\phi,f)=&(\E_{ac}(\Delta)\phi_{ac},f)=
\int_{\Delta}(\E_{ac}'(\l)\phi_{ac},f)d\l=\int_\Delta \sum_j \tilde\phi_j(\l)(\E'(\l)g_j,f)d\l\\
=&\int_\Delta \sum_j (\E'(\l)g_j,\phi)(\E'(\l)g_j,f)d\l,\quad \phi,f\in C_0^\infty(\mathring{J}\cup \mathring{E}),
\end{split}
\ee
where $\tilde\phi=\U_{ac}\phi$.
Therefore, the kernel of $\E_{ac}'(\l)$ can be written in the form
\be\label{resol}
E_{ac}'(x,y;\l)=\sum_{j=1}^n Q_j(x;\l)Q_j(y;\l),\ |\l|\in (0,1].
\ee
The above equality is understood in the weak sense, i.e. given any $\phi,f\in C_0^\infty(\mathring{J}\cup \mathring{E})$, one has
\be\label{resol-weak}
\int_U\int_U E_{ac}'(x,y;\l)\phi(x)f(y)dxdy=\sum_{j=1}^n \int_U Q_j(x;\l)\phi(x)dx\int_U Q_j(y;\l)f(y)dy,\ 
|\l|\in (0,1].
\ee 
The left side of \eqref{resol-weak} is a $C^\infty([-1,1]\setminus\{0\})$ function by Proposition~\ref{prop: smoothness of res of id}. 
The right side is in $L^1([-1,1])$, because $\Q_j:\,L^2(U)\to L^2([-1,1])$, $1\leq j\leq n$, are all bounded.

As is well-known, the decomposition in \eqref{resol} is not unique. Let $V(\l)$ be an orthogonal matrix-valued function. Set  $\vec Q^{(1)}(x;\l):=V(\l)\vec Q(x;\l)$, where $\vec Q(x;\l)=(Q_1(x;\l),\dots,Q_n(x;\l))^t$. Clearly, \eqref{resol} does not change its form if $\vec Q$ is replaced with $\vec Q^{(1)}$:
\be\label{resol-1}
E_{ac}'(x,y;\l)=\sum_{j=1}^n Q_j^{(1)}(x;\l)Q_j^{(1)}(y;\l),\ |\l|\in (0,1].
\ee
Let $V:\,L^2([-1,1],\R^n)\to L^2([-1,1],\R^n)$ be the operator of multiplication with $V(\l)$. Clearly, $V$ is unitary. Define
\be\label{isometry-part-alt}\begin{split}
&\U_{ac}^{(1)} f:=V \U_{ac}f,\\ 
&(\U_{ac}^{(1)} f)(\l)=\left(\int_U Q_1^{(1)}(x;\l) f(x)dx,\dots,\int_U Q_n^{(1)}(x;\l) f(x)dx\right):\ L^2(U)\to L^2([-1,1],\R^n).
\end{split}
\ee
By construction, $\U_{ac}\mathscr{K}_{ac} \U_{ac}^\dagger \tilde f=\l\tilde f$ for any $\tilde f\in L^2([-1,1],\R^n)$. Therefore 
\be\label{diag-remains}
\U_{ac}^{(1)}\mathscr{K}_{ac} (\U_{ac}^{(1)})^\dagger 
=V\U_{ac}\mathscr{K}_{ac} \U_{ac}^\dagger V^\dagger=V\l V^\dagger=\l 
\ee
on $L^2([-1,1],\R^n)$. Thus, by appending an appropriate projector onto $H_p$ (the subspace of discontinuity with respect to $\mathscr{K}$) to $\U_{ac}^{(1)}$, we obtain a unitary transformation $\U^{(1)}$ that also canonically diagonalizes $\mathscr{K}$. Our argument proves statement (2) of Theorem~\ref{thm:main}.

\subsection{Piecewise smooth diagonalization of $\mathscr{K}$}\label{sec:diagK}

For another set of functions $h_j\in C_0^\infty(\mathring{J}\cup\mathring{E})$, $1\leq j\leq n$, let $h_{ac}^j$ be the orthogonal projection of $h_j$ onto $H_{ac}$. By \cite[eq. (5.11), p. 355]{Kato},
\be\label{diff Epr form}
\int_{-1}^{1}\left|(\E'(\l) g_j,g_k)-(\E'(\l) h_{ac}^j,h_{ac}^k)\right|d\l
\leq \Vert g_j\Vert \Vert g_k-h_{ac}^k\Vert+\Vert g_j-h_{ac}^j\Vert\Vert g_k\Vert.
\ee
Similarly to \eqref{S-fn}, denote $S_{\vec h}(\l)=(\E_{ac}(\l) h_j,h_k)_{j,k=1}^n$.
Since $\Vert g_j-h_{ac}^j\Vert\leq \Vert g_j-h_j\Vert$ and $\det(S_{\vec h}(\l))$ is analytic in $\l$ by Proposition~\ref{prop: smoothness of res of id}, we can find $h_j\in C_0^\infty(\mathring{J}\cup\mathring{E})$ with $\Vert h_j-g_j\Vert\ll 1$, $1\leq j\leq n$, such that $\det\left(S_{\vec h}'(\l)\right)>0$ except at a countable collection of $\l$, which may accumulate at $\l=0$.

Let $\Lambda$ be the open interval between any two consecutive roots (either positive or negative) of $\det\left(S_{\vec h}'(\l)\right)$. Combined with \cite[Lemma~12.1.6]{MSch02}, this implies that there exists a Cholesky factor $\hat Q^{(1)}(\l)$ of $S_{\vec h}'(\l)$ so that $\hat Q^{(1)}(\l)$ is a smooth matrix-valued function of $\l\in\Lambda$, and $\det (\hat Q^{(1)}(\l))>0$, $\l\in\Lambda$. To be precise, \cite[Lemma~12.1.6]{MSch02} establishes only continuity, but the same proof can be modified to show that $\hat Q^{(1)}(\l)$ is differentiable any number of times as long as $S_{\vec h}'(\l)$ is smooth and nondegenerate.

By \eqref{resol-weak}, $\hat Q(\l)$, where 
\be\label{Qphi-appl}
\hat Q(\l):=\Q(\l)\vec h,\ \Q(\l):=(\Q_1(\l),\dots,\Q_n(\l))^t,\ \vec h:=(h_1,\dots,h_n),\ |\l|\in(0,1],
\ee
(cf. \eqref{isometry-part} for the definition of the operators $\Q_j$), is also a Cholesky factor of $S_{\vec h}'(\l)$. By replacing $Q_1(x;\l)$ with $-Q_1(x;\l)$ for $\l\in\Lambda$ wherever necessary, we can ensure that $\det (\hat Q(\l))>0$ for all $\l\in\Lambda$. Therefore
\be\label{two-chols}\begin{split}
S_{\vec h}'(\l)=(\hat Q^{(1)})^t(\l)\hat Q^{(1)}(\l)=\hat Q^t(\l)\hat Q(\l),\ \l\in\Lambda.
\end{split}
\ee

\bl\label{lem:can_chol} 
If $\det \hat M\not=0$, and $\hat Q_1$ and $\hat Q_2$ are two Cholesky factors of $\hat M$, then 
$\hat Q_2=V \hat Q_1$ for some orthogonal matrix $V$.
\el
\noindent
The proof is immediate, because $V=\hat Q_2\hat Q_1^{-1}$ satisfies $V^tV=\1$.

By the Lemma, $\hat Q^{(1)}(\l)=V(\l)\hat Q(\l)$ for an orthogonal matrix-valued function $V(\l)$, $\l\in\Lambda$. Clearly, $V(\l)\in SO(n)$ (i.e., $\det(V(\l))\equiv1$), because the determinants of $\hat Q^{(1)}(\l)$ and $\hat Q(\l)$ are positive. Define a new kernel by $\vec Q^{(1)}(x;\l):=V(\l)\vec Q(x;\l)$. Then, by \eqref{Qphi-appl},
\be\label{two-Qs}
\Q^{(1)}(\l)\vec h=V(\l)\Q(\l)\vec h=V(\l)\hat Q(\l)=\hat Q^{(1)}(\l).
\ee
Recall that $\Q$ and $\Q^{(1)}$ are column-vectors consisting of the operators $\Q_j$ and $\Q_j^{(1)}$, respectively (cf. \eqref{Qphi-appl}).
As mentioned above (see \eqref{resol-1} and the text preceding it), the kernel $\vec Q^{(1)}(x;\l)$ of $\Q^{(1)}$ satisfies \eqref{resol-1} for $\l\in\Lambda$. 

Consider the equations
\be\label{smoothness-2}
(\E_{ac}'(\l)h_m,f)
=\sum_{j=1}^n (\Q_j^{(1)} h_m)(\l)(\Q_j^{(1)}f)(\l),\ m=1,\dots, n,
\ee
for any $f\in C_0^\infty(\mathring{J}\cup \mathring{E})$ and $\l\in\Lambda$. We solve these equations for $(\Q_j^{(1)}f)(\l)$, $j=1,\dots, n$, which can be done because the matrix of the system is $\hat Q^{(1)}(\l)$ (cf. \eqref{two-Qs}). The latter is nondegenerate by construction. By the Schwartz kernel Theorem \cite[Theorem 5.2.1]{hor1}, $Q_j^{(1)}\in C^{\infty}\left((\mathring{J}\cup \mathring{E})\times\Lambda\right)$. This claim follows, because $\hat Q^{(1)}\in C^{\infty}(\Lambda)$ by construction.

Repeating the above argument for all intervals $\Lambda\subset (-1,1)\setminus\{0\}$ such that $\det\left(S_{\vec h}'(\l)\right)>0$, $\l\in\Lambda$, we construct the kernel $\vec Q^{(1)}(x;\l)$ in \eqref{resol-1}, which is smooth except for a countable collection of $\l$, which may accumulate at $\l=0$. Next we show that such accumulation does not happen.

In Proposition~\ref{Eprime simple} we find another smooth kernel $L\in C^\infty\left((\mathring{J}\cup \mathring{E})\times (\mathring{J}\cup \mathring{E})\times ((-1,1)\setminus\{0\})\right)$, which is close to $|\l| E'$ pointwise for all $|\l|>0$ sufficiently small (in this case $E'\equiv E_{ac}'$) and admits an expansion similar to \eqref{resol}: 
\be\label{resol-pr}
L(x,y;\l)=\sum_{j=1}^n G_j(x;\l)G_j(y;\l),\,  G_j\in C^\infty\left((\mathring{J}\cup \mathring{E})\times ((-1,1)\setminus\{0\})\right),\, x,y\in \mathring{J}\cup \mathring{E},\, |\l|\in (0,1).
\ee

\bl \label{lem:le_appr} Given any interval $I:=[x_l,x_r]\subset \mathring{J}\cup \mathring{E}$, one has 
\be\label{prox}
\max_{x,y\in I}\left||\l|E'(x,y;\l)-L(x,y;\l)\right|=\BigO{\ka^{-1}},\ \ka=-\ln\left|{\l}/{2}\right|\to\infty.
\ee
Also, there exist $\delta,c,\l_0>0$ such that for any $\l'$, $0 <|\l'|< \l_0$, one can find $x_k(\l)\in I$, $1\leq k\leq n$, which are all smooth in a neighborhood of $\l'$, and 
\be\label{det-cond}
\det\left(\hat L(\l)\right)>\delta,\ \Vert \hat L(\l)\Vert\leq c,\ 0 <|\l|\leq \l_0,
\ee
where $\hat L(\l):=[L(x_m(\l),x_k(\l);\l)]_{m,k=1}^n$.
\el
\begin{proof} The property \eqref{prox} is proven in Proposition~\ref{Eprime simple}, see \eqref{Eprime approx}. Now we prove \eqref{det-cond}. 
By Proposition~\ref{Eprime simple}, $G_j(x;\l)$ can be written in the form: 
\be\label{appr-fns}
    G_j(x;\l)=\A_j^\pm(x)\cos(\ka \gg_{\text{im}}(x)+c_j),\ 1\leq j\leq n,\ \gg_{\text{im}}(x):=\Im(\gg(x_+)),
    \ x\in I,
\ee
for some real-valued functions $\A_j^\pm$ and real constants $c_j$, where the functions $\A_j^\pm(x)$ are analytic in a neighborhood of $I$ and linearly independent. The signs `$+$' and `$-$' in the superscript are taken if $\l>0$ and $\l<0$, respectively. Using \eqref{h prop}, \eqref{f prop} we have
\be\label{Ac coefs}
\A_j^\pm(x)=A_j(x)\begin{cases} s(j),& x\in E,\\ 
\mathrm{sgn}(\l),& x\in J, \end{cases}\quad
c_j=s(j)\frac\pi4.
\ee
Since $\A_j^\pm(x)$ are linearly independent (as analytic functions), we can find $n$ points $\check x_k\in \mathring{I}$ so that the vectors $(\A_j^\pm(\check x_1),\dots,\A_j^\pm(\check x_n))$, $1\leq j\leq n$, are linearly independent. Here and below, the cases $\l>0$ and $\l<0$ are treated separately.

Find $c_0$ so that $\cos(c_0+c_j)\not=0$ for any $j$. Define $N_k(\ka):=\lfloor (\ka \gg_{\text{im}}(\check x_k)-c_0)/(2\pi)\rfloor$ and find $x_k(\l)\in I$ by solving $\gg_{\text{im}}(x)=(c_0+2\pi N_k(\ka))/\ka$. Then 
\be\label{g-root finding}
\gg_{\text{im}}'(\check x_k)\Delta x+O(\Delta x^2)=\frac{2\pi \left[N_k(\ka)-(\ka\gg_{\text{im}}(\check x_k)-c_0)/(2\pi)\right]}{\ka}=\frac{O(1)}\ka,
\ee
where $\Delta x=x_k(\l)-\check x_k$. The magnitude of the $O(1)$ term in \eqref{g-root finding} does not exceed $2\pi$, so
\be\label{delta x est}
|\Delta x|\lesssim \frac{2\pi}{\min_{x\in I}|\gg_{\text{im}}'(x)|}\frac1\ka.
\ee
Since $I$ is at a positive distance from the endpoints of $J$ and $E$, statements (3) and (4) of Proposition~\ref{prop: g-function} imply that the denominator in \eqref{delta x est} is positive. Since $\check x_k\in \mathring{I}$ are all fixed, find $\ka_0=-\ln(\l_0/2)$ large enough so that $x_k(\l)\in I$, $1\leq k\leq n$, if $|\l|\geq\l_0$. 
By construction, $x_k(\l)$ are piecewise smooth functions of $\l$, and they satisfy $\ka \gg_{\text{im}}(x_k(\l))\equiv c_0\mod 2\pi$ and $x_k(\l)\to \check x_k$, $\l\to 0$, $1\leq k\leq n$. 
Clearly, 
\be\label{appr_fns}
(G_j(x_k(\l);\l)_{j,k=1}^n=\left(\A_j^\pm(x_k(\l))\cos(c_0+c_j)\right)_{j,k=1}^n \to
\left(\A_j^\pm(\check x_k)\cos(c_0+c_j)\right)_{j,k=1}^n \text{ as }\l\to0,
\ee
and the last matrix above is non-degenerate by construction. Decreasing $\l_0>0$ even further if necessary, all the inequalities in \eqref{det-cond} follow from \eqref{resol-pr} and the definition of $\hat L(\l)$ in Lemma~\ref{lem:le_appr}.

The $x_k(\l)\in I$ are smooth at $\l'$ if $\ka'=-\ln|\l'/2|$ is not a jump point of $N_k(\ka)$ for all $1\leq k\leq n$. Suppose $N_k(\ka)$ is discontinuous at $\ka'$ for some $k$. At each such point the value of the jump equals $1$ (if $\gg_{\text{im}}(\check x_k)>0$) or $-1$ (if $\gg_{\text{im}}(\check x_k)<0$).
If  $\gg_{\text{im}}(\check x_k)=0$, $N_k(\ka)$ is constant. To avoid the discontinuity at $\ka'$, we redefine $N_k(\ka)$ by keeping it constant $N_k(\ka)=N_k(\ka'-0)$, $\ka\geq\ka'$, until the next discontinuity, where it jumps by 2 or -2. The rest of $N_k(\ka)$ is unchanged.
If several $N_k(\ka)$ are discontinuous at $\ka'$, each one is modified the same way. If a modification is performed, $O(1)$ in \eqref{g-root finding} is bounded by $4\pi$, so the modified $x_k(\l)$ still satisfy \eqref{det-cond}. 
\end{proof}

The following Corollary follows immediately from Lemma~\ref{lem:le_appr}.

\bc \label{cor:he-nondeg} Given any interval $I:=[x_l,x_r]\subset \mathring{J}\cup \mathring{E}$, there exist $\tilde \delta,\tilde c,\tilde \l_0>0$, such that for any $\l'$, $0 <|\l'|<\tilde \l_0$, 
\be\label{det-cond-E}
\det\left(\hat E'(\l)\right)>\tilde \delta,\ \Vert \hat E'(\l)\Vert\leq \tilde c,\ |\l|\in(0,\tilde \l_0],
\ee
where $\hat E'(\l):=(|\l|E'(x_m(\l),x_k(\l);\l)_{m,k=1}^n)$, and $x_k(\l)\in I$ are the same as in Lemma~\ref{lem:le_appr}.
%
\ec

Note that the factor $|\l|$ is included in the definition of $\hat E'(\l)$. 
By Corollary~\ref{cor:he-nondeg} and \cite[Lemma~12.1.6]{MSch02}, there exists a Cholesky factor $\hat Q^{(2)}(\l)$ of $|\l|^{-1}\hat E'(\l)$, $|\l|\in(0,\tilde \l_0]$, so that $\hat Q^{(2)}(\l)$ is a piecewise smooth matrix-valued function of $\l$, and $\det (\hat Q^{(2)}(\l))>0$. The factor is  smooth wherever $x_k(\l)$ are all smooth and, therefore, $\hat E'(\l)$ is smooth. 

Pick any $\l'$, $|\l'|\in (0,\tilde\l_0]$, where $\vec Q^{(1)}(x;\l)$ constructed above (see \eqref{smoothness-2} and the text that follows it) is not smooth. By Corollary~\ref{cor:he-nondeg}, we can find $x_k(\l)\in I$, $1\leq k\leq n$, which are smooth in $(\l'-\e,\l'+\e)$ for some small $\e>0$. By looking at the matrix $|\l|^{-1}\hat E'(\l)$ and arguing similarly to \eqref{two-chols}--\eqref{smoothness-2}, we get the kernel $\vec Q^{(2)}(x;\l)$, which is smooth in this interval. 

Consider the interval $(\l'-\e,\l')$. On this interval we have two kernels, which are related as follows $\vec Q^{(1)}(x;\l)=V_-(\l)\vec Q^{(2)}(x;\l)$, $\l\in(\l'-\e,\l')$, for some $SO(n)$-valued function $V_-(\l)$. This follows, because
\be\label{Q demonstr I}\begin{split}
&E'(x_m(\l),x_k(\l);\l)=\sum_{j=1}^n Q_j^{(1)}(x_m(\l);\l)Q_j^{(1)}(x_k(\l);\l)=\sum_{j=1}^n Q_j^{(2)}(x_m(\l);\l)Q_j^{(2)}(x_k(\l);\l),\\
&E'(x_m(\l),x;\l)=\sum_{j=1}^n Q_j^{(1)}(x_m(\l);\l)Q_j^{(1)}(x;\l)=\sum_{j=1}^n Q_j^{(2)}(x_m(\l);\l)Q_j^{(2)}(x;\l),\\
&1\leq k,m\leq n,\ x\in \mathring{J}\cup \mathring{E},\ \l\in(\l'-\e,\l').
\end{split}
\ee
From the first line we obtain that the matrices $(Q_j^{(1)}(x_m(\l);\l))$ and $(Q_j^{(2)}(x_m(\l);\l))$ are non-degenerate and are, therefore, related by some $V_-(\l)$. As before, we can ensure that the determinant of the second matrix is positive (the determinant of the first matrix is positive by construction).
Solving the two equations on the second line for $Q_j^{(1)}(x;\l)$ and $Q_j^{(2)}(x;\l)$ gives the desired relation.

Since $V_-(\l)$ can be changed arbitrarily, we replace it with a smooth $SO(n)$-valued function $\tilde V(\l)$, which equals $\1$ on $(\l'-\e/3,\l')$ and equals $V_-(\l)$ on $(\l'-\e,\l'-2\e/3)$. Here we use that $SO(n)$ is connected. The same trick can be applied in $(\l',\l'+\e)$, which gives another $SO(n)$-valued function $V_+(\l)$ defined on $(\l',\l'+\e)$. Smoothly modifying it as well in a similar fashion, we get a smooth function $\tilde V(\l)$ defined on  $(\l'-\e,\l'+\e)$ and a kernel $\vec Q^{(0)}(x;\l)=\tilde V(\l)\vec Q^{(2)}(x;\l)$, which is smooth across $\l'$ and smoothly transitions to $\vec Q^{(1)}(x;\l)$ on either side of $\l'$. Thus, all the discontinuities of $\vec Q^{(1)}(x;\l)$ sufficiently close to $\l=0$ can be eliminated. 

In what follows, the obtained kernel will be denoted $\vec Q^{(0)}(x;\l)$. By construction, $Q_j^{(0)}\in\linebreak C^{\infty}\left((\mathring{J}\cup\mathring{E})\times \left((-1,1)\setminus\Xi\right)\right)$, $1\leq j\leq n$, where $\Xi$ is a finite set of points. Differentiating \eqref{resol} with $\vec Q=\vec Q^{(0)}$ with respect to $x$ and $y$ and setting $x=y$ gives
\be\label{resol diff}
\frac{\partial^n}{\partial x^n}\frac{\partial^n}{\partial y^n} E_{ac}'(x,y;\l)|_{y=x}=\sum_{j=1}^n \left(\frac{\partial^nQ_j^{(0)}(x;\l)}{\partial x^n} \right)^2,\ |\l|\in (0,1].
\ee
The property \eqref{E phi 3 II} now follows immediately from Proposition~\ref{prop: smoothness of res of id}. Our argument proves statement (3) of Theorem~\ref{thm:main}.

\subsection{Approximating a smooth diagonalization of $\mathscr{K}$}
Here we construct the functions $Q_j$, which have the properties stated in statement (3) of Theorem~\ref{thm:main} and are close to $G_j$. In this subsection we always assume $|\l|\in (0,\tilde\l_0]$. In this subsection we can use any collection $x_k(\l)$, $1\leq k\leq n$, constructed in Lemma~\ref{lem:le_appr}. There is no need to adjust any of the $x_k(\l)$ to ensure its continuity at some $\l'$.

By Lemma~\ref{lem:le_appr}, Corollary~\ref{cor:he-nondeg}, and \cite[Lemma~12.1.6]{MSch02}, there exist Cholesky factors $\hat G^{(1)}(\l)$ and $\hat Q^{(1)}(\l)$ of $\hat L(\l)$ and $\hat E'(\l)$, respectively, so that $\hat G^{(1)}(\l)$ and $\hat Q^{(1)}(\l)$ are piecewise smooth matrix-valued functions of $\l$, and $\det (\hat G^{(1)}(\l)),\det (\hat Q^{(1)}(\l))>0$. Both factors are smooth wherever $x_k(\l)$ are all smooth and, therefore, $\hat L(\l)$ and $\hat E'(\l)$ are smooth. We may assume that both $\hat Q^{(1)}(\l)$ and $\hat G^{(1)}(\l)$ are upper triangular. This can be achieved by applying the appropriate numerical algorithm to $\hat L(\l)$ and $\hat E'(\l)$ (e.g., as in the proof of \cite[Lemma~12.1.6]{MSch02}). The required form of the Cholesky factors guarantees that if $\hat L(\l)$ and $\hat E'(\l)$ are close, then $\hat G^{(1)}(\l)$ and $\hat Q^{(1)}(\l)$ are close as well. Using the same argument as in the previous subsection, these factors lead to the kernels $\vec G^{(1)}(x;\l)=V_G(\l)\vec G(x;\l)$ and $\vec Q^{(1)}(x;\l)=V_Q(\l)\vec Q^{(0)}(x;\l)$ for some $SO(n)$-valued functions $V_G(\l)$, $V_Q(\l)$.

In this subsection we do not divide $\hat E'$ by $|\l|$ when computing its Cholesky factors, so $\hat Q^{(1)}(\l)$ here has an extra factor $|\l|^{1/2}$ compared with an analogous Cholesky factor of the previous subsection.

The kernels $\vec Q^{(1)}(x;\l)$ and $\vec G^{(1)}(x;\l)$ are smooth in $x\in\mathring{J}\cup \mathring{E}$ and piecewise-smooth in $\l$. The kernel $\vec Q^{(0)}(x;\l)$ has at most finitely many jumps in $\l$, and $\vec G(x;\l)$ (cf. \eqref{resol-pr}) is smooth. Additional jumps in $\l$ in $\vec Q^{(1)}(x;\l)$ and $\vec G^{(1)}(x;\l)$ arise due to the discontinuities in $x_k(\l)$, and the latter are reflected in the discontinuities of $V_G(\l)$ and $V_Q(\l)$. 

Let us now formulate a result on stability of the Cholesky factorization. Let $\hat L$ be a positive definite symmetric matrix, and $\hat G$ be its upper triangular Cholesky factor. Let $\Delta \hat L$ and be a symmetric perturbation of $\hat L$, and $\Delta \hat G$ be the corresponding upper triangular perturbation of $\hat G$. Suppose
\be\label{chol-pert}
\e:=\frac{\Vert \Delta \hat L\Vert_F}{\Vert \hat L\Vert_2},\ \kappa_2:=\Vert \hat L\Vert_2\Vert \hat L^{-1}\Vert_2,\ \kappa_2\e <1.
\ee
Here $\Vert \cdot \Vert_{F}$ denotes the Frobenius norm of a matrix, and $\Vert \cdot \Vert_2$ is the matrix 2-norm. Then \cite{sun91, ste93}
\be\label{chol-pert-bnd}
\frac{\Vert \Delta \hat G\Vert_F}{\Vert \hat G\Vert_2}\leq \frac1{\sqrt2}\kappa_2\e+\BigO{\e}^2.
\ee
In what follows we do not specify which matrix norm is used, since they are all equivalent for our purposes. By \eqref{prox} and \eqref{det-cond}, \eqref{chol-pert} is satisfied, where $\Delta \hat L(\l):=\hat E'(\l)-\hat L(\l)$ and $\e=\BigO{\varkappa^{-1}}$. Equation \eqref{chol-pert-bnd} gives
\be\label{pert-bnd}
\frac{\Vert \hat Q(\l) - \hat G(\l)\Vert}{\Vert \hat G(\l)\Vert} = \BigO{\varkappa^{-1}}.
\ee
Consequently, \eqref{prox}, \eqref{det-cond}, and \eqref{pert-bnd} imply
\be\label{invm-bnd}
\Vert \hat Q^{-1}(\l) - \hat G^{-1}(\l)\Vert = \BigO{\varkappa^{-1}}.
\ee

For any $x\in \mathring{J}\cup \mathring{E}$, consider two systems of equations
\be\label{resol-pr-2}\begin{split}
L(x_k(\l),x;\l)=&\sum_{j=1}^n G_{j}^{(1)}(x_k(\l);\l)G_{j}^{(1)}(x;\l),\  1\leq k\leq n,\\ 
|\l|E'(x_k(\l),x;\l)=&\sum_{j=1}^n |\l|^{1/2}Q_{j}^{(1)}(x_k(\l);\l)|\l|^{1/2}Q_{j}^{(1)}(x;\l),\ 1\leq k\leq n,
\end{split}
\ee
which are solved for the vectors $\vec G^{(1)}(x;\l)$ and $|\l|^{1/2}\vec Q^{(1)}(x;\l)$, respectively. The matrices of the two systems are $\hat G^{(1)}(\l)$ and $\hat Q^{(1)}(\l)$, respectively. This follows similarly to \eqref{two-Qs}. Combining \eqref{prox} and \eqref{invm-bnd} gives the desired bound 
\be\label{final-bnd}
\Vert \vec G^{(1)}(x;\l)-|\l|^{1/2}\vec Q^{(1)}(x;\l)\Vert = \BigO{\varkappa^{-1}},
\ee
which is uniform with respect to $x$ in compact subsets of $\mathring{J}\cup \mathring{E}$.

Recall that left multiplication with $V_G^t(\l)$ converts $\vec G^{(1)}(x;\l)$ back into the original kernel $\vec G(x;\l)$ of \eqref{resol-pr}: $\vec G(x;\l)=V_G^t(\l)\vec G^{(1)}(x;\l)$. Define $\vec Q^{(2)}(x;\l):=V_G^t(\l)\vec Q^{(1)}(x;\l)$. 
Multiplication with $V_G^t(\l)$ does not affect \eqref{final-bnd}, which therefore holds for the pair $\vec G(x;\l)$, $\vec Q^{(2)}(x;\l)$. As was mentioned above, $V_G^t(\l)$ is smooth unless one or more of $x_k(\l)$ have a jump, and $\vec Q^{(1)}(x;\l)$ is smooth in $x\in\mathring{J}\cup \mathring{E}$ and piecewise-smooth in $\l$. Therefore, $\vec Q^{(2)}(x;\l)$ has the same properties. 

Finally, we modify $\vec Q^{(2)}(x;\l)$ slightly, so that it still satisfies \eqref{resol}, remains close to $\vec G(x;\l)$, but is smooth in $\l$ for $\l$ close to zero. By construction, $\vec Q^{(2)}(x;\l)=V(\l)\vec Q^{(0)}(x;\l)$, where $V(\l)=V_G^t(\l)V_Q(\l)$, and $\Vert \vec G(x;\l)-V(\l)|\l|^{1/2}\vec Q^{(0)}(x;\l)\Vert = \BigO{\varkappa^{-1}}$. Both $\vec G(x;\l)$ and $\vec Q^{(0)}(x;\l)$ are smooth (recall that $|\l|\in (0,\tilde\l_0]$). Also, by Lemma~\ref{lem:le_appr} and Corollary~\ref{cor:he-nondeg}, the corresponding determinants $\det (\hat G(\l))$ and $\det (\hat Q^{(0)}(\l))$ are bounded away from zero as $\l\to0$. Let $\l_m$ be the points where $V(\l)$ is discontinuous. Then, $\Vert V(\l_m+0)-V(\l_m-0)\Vert=\BigO{\varkappa_m^{-1}}$ as $\l_m\to0$. Since $SO(n)$ is connected, we can find a small $\e_m>0$ neighborhood around each $\l_m$ and a smooth path $V_{final}(\l)\in SO(n)$ that connects $V(\l_m-\e_m)$ with $V(\l_m+\e_m)$ without deviating from them far. In other words,
\be\label{Vfinal}\begin{split}
&V_{final}(\l)=\1,\ |\l|\in[\tilde \l_0+\e_0,1];\\
&V_{final}(\l)=V(\l),\ |\l|\in(0,\tilde \l_0-\e_0],\, \l\not\in\cup_m(\l_m-\e_m,\l_m+\e_m);\\
&V_{final}\in C^\infty\left((-1,1)\setminus\{0\}\right);\ \Vert V_{final}(\l)-V(\l)\Vert=\BigO{\varkappa^{-1}},\l\to0.
\end{split}
\ee
It is now clear that the kernel $\vec Q_{final}(x;\l)=V_{final}(\l)\vec Q^{(0)}(x;\l)$ has all the required properties. 
Our argument proves statement (4) of Theorem~\ref{thm:main}. 
 
Our argument establishes that the operator with the kernel $|\l|^{-1/2}\vec G(x;\lambda)$ to leading order as $\l\to0$ diagonalizes $\mathscr{K}$. Similarly, Corollary~\ref{cor:main} and equation \eqref{Eac h f} imply that approximate diagonalization of $A^\dagger A$ and $AA^\dagger$ is achieved by the operators with the kernels $|\l|^{-1/2}(\hat{f}_1(x;\lambda),\dots,\hat{f}_n(x;\lambda))$ and $|\l|^{-1/2}(\hat{h}_1(x;\lambda),\dots,\hat{h}_n(x;\lambda))$, respectively.

\subsection{Degree of ill-posedness of inverting $\mathscr{K}$}
The discussion in this section allows us to estimate the degree of ill-posedness of inverting $\mathscr{K}$. Consider the equation $\mathscr{K} f=\phi$, where $f,\phi\in L^2(U)$, and $\phi$ is in the range of $\mathscr{K}$. Pick any $x_0\in \mathring{J}\cup \mathring{E}$ and suppose $\text{sing\,supp}(f)=\{x_0\}$. Our goal is to estimate how unstable it is to reconstruct the singularity of $f$ at $x_0$. Thus, we can assume without loss of generality that $\text{supp}(f)\subset \mathring{J}\cup \mathring{E}$.
To simplify notation, the kernel $\vec Q_{final}(x;\l)$ is denoted $\vec Q(x;\l)$ in this subsection.

Pick any $\delta\in(0,1)$. By \eqref{resol}, converting the equation to the spectral domain, i.e. applying the operator $\U_{ac}$ of \eqref{isometry}, and then using the first line in \eqref{diag-two-comps} gives
\be\label{solution}
\l \tilde f_j(\l)=\tilde \phi_j(\l),\  \tilde f_j(\l)=\l^{-1}\tilde \phi_j(\l),\ 1\leq j\leq n,\ |\l|\in (0,1],
\ee
where 
\be\label{fourier_tr}\begin{split}
\tilde f_j(\l)=&\int_U Q_j(x;\l)f(x)dx,\ \tilde \phi_j(\l)=\int_U Q_j(x;\l)\phi(x)dx,\ 1\leq j\leq n,\ |\l|\in (0,1],\\ 
f(x)=&\sum_{j=1}^n \int_{-\delta}^{\delta} \l^{-1}Q_j(x;\l)\tilde \phi_j(\l)d\l + f_{sm}(x),\ x\in \mathring{J}\cup \mathring{E}.
\end{split}
\ee
The term $f_{sm}$ can be written in the form
\be\label{smooth part}\begin{split}
f_{sm}(x)=&\sum_{j=1}^n \int_{|\l|\in [\delta,1]} Q_j(x;\l)\tilde f_j(\l)d\l + f_p(x),\ x\in \mathring{J}\cup \mathring{E},
\end{split}
\ee
where $f_p$ is the projection of $f$ onto $H_p$. Since $\text{supp}(f)\subset \mathring{J}\cup \mathring{E}$, \eqref{E phi 3 II} implies $\max_{j,|\l|\in [\delta,1]}|\tilde f_j(\l)|<\infty$. By assertion (1) of Theorem~\ref{thm:main}, $f_p\in C^\infty(\mathring{J}\cup \mathring{E})$. Therefore, $f_{sm}\in C^\infty(\mathring{J}\cup \mathring{E})$ by \eqref{E phi 3 II}. Clearly, the operator that reconstructs $f_{sm}$ from $\phi$ is stable (i.e., bounded between the $L^2$ spaces). Hence the instability of reconstruction comes from a neighborhood of $\l=0$.

By Proposition~\ref{Eprime simple}, \eqref{appr-fns}, \eqref{solution}, and \eqref{fourier_tr}, a leading order reconstruction of $f$ from $\phi$ becomes 
\be\label{lead_order}\begin{split}
    f(x)\sim & \sum_{j=1}^n \int_{-\delta}^{\delta} \l^{-1}|\l|^{-1/2}A_j^\pm(x)\cos((-\ln|\l/2|) \gg_{\text{im}}(x)+c_j)\tilde \phi_j(\l)d\l.
\end{split}
\ee
The sign `$\sim$' here
means that the difference between the left- and right-hand sides is smoother than $f$ whenever $f$ is of a limited smoothness. 

Next, we change variables $\l\to\ka=-\ln|\l/2|$ in \eqref{lead_order}. To preserve the $L^2$-norm, we need to multiply $\tilde \phi_j$ by $|\l|^{1/2}$. Setting $\tilde \Phi_j^\pm(\ka):=|\l|^{1/2}\tilde \phi_j(\pm|\l|)$, we have
\be\label{two norms}
\int_{\mathbb R}\tilde \phi_j^2(\l)d\l=\int_0^\infty \left(\tilde \Phi_j^+(\ka)\right)^2d\ka+\int_0^\infty \left(\tilde \Phi_j^-(\ka)\right)^2d\ka,
\ee
and \eqref{lead_order} becomes
\be\label{lead_order_v2}\begin{split}
&f(x)\sim \frac12\sum_{j=1}^n \int^{\infty}_{-\ln(\delta/2)} e^{\ka}\cos(\ka \gg_{\text{im}}(x)+c_j)\left[A_j^+(x)\tilde \Phi_j^+(\ka)-A_j^-(x)\tilde \Phi_j^-(\ka)\right]d\ka.
\end{split}
\ee
Finally, we linearize $\gg_{\text{im}}(x)$ near $x=x_0$: $\gg_{\text{im}}(x)\sim \gg_{\text{im}}(x_0)+\gg_{\text{im}}'(x_0)\Delta x$. This implies that to reconstruct $f$ to leading order near $x_0$ involves computing integrals of the form
\be\begin{split}
    \int^{\infty}_{*} e^{\ka}{\begin{Bmatrix}\cos\\ \sin
    \end{Bmatrix}}(\gg_{\text{im}}'(x_0)\Delta x\ka)\tilde G_k(\ka)d\ka,
\end{split}
\ee
where $\tilde G_k(\ka)$ are obtained from $\tilde \Phi_k^\pm(\ka)$ via linear combinations with bounded coefficients depending on {$\ka$ and} \blue{$x_0$}. Consequently, the degree of ill-posedness is $\exp\left(\ka/|\gg_{\text{im}}'(x_0)|\right)$. This means that
\begin{enumerate}
    \item the inversion of $\mathscr{K}$ is exponentially unstable,
    \item the degree of instability of inverting $\mathscr{K}$ is spatially variant,
    \item the smaller the value of $|\gg_{\text{im}}'(x_0)|$, the more unstable the reconstruction of $f$ at this point.
\end{enumerate}

The same argument allows us to estimate the degree of ill-posedness of inverting the operators $AA^\dagger$ and $A^\dagger A$, and we obtain qualitatively similar conclusions. {Indeed, by Corollary~\ref{cor:main}, the operators that diagonalize the absolutely continuous parts of $AA^\dagger$, $A^\dagger A$, are obtained by the appropriate truncation of $\U_{ac}$. So the above argument carries over in an essentially the same way. The differences are that we have to divide by $\l^2$ rather than by $\l$, and the spectral interval is $[0,1]$. Hence the degree of ill-posedness becomes $\exp\left(2\ka/|\gg_{\text{im}}'(x_0)|\right)$.}

 

\appendix

\section{Hypergeometric Parametrix}

In \cite{BBKT19}, the authors solve RHP \ref{RHPGamma} when $E=[-a,0], J=[0,a]$, where $a>0$, explicitly in terms of hypergeometric functions for $\lambda\in\mathbb{C}\setminus[-1,1]$.  For $\lambda\in[-1,1]$, the solution of RHP \ref{RHPGamma} is not unique but two important solutions can be obtained via continuation in $\lambda$ from $\mathbb{C}\setminus[-1,1]$ onto $[-1,1]$ from the upper or lower half plane.  We call this explicit solution $\Gamma_4(z;\lambda)$ and refer the reader to \cite[eq. (17)]{BBKT19} for the exact expression.  Moreover, the small-$\lambda$ asymptotics of $\Gamma_4(z;\lambda)$ was obtained in various regions of the complex $z$-plane.  We briefly summarize these results, as they will be key components in the parametrices at the double points $b_k$, $k=1,\dots,n$, see \eqref{ParamIeLeft}, \eqref{ParamIeRight}.

\subsection{Model solution for the two interval problem with no gap}\label{sect-modsol2}
The model RHP for the two symmetrical intervals problem with no gap, considered here, has the jump matrix $-i\s_1$ on the interval $(-a,0)$ and 
the jump matrix $i\s_1$ on the interval $(0,a)$. 
Because of the discontinuity of the jump matrix at $z=0$, we expect solutions to have $\BigO{z^{-\frac{1}{2}}}$ behavior at $z=0$. 

Since the jump matrices $\pm i\sigma_1$ at any $z\in(-a,0)\cup(0,a)$ commute with each other, a solution to this RHP can be written as
\be\label{RHP_4_model0}
\tilde\Phi(z)=\le(\frac{(z-a)(z+a)}{z^2}\ri)^{\frac{\s_1}{4}}.
\ee
However, because of the $\BigO{z^{-\frac{1}{2}}}$ behavior near $z=0$, this solution of the model RHP is not unique. In fact, the most general solution has the form
\be\label{RHP_4_model1}
\Phi(z)=\le(\1+ \frac Az\ri)\le(\frac{(z-a)(z+a)}{z^2}\ri)^{\frac{\s_1}{4}},
\ee
where
\be\label{A-mod_4}
A= \le[
\begin{array}{cc}
x & -x\\
y & -y
\end{array}
\ri]
\ee
with $x,y\in\C$ being arbitrary constants. It is shown in Theorem \ref{GammaAsmptotics} that we should take $x=y=\frac{i}{2}$ when $\Im[\lambda]\geq0$ and $x=y=-\frac{i}{2}$ when $\Im[\lambda]\leq0$, so the model solution $\Psi_4(z;\varkappa)$ that will be using is
\be\label{Psi_4}
\Psi_4(z;\varkappa)=\begin{cases}
\le(\1+ \frac {iaB}{2z}\ri)\le(\frac{(z-a)(z+a)}{z^2}\ri)^{\frac{\s_1}{4}}, & \text{ for } \Im\varkappa\leq0, \\
\sigma_1\le(\1+ \frac {iaB}{2z}\ri)\le(\frac{(z-a)(z+a)}{z^2}\ri)^{\frac{\s_1}{4}}\sigma_1, &\text{ for } \Im\varkappa\geq0,
\end{cases}
\ee
where $B=\s_3-i\s_2$ and equality in $\geq,\leq$ is to be understood as continuation from the upper, lower half plane, respectively.

One can also obtain the model solution \eqref{Psi_4} following the method, described in Section \ref{ssect-sol_mod_gen}. According to Lemma \ref{lem-mod-sol},
a solution to the model problem can be written as 
\be\label{RHP_4_model2}
\Phi_4(z;\varkappa)=e^{\tgg_4(\infty;\varkappa)\s_3}\le(\frac{z+a}{z-a}\ri)^{\frac{\s_1}{4}}e^{-\tgg_4(z;\varkappa)\s_3}
\ee
where $\tgg_4(z;\varkappa)$ solves the scalar RHP
\begin{align*}
	\tgg_4(z_+;\varkappa) + \tgg_4(z_-;\varkappa)&=0, & & z\in(-a,0), \\
	\tgg_4(z_+;\varkappa) + \tgg_4(z_-;\varkappa)&=i\pi\cdot\text{sgn}(\Im\varkappa), & & z\in(0,a).
\end{align*}
Then 
\be\label{tg4}
\tgg_4(z;\varkappa)=\frac{R_2(z)}{2\pi i}\int_0^a\frac{ i\pi \cdot\text{sgn}(\Im\varkappa)}{(\z-z)R_{2+}(\z)}d\z,
\ee
where $R_2:\mathbb{C}\setminus[-a,a]\to\mathbb{C}$ is defined as $R_2(z)=\sqrt{z^2-a^2}$, satisfies $R_{2+}(z)=-R_{2-}(z)$ for $z\in(-a,a)$, and $R_2(z)=z+\BigO{1}$ as $z\to\infty$.

\bp\label{prop-Psi_4}
Solutions \eqref{Psi_4} and \eqref{RHP_4_model2} to the model RHP coincide.
\ep

\begin{proof}
Let us begin with $\Im\varkappa\leq0$.  Any solution to the model problem has form \eqref{RHP_4_model1} with matrix $A$ given  by \eqref{A-mod_4}.
 So, to prove the Proposition, it is sufficient to show that residues at $z=\infty$ of solutions \eqref{RHP_4_model2} and \eqref{Psi_4} coincide. It is easy to see that the residue of \eqref{Psi_4} is 
 \be\label{res_Psi_4}
 \frac {ia}{2}(\s_3-i\s_2).
 \ee
 According to \eqref{tg4},
 \be\label{ass_g_mod4}
 \tgg_4(z;\varkappa)=-\frac{i\pi}{4}-\frac{ia}{2z} +\BigO{z^{-2}}
 \ee
 as $z\ra\infty$.
 Substituting \eqref{ass_g_mod4} and  
 \be
 \le(\frac{z+a}{z-a}\ri)^{\frac{\s_1}{4}}=\le(1+\frac{2a}{z-a}\ri)^{\frac{\s_1}{4}}=\1+\frac{a\s_1}{2z}+\BigO{z^{-2}},
 \ee
 into \eqref{RHP_4_model2}, we obtain 
 \be\label{RHP_4_model2_ass}
\Phi_4(z)=i^{-\frac{\s_3}2}\le( \1+\frac{a\s_1}{2z}+\BigO{z^{-2}}  \ri)i^{\frac{\s_3}2} e^{\frac{ia\s_3}{2z}}(1+\BigO{z^{-2}})
=\1+\frac{ia}{2z}(\s_3-i\s_2)+\BigO{z^{-2}},
\ee
 which, together with \eqref{res_Psi_4}, proves the Proposition for $\Im\varkappa<0$.  Now for $\Im\varkappa\geq0$, we use the fact that $\tgg_4(z;\varkappa)$ is odd in $\varkappa$ and that $\left(\frac{z+a}{z-a}\right)^{\frac{\sigma_1}{4}}$ commutes with $\sigma_1$ to obtain 
\begin{align*}
\Phi_4(z;\varkappa)=\sigma_1\Psi_4(z;-\varkappa)\sigma_1=\Psi_4(z;\varkappa),
\end{align*}
as desired.
\end{proof}

\subsection{Approximation of $\Gamma_4(z;\lambda)$ on an annulus centered at the double point $z=0$}

\begin{figure}
\begin{center}
\begin{tikzpicture}[scale=1.5]

\draw[thick] (0.4,0) arc (0:28:0.4);
\node[scale=1.0] at (0.5,0.12) {$\theta$};

\draw[thick] (0.4,0) arc (0:-28:0.4);
\node[scale=1.0] at (0.53,-0.12) {$-\theta$};

\draw[dashed] (0,0) circle [radius=32pt];
\draw[dashed] (0,0) circle [radius=20pt];

\draw[black,thick] (-3,0) -- (3,0);

\draw[black,thick] (-2,1) -- (2,-1);
\draw[black,thick] (-2,-1) -- (2,1);

\filldraw[black] (0,0) circle [radius=2pt];

\node[below,scale=1.5] at (0,-0.0) {$0$};
\node[scale=1.5] at (0,0.92) {$\Omega$};
\node[scale=1.25] at (-1.75,0.4) {$\mathcal{L}_1^{(+)}$};
\node[scale=1.25] at (-1.75,-0.4) {$\mathcal{L}_1^{(-)}$};
\node[scale=1.25] at (1.75,0.4) {$\mathcal{L}_2^{(+)}$};
\node[scale=1.25] at (1.75,-0.4) {$\mathcal{L}_2^{(-)}$};

\end{tikzpicture}
\end{center}
\caption{\hspace{.05in}The set $\Omega$ and lenses.}\label{figLenseOmega}
\end{figure}
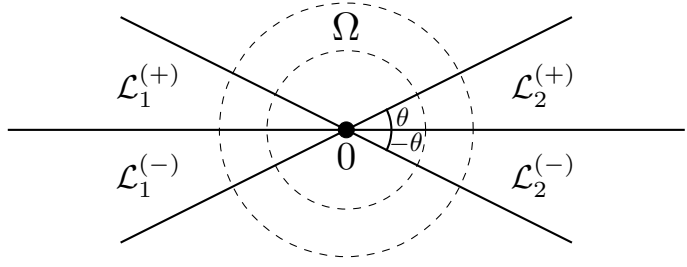

We are now ready to state the needed result from \cite[Theorem 15]{BBKT19} and refer the reader to that text for a more refined description of the set $\Omega$, as shown in Figure \ref{figLenseOmega}.  We recall that the 4 point $\gg$-function is given by
\be\label{g4}
\gg_4(z)=\frac{1}{i\pi}\ln\left(\frac{a+iR_2(z)}{z}\right)-\hf.
\ee

\begin{theorem}\label{GammaAsmptotics}
Let $\theta\in(0,\pi/2)$ be fixed. Then, as $\lambda=e^{-\varkappa}\to0$ with $|\Im\varkappa|\leq\pi$,
\begin{equation}\label{appr_Gam_4}
\Gamma_4(z;\lambda) =
  \begin{cases}
   \Psi_4(z;\varkappa)\left(\1+\BigO{\varkappa^{-1}}\right)e^{-\varkappa\gg_4(z)\sigma_3}, &z\in\Omega\setminus(\mathcal{L}_1^{(\pm)}\cup\mathcal{L}_2^{(\pm)}), \\
   \Psi_4(z;\varkappa)\left(\1+\BigO{\varkappa^{-1}}\right)\begin{bmatrix} 1 & 0 \\ \pm ie^{\varkappa(2\gg_4(z)-1)} & 1 \end{bmatrix}e^{-\varkappa\gg_4(z)\sigma_3}, &z\in\Omega\cap\mathcal{L}_1^{(\pm)}, \\
   \Psi_4(z;\varkappa)\left(\1+\BigO{\varkappa^{-1}}\right)\begin{bmatrix} 1 & \mp ie^{-\varkappa(2\gg_4(z)+1)} \\ 0 & 1 \end{bmatrix}e^{-\varkappa\gg_4(z)\sigma_3}, &z\in\Omega\cap\mathcal{L}_2^{(\pm)},
  \end{cases} \nonumber 
\end{equation}
uniformly for $z\in\Omega$.  See Figure \ref{figLenseOmega} for $\theta, \Omega, \mathcal{L}_j^{(\pm)}$.
\end{theorem}


\section{Positive definiteness of the matrix $\mathbb{M}$}\label{sec: appendix pos def}

The objective of this Appendix is to prove Proposition \ref{prop: M pos def}, which claims that the matrix $\mathbb{M}$, defined in \eqref{M-def}, is positive definite.  We begin with some preparatory observations and Lemmas.  First, recall that $\mathbb{M}$ is the $n\times n$ block diagonal matrix with blocks $\mathbb{M}_1, \mathbb{M}_2$ of size $\tilde{N}\times\tilde{N}, N\times N$, respectively.  Using \eqref{M1 diag}-\eqref{offdiag-ent} and the identity $b_j-b_l=(a_2-a_1)\sin(\nu_l+\nu_j)\sin(\nu_l-\nu_k)$, we can see that
\begin{align}\label{M1M2 fact}
    \mathbb{M}_1=\tilde{K}\tilde{D}+\frac{2\tilde{K}\tilde{C}\tilde{K}}{a_2-a_1}, \hspace{1cm} \mathbb{M}_2=KD+\frac{2KCK}{a_2-a_1},
\end{align}
where $\tilde{K}=\text{diag}[\{k_{2j-1}\}_{j=1}^{\tilde{N}}]$, $K=\text{diag}[\{k_{2j}\}_{j=1}^N]$, $\tilde{C}, C$ are the cosecant matrices 
\begin{align}\label{cosec mat}
    \tilde{C}&=[\csc(\nu_{2j-1}+\nu_{2k-1})]_{j,k=1}^{\tilde{N}}, \hspace{1cm} C=[\csc(\nu_{2j}+\nu_{2k})]_{j,k=1}^{N},
\end{align}
and
\begin{align}
    \tilde{D}&=\text{diag}\bigg[\bigg\{\cos(\alpha-\nu_{2m-1})-\sum_{j=1}^{\tilde{N}}\frac{2k_{2j-1}}{(a_2-a_1)\sin(\nu_{2j-1}+\nu_{2m-1})}\bigg\}_{m=1}^{\tilde{N}}\bigg], \label{tilde D} \\
    D&=\text{diag}\bigg[\bigg\{\sin(\alpha+\nu_{2m})-\sum_{j=1}^{N}\frac{2k_{2j}}{(a_2-a_1)\sin(\nu_{2j}+\nu_{2m})}\bigg\}_{m=1}^{N}\bigg]. \label{D}
\end{align}


\bl\label{lemma: tot pos}
The kernel $\hat{C}(\theta,\phi):= \csc(\theta+\phi)$ is totally positive on $(0,\frac{\pi}{2})$.
\el
\begin{proof}
Let $\xi_j:= {\rm e}^{2i\theta_j}$. 
We want to prove that 
\be
\det \bigg[\hat{C}(\theta_i, \theta_j)\bigg]_{i,j=1}^N>0 \nonumber
\ee
for all $N\in \N$ and $\theta_j\in (0,\frac{\pi}{2})$.
We have 
\bea
\hat{C}(\theta,\phi) = \frac {2i}{{\rm e}^{i(\theta+\phi)}-{\rm e}^{-i(\theta+\phi)}} = {\rm e}^{i\phi-i\theta}\frac {2i}{{\rm e}^{2i\phi}-{\rm e}^{-2i\theta}}, \nonumber
\eea
and thus
\begin{align*}
    \det \bigg[\hat{C}(\theta_i, \theta_j)\bigg]_{i,j=1}^N
=\det \bigg[\xi_i^{\frac 12} \ov\xi_j^{\frac 1 2}\frac {2i}{\xi_i-\ov \xi_j}\bigg]_{i,j=1}^N
=(2i)^N\det \bigg[\frac {1}{\xi_i-\ov \xi_j}\bigg]_{i,j=1}^N.
\end{align*}
Using the Cauchy determinant formula, we obtain
\begin{align*}
    (2i)^N \frac {\prod_{i<j} (\xi_i-\xi_j)(\ov\xi_j-\ov\xi_i)}{\prod_{i,j=1}^N (\xi_i-\ov \xi_j)}=\frac {(2i)^N\prod_{i<j} |\xi_i-\xi_j|^2}{\prod_{i=1}^N (\xi_i-\ov \xi_i) \prod_{i<j} |\ov\xi_i -  \xi_j|^2}=\frac{\prod_{i<j}|\xi_i-\xi_j|^2}
{\prod_{j=1}^N\sin(2\theta_j) \prod_{i<j} |\ov \xi_i-\xi_j|^2}.
\end{align*}
Therefore
\begin{align*}
    \det \bigg[\hat{C}(\theta_i, \theta_j)\bigg]_{i,j=1}^N=\frac{\prod_{i<j}|\xi_i-\xi_j|^2}
{\prod_{j=1}^N\sin(2\theta_j) \prod_{i<j} |\ov \xi_i-\xi_j|^2}>0
\end{align*}
because $\sin(2\theta_j)>0$ for any $\theta_j\in(0,\frac{\pi}{2})$.  Since this statement is true for any diagonal minor, the matrix $\left[\hat{C}(\theta_i, \theta_j)\right]_{i,j=1}^N$
is positive definite.
\end{proof}

\begin{lemma}\label{lemma: diag pos}
The diagonal elements of the matrices $\tilde{D}, D$ are positive for any $n\in\mathbb{N}$.
\end{lemma}

\begin{proof}
Define $\xi_j=e^{2i\nu_{2j}}$, $\eta_j=e^{2i\nu_{2j-1}}$, polynomials
\begin{align*}
    Q(z)=\prod_{j=1}^{\tilde{N}}(z-\eta_j), \hspace{1cm} P(z)=\prod_{j=1}^N(z-\xi_j),
\end{align*}
and rational functions
\begin{align*}
    F(z)=\frac{ie^{-i\alpha}(z-1)^{\frac{1+(-1)^n}{2}}Q(z)}{2\xi_m^\frac{1}{2}z(z-\overline{\xi}_m)P(z)}, \hspace{1cm} \tilde{F}(z)=\frac{e^{i\alpha}(z+1)(z-1)^\frac{1-(-1)^n}{2}P(z)}{2\eta_m^\frac{1}{2}z(z-\overline{\eta}_m)Q(z)}.
\end{align*}
Inspecting \eqref{tilde D}, \eqref{D} and applying the identity $\sin(z)=\frac{1}{2i}(e^{iz}-e^{-iz})$, we find that
\begin{align*}
    \frac{2k_{2j-1}}{(a_2-a_1)\sin(\nu_{2j-1}+\nu_{2m-1})}&=\frac{e^{i\alpha}(\eta_j+1)(\eta_j-1)^{\frac{1-(-1)^n}{2}}P(\eta_j)}{2\eta_j\eta_m^\frac{1}{2}(\eta_j-\overline{\eta}_m)Q'(\eta_j)}=\text{Res}\tilde{F}(z)\big|_{z=\eta_j}, \\
    \frac{2k_{2j}}{(a_2-a_1)\sin(\nu_{2j}+\nu_{2m})}&=\frac{ie^{-i\alpha}(\xi_j-1)^{\frac{1+(-1)^n}{2}}Q(\xi_j)}{2\xi_j\xi_m^\frac{1}{2}(\xi_j-\overline{\xi}_m)P'(\xi_j)}=\text{Res}F(z)\big|_{z=\xi_j}.
\end{align*}
By contour deformation, we have identities
\begin{align*}
    \sum_{j=1}^{\tilde{N}}\frac{2k_{2j-1}}{(a_2-a_1)\sin(\nu_{2j-1}+\nu_{2m-1})}&=\sum_{j=1}^{\tilde{N}}\text{Res}\tilde{F}(z)\big|_{z=\eta_j}=\text{Res}\tilde{F}(z)\big|_{z=\infty}-\text{Res}\tilde{F}(z)\big|_{z=0}-\text{Res}\tilde{F}(z)\big|_{z=\overline{\eta}_m}, \\
    \sum_{j=1}^N\frac{2k_{2j}}{(a_2-a_1)\sin(\nu_{2j}+\nu_{2m})}&=\sum_{j=1}^N\text{Res}F(z)\big|_{z=\xi_j}=\text{Res}F(z)\big|_{z=\infty}-\text{Res}F(z)\big|_{z=0}-\text{Res}F(z)\big|_{z=\overline{\xi}_m}.
\end{align*}
It can now be verified that 
\begin{align*}
    \text{Res}\tilde{F}(z)\big|_{z=\infty}-\text{Res}\tilde{F}(z)\big|_{z=0}&=\cos(\alpha-\nu_{2m-1}), \\
    \text{Res}F(z)\big|_{z=\infty}-\text{Res}F(z)\big|_{z=0}&=\sin(\alpha+\nu_{2m}),
\end{align*}
and
\begin{align*}
    \text{Res}\tilde{F}(z)\big|_{z=\overline{\eta}_m}&=\cos(\nu_{2m-1})(\sin\nu_{2m-1})^\frac{1-(-1)^n}{2}\frac{\prod_{s=1}^N\sin(\nu_{2m-1}+\nu_{2s})}{\prod_{s=1}^{\tilde{N}}\sin(\nu_{2m-1}+\nu_{2s-1})}, \\
    \text{Res}F(z)\big|_{z=\overline{\xi}_m}&=(\sin\nu_{2m})^\frac{1+(-1)^n}{2}\frac{\prod_{s=1}^{\tilde{N}}\sin(\n_{2m}+\n_{2s-1}) }{\prod_{s=1}^N \sin(\n_{2m}+\n_{2s})}.
\end{align*}
Letting $\tilde{D}_m, D_m$ denote the $m$-th diagonal entries of the matrices $\tilde{D}, D$, respectively, we have shown
\begin{align}
    \tilde{D}_m&=\cos(\nu_{2m-1})(\sin\nu_{2m-1})^\frac{1-(-1)^n}{2}\frac{\prod_{s=1}^N\sin(\nu_{2m-1}+\nu_{2s})}{\prod_{s=1}^{\tilde{N}}\sin(\nu_{2m-1}+\nu_{2s-1})}>0, \label{tDm} \\ D_m&=(\sin\nu_{2m})^\frac{1+(-1)^n}{2}\frac{\prod_{s=1}^{\tilde{N}}\sin(\n_{2m}+\n_{2s-1}) }{\prod_{s=1}^N \sin(\n_{2m}+\n_{2s})}>0, \label{Dm}
\end{align}
since $\nu_j\in(0,\frac{\pi}{2})$ for $j=1,\ldots,n$, $\cos(x)>0$ for $x\in(0,\frac{\pi}{2})$, and $\sin(x)>0$ for $x\in(0,\pi)$, completing the proof.

\end{proof}

\begin{proposition}\label{prop: M pos def}
The matrix $\mathbb{M}$, defined in \eqref{M-def}, is positive definite for any $n\in\mathbb{N}$.
\end{proposition}

\begin{proof}
Lemmas \ref{lemma: tot pos}, \ref{lemma: diag pos} have shown that $\tilde{C}, C, \tilde{D}, D$ are positive definite. It is clear from \eqref{k_j} that $k_j>0$ for $j=1,\ldots,n$ and thus $\tilde{K}\tilde{D}, KD, \frac{2\tilde{K}\tilde{C}\tilde{K}}{a_2-a_1}, \frac{2KCK}{a_2-a_1}$ are also positive definite.  So by \eqref{M1M2 fact}, we have that $\mathbb{M}_1, \mathbb{M}_2$ are positive definite because the sum of positive definite matrices is positive definite.  Finally, we conclude that $\mathbb{M}$ is positive definite because it is a block diagonal matrix with positive definite blocks.
\end{proof}

\end{document}